\documentclass[a4paper,12pt,reqno]{amsart}
\usepackage{amscd,amssymb,graphicx,color,epigraph,amsmath}
\usepackage{tree-dvips}
\usepackage{qtree}
\usepackage[shortalphabetic,initials]{amsrefs}
\usepackage{hyperref}

\usepackage[dvipsnames]{xcolor}
\usepackage[T1]{fontenc}
\usepackage{hyperref}
\hypersetup{
colorlinks = True,
linkcolor=Bittersweet,
citecolor=red,
}
\usepackage{enumitem}
\setlist[itemize]{leftmargin=18pt}
\setlist[enumerate]{leftmargin=18pt}
\usepackage{comment}

\usepackage{eufrak}
\usepackage{mathrsfs}
\usepackage{caption}
\usepackage{xparse}
\usepackage{amscd}
\usepackage[all,cmtip]{xy}
\usepackage[utf8]{inputenc}
\usepackage[T1]{fontenc}
\usepackage{graphicx}
\usepackage{xcolor}
\usepackage{array}
\usepackage{tikz}
\usetikzlibrary{decorations.pathreplacing, calligraphy}
\usepackage{multirow}
\usepackage{makecell}
\usepackage{float} % Para forzar la posición de las imágenes con [H]

\definecolor{customred}{rgb}{0.82,0.01,0.11}

\usepackage[margin=2.5cm]{geometry}

\usepackage{placeins}
\usepackage{afterpage}
\usepackage{amscd}
\usepackage{float}
\usepackage{graphics}
\usepackage{tikz}
\usetikzlibrary{graphs}
\usetikzlibrary{decorations.pathreplacing,angles,quotes}
\usetikzlibrary{shapes.geometric}

\usepackage{tikz-cd}
\usepackage{comment}
\usepackage{xspace}
\usepackage{mathtools}
\usepackage{pifont} %
\usepackage{amssymb}
\usepackage{amsthm}
\usepackage{graphicx}
\usepackage{subfig}
\usepackage{enumerate}
\usepackage{hyperref}

\usetikzlibrary{patterns,decorations.pathreplacing}
\usepackage{marginnote}

\usepackage{cancel}

\theoremstyle{plain}

\newtheorem{theorem}{Theorem}[section]
\newtheorem{proposition}[theorem]{Proposition}
\newtheorem{lemma}[theorem]{Lemma}
\newtheorem{corollary}[theorem]{Corollary}
\newtheorem{conjecture}[theorem]{Conjecture}

\theoremstyle{definition}

\newcommand{\appsection}[1]{\let\oldthesection\thesection
\renewcommand{\thesection}{Appendix \oldthesection}
\section{#1}\let\thesection\oldthesection}

\newtheorem{definition}[theorem]{Definition}
\newtheorem{notation}[theorem]{Notation}

\theoremstyle{remark}

\newtheorem{remark}[theorem]{Remark}
\newtheorem{example}[theorem]{Example}

\DeclareMathOperator{\spec}{Spec}

\def\D{{\mathbb{D}}}
\def\R{{\mathbb{R}}}
\def\Z{{\mathbb{Z}}}
\def\F{{\mathbb{F}}}
\def\Q{{\mathbb{Q}}}
\def\C{{\mathbb{C}}}
\def\P{{\mathbb{P}}}

\def\B{{\mathcal{B}}}
\def\M{\mathcal{M}}

\def\O{{\mathcal{O}}}

\def\L{{\mathcal{L}}}

\def\Y{{\mathcal{Y}}}
\def\T{{\mathcal{T}}}
\def\W{{\mathcal{W}}}

\usepackage{xcolor}

\newcommand{\WHS}{\text{WHS}\xspace}
\newcommand{\QHD}{$\mathbb{Q}$\text{HD}\xspace}

\newcommand{\QHDS}{$\mathbb{Q}$HD smoothing\xspace}

\title{Rational homology disk degenerations of elliptic surfaces}

\author{Marcos Canedo}
\address{Facultad de Matem\'aticas,
Pontificia Universidad Cat\'olica de Chile, Santiago, Chile.}
\email{mgcanedo@uc.cl}

\author{Giancarlo Urz\'ua}
\address{Facultad de Matem\'aticas,
Pontificia \allowbreak Universidad \allowbreak{Cat\'olica} de Chile, Santiago, Chile.}
\email{gianurzua@gmail.com}
\begin{document}

\maketitle

\begin{abstract}
In this paper, a \QHD singularity is a weighted homogeneous normal surface singularity admitting a rational homology disk (\QHD) smoothing. These singularities are rational but often not log canonical. We classify all \QHD degenerations of nonsingular projective elliptic surfaces, extending Kawamata's classification of the case with only Wahl singularities (i.e., log terminal \QHD singularities). We also realize all \QHD degenerations of Dolgachev surfaces $D_{a,b}$ with one \QHD singularity, for every pair of integers $a,b$. For each such degeneration, we construct a minimal slc birational model via a Seifert partial resolution in the sense of Wahl followed by semistable flips. Finally, we prove that these minimal slc models are unobstructed and deform to the recent degenerations of Dolgachev surfaces constructed by D. Lee and Y. Lee.
\end{abstract}

\tableofcontents

%%%%%%%%%%%%%%%%%%%%%%%%%%%%%%%%%%%
\section{Introduction} \label{s0}
%%%%%%%%%%%%%%%%%%%%%%%%%%%%%%%%%%%%%%%%%%%%%%%%%%%%%%%%%%%%%%%%%%%%%%%%%%%%%%

%%%%%%%%%%%%%%%%%%%%%%%%%%%%%%%%%%%%%%%%%%%%%%%%%%%%%%%%%%%%%%%%%%%%%%%%%%%%%%

Broadly speaking, a \QHD singularity is a normal complex surface singularity admitting a smoothing whose Milnor fiber is a rational homology disk (i.e., with Milnor number zero). (See Section \ref{s1} for some basic facts, and \href{https://colab.research.google.com/drive/1BgJJCC0qD8eTAmpa3kdrB7CG2GsvNc6d?usp=sharing#scrollTo=YHzJ42TQLYWI}{QHD discrepancies} \cite{programa} for some computations.) They were originally studied by J. Wahl in the 1980s \cite{Wahl80,Wahl_1981}. By \cite{SSW_2008} and \cite{BS_2011}, there is an explicit classification of such singularities when, in addition, they are weighted homogeneous. Their minimal resolution graphs are represented in Figures \ref{fig QHD V3} and \ref{fig QHD V4}. Wahl conjectured \cite{Wahl_2011,Wahl_21} that \QHD singularities are weighted homogeneous (see \cite{PSS14,Be25,BPS26}). In this paper, the term \QHD singularity will always refer to a weighted homogeneous \QHD singularity. These singularities are rational, but typically \textbf{not log canonical}. In fact, the log canonical \QHD singularities are precisely Wahl singularities $\frac{1}{n^2}(1,na-1)$ with gcd$(n,a)=1$, and $3$ particular quotients of elliptic singularities. 

A projective surface with only \QHD singularities will be called a \QHD surface. A \QHD smoothing of a \QHD surface $W$ is a proper deformation $(W \subset \W) \to (0 \in \D)$ over a disk $\D$ such that its fiber over $0$ is $W$, the fibers $W_t$ with $t \neq 0$ are nonsingular projective surfaces, and the deformation induces a \QHD smoothing for each singularity in $W$. Wahl \cite{Wahl_2013} proves that $\W$ is log terminal and $\Q$-Gorenstein, which means that it is locally the quotient of an equivariant deformation of the canonical cover of the singularities in $W$. \QHD smoothings can be seen as the mildest degenerations of surfaces. For example, $K_{W_t}^2$, $\chi_{top}(W_t)$, $q(W_t)=h^1(\O_{W_t})$, $p_g(W_t)=h^2(\O_{W_t})$ are constant for every $t$. Although \QHD degenerations are usually not log canonical, their semi log canonical replacements produce nontrivial degenerations with a reducible central fiber and orbifold normal crossing singularities. It is known that minimal models of arbitrary degenerations of surfaces have Wahl and orbifold normal crossing singularities \cite{Kawa88} (see \cite[Thm. 3.4.2]{Kol91}). \QHD degenerations reveal a significant part of that picture \cite{programaReyes,CU26b}. In this paper, we begin a systematic study of \QHD surfaces and their \QHD smoothings by extending the following result of Kawamata \cite{Kawa92} on degenerations of elliptic surfaces.

\begin{theorem} \label{introthm1}
Let $W$ be a \QHD surface with only Wahl singularities and $K_W$ nef. Assume that $(W\subset\mathcal{W})\to(0\in\mathbb{D})$ is a \QHD smoothing with Kodaira dimension $\kappa(W_t)=1$ for $t\neq 0$. Then we have an elliptic fibration $W \to B$ over some curve $B$, together with a classification of the fibers containing singularities of $W$.
\end{theorem}

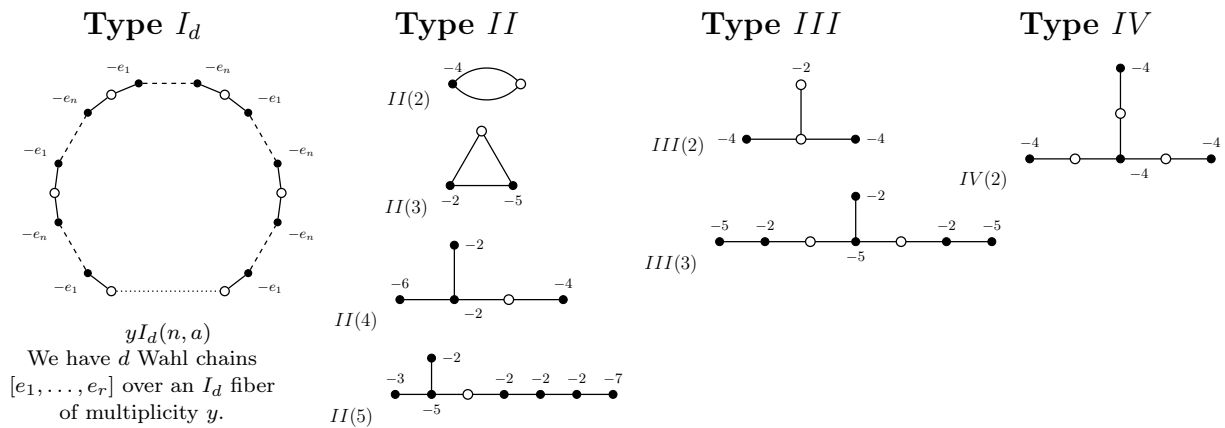
\begin{figure}[ht]
    \centering
    % --- DEFINICIÓN GLOBAL DE ESTILOS (Para no repetir código) ---
    \tikzset{
        bvertex/.style={circle, fill=black, draw=black, minimum size=5pt, inner sep=0pt, outer sep=0pt},
        wvertex/.style={circle, draw=black, fill=white, thick, minimum size=6pt, inner sep=0pt, outer sep=0pt},
        edge/.style={thick, draw=black},
        % Ajuste para etiquetas matemáticas pequeñas
        every label/.append style={font=\scriptsize}
    }
   
    % --- TYPE I ---
    \begin{minipage}[t]{0.22\textwidth}
        \centering
        \textbf{Type $I_d$}\par\vspace{5pt}
        \scalebox{0.6}{
    \begin{tikzpicture}
        % Definición de estilos si no están en tu preámbulo
        \tikzset{
            wvertex/.style={draw, circle, inner sep=2pt, fill=white, thick},
            bvertex/.style={draw, circle, inner sep=1.5pt, fill=black, thick},
            edge/.style={thick}
        }
    
        \def\R{2.5} 
        
        % Vértices blancos
        \node[wvertex] (w1) at (60:\R) {};
        \node[wvertex] (w6_top) at (120:\R) {};
        \node[wvertex] (w5_left) at (180:\R) {};
        \node[wvertex] (w5_bot) at (240:\R) {};
        \node[wvertex] (w6_bot) at (300:\R) {};
        \node[wvertex] (w1_right) at (0:\R) {};

        % Vértices negros con etiquetas
        % Top arc 
        \node[bvertex, label=105:$-e_1$] (bt1) at (105:\R) {};
        \node[bvertex, label=75:$-e_n$] (bt2) at (75:\R) {};
        
        % Top-left arc 
        \node[bvertex, label=165:$-e_1$] (btl1) at (165:\R) {};
        \node[bvertex, label=135:$-e_n$] (btl2) at (135:\R) {};
        
        % Top-right arc 
        \node[bvertex, label=45:$-e_1$] (btr1) at (45:\R) {};
        \node[bvertex, label=15:$-e_n$] (btr2) at (15:\R) {};
        
        % Bottom-left arc 
        \node[bvertex, label=225:$-e_1$] (bbl1) at (225:\R) {};
        \node[bvertex, label=195:$-e_n$] (bbl2) at (195:\R) {};
        
        % Bottom-right arc 
        \node[bvertex, label=315:$-e_1$] (bbr1) at (315:\R) {};
        \node[bvertex, label=345:$-e_n$] (bbr2) at (345:\R) {};

        % Aristas externas
        \draw[edge] (w6_top) -- (bt1);
        \draw[dashed, thick] (bt1) -- (bt2);
        \draw[edge] (bt2) -- (w1);

        \draw[edge] (w5_left) -- (btl1);
        \draw[dashed, thick] (btl1) -- (btl2);
        \draw[edge] (btl2) -- (w6_top);

        \draw[edge] (w1) -- (btr1);
        \draw[dashed, thick] (btr1) -- (btr2);
        \draw[edge] (btr2) -- (w1_right);

        \draw[edge] (w5_bot) -- (bbl1);
        \draw[dashed, thick] (bbl1) -- (bbl2);
        \draw[edge] (bbl2) -- (w5_left);

        \draw[edge] (w6_bot) -- (bbr1);
        \draw[dashed, thick] (bbr1) -- (bbr2);
        \draw[edge] (bbr2) -- (w1_right);

        % Línea punteada en la parte inferior (entre w5_bot y w6_bot)
        \draw[thick, dotted] (w5_bot) -- (w6_bot);
        
    \end{tikzpicture}
}
        \par\vspace{5pt}
        \tiny{ \ \ \ \ \ \ \ $yI_d(n,a)$ \\ We have $d$ Wahl chains $[e_1,\dots,e_r]$ over an $I_d$ fiber of multiplicity $y$.}
    \end{minipage}
    \hfill
    % --- TYPE II (Múltiples gráficos apilados) ---
    \begin{minipage}[t]{0.22\textwidth}
        \centering
        \textbf{Type $II$}\par\vspace{5pt}
        % 1. Bigon
        \scalebox{0.6}{ $II(2)$
            \begin{tikzpicture}
                \node[bvertex, label=90:$-4$] (v1) at (0,0) {};
                \node[wvertex] (v2) at (1.5,0) {};
                \draw[edge] (v1) to[bend left=45] (v2);
                \draw[edge] (v1) to[bend right=45] (v2);
            \end{tikzpicture}
        }
        \par\vspace{5pt}
        % 2. Triángulo
        \scalebox{0.6}{ $II(3)$
            \begin{tikzpicture}
                \def\R{0.8}
                \node[wvertex] (top) at (90:\R) {};
                \node[bvertex, label=below:$-2$] (left) at (210:\R) {};
                \node[bvertex, label=below:$-5$] (right) at (330:\R) {};
                \draw[edge] (top) -- (left) -- (right) -- (top);
            \end{tikzpicture}
        }
        \par\vspace{5pt}
        % 3. Star Graph 1
        \scalebox{0.6}{ $II(4)$
            \begin{tikzpicture}
                \node[bvertex, label=below right:$-2$] (c) at (0,0) {};
                \node[bvertex, label=above:$-6$] (l1) at (180:1.2) {};
                \node[bvertex, label=right:$-2$] (l2) at (90:1.2) {};
                \node[wvertex] (l3a) at (0:1.2) {};
                \node[bvertex, label=above:$-4$] (l3b) at (0:2.4) {};
                \draw[edge] (l1) -- (c) -- (l2);
                \draw[edge] (c) -- (l3a) -- (l3b);
            \end{tikzpicture}
        }
        \par\vspace{5pt}
        % 4. Star Graph 2
        \scalebox{0.6}{ $II(5)$
            \begin{tikzpicture}
                \def\L{0.8}
                \node[bvertex, label=below:$-5$] (c) at (0,0) {};
                \node[bvertex, label=above:$-3$] (l1) at (180:\L) {};
                \node[bvertex, label=right:$-2$] (l2) at (90:\L) {};
                \node[wvertex] (l3a) at (0:\L) {};
                \node[bvertex, label=above:$-2$] (l3b) at (0:2*\L) {};
                \node[bvertex, label=above:$-2$] (l3c) at (0:3*\L) {};
                \node[bvertex, label=above:$-2$] (l3d) at (0:4*\L) {};
                \node[bvertex, label=above:$-7$] (l3e) at (0:5*\L) {};
                \draw[edge] (l1) -- (c) -- (l2);
                \draw[edge] (c) -- (l3a) -- (l3b) -- (l3c) -- (l3d) -- (l3e);
            \end{tikzpicture}
        }
    \end{minipage}
    \hfill
    % --- TYPE III ---
    \begin{minipage}[t]{0.22\textwidth}
        \centering
        \textbf{Type $III$}\par\vspace{5pt}
        % Graph A
        \scalebox{0.6}{ $III(2)$
            \begin{tikzpicture}
                \def\L{1.2}
                \node[wvertex] (c) at (0,0) {};
                \node[bvertex, label=left:$-4$] (l1) at (180:\L) {};
                \node[wvertex, label=above:$-2$] (l2) at (90:\L) {};
                \node[bvertex, label=right:$-4$] (l3) at (0:\L) {};
                \draw[edge] (l1) -- (c) -- (l2); \draw[edge] (c) -- (l3);
            \end{tikzpicture}
        }
        \par\vspace{10pt}
        % Graph B
        \scalebox{0.6}{ $III(3)$
            \begin{tikzpicture}
                \def\L{1}
                \node[bvertex, label=below:$-5$] (c) at (0,0) {};
                % Pata 1
                \node[wvertex] (l1a) at (180:\L) {};
                \node[bvertex, label=above:$-2$] (l1b) at (180:2*\L) {};
                \node[bvertex, label=above:$-5$] (l1c) at (180:3*\L) {};
                % Pata 2
                \node[bvertex, label=right:$-2$] (l2) at (90:\L) {};
                % Pata 3
                \node[wvertex] (l3a) at (0:\L) {};
                \node[bvertex, label=above:$-2$] (l3b) at (0:2*\L) {};
                \node[bvertex, label=above:$-5$] (l3c) at (0:3*\L) {};
               
                \draw[edge] (l1c) -- (l1b) -- (l1a) -- (c) -- (l2);
                \draw[edge] (c) -- (l3a) -- (l3b) -- (l3c);
            \end{tikzpicture}
        }
    \end{minipage}
    \hfill
    % --- TYPE IV ---
    \begin{minipage}[t]{0.22\textwidth}
        \centering
        \textbf{Type $IV$}\par\vspace{5pt}
        \scalebox{0.6}{ $IV(2)$
            \begin{tikzpicture}
                \def\L{1}
                \node[bvertex, label=below right:$-4$] (c) at (0,0) {};
                % Pata 1
                \node[wvertex] (l1a) at (180:\L) {};
                \node[bvertex, label=above:$-4$] (l1b) at (180:2*\L) {};
                % Pata 2
                \node[wvertex] (l2a) at (90:\L) {};
                \node[bvertex, label=right:$-4$] (l2b) at (90:2*\L) {};
                % Pata 3
                \node[wvertex] (l3a) at (0:\L) {};
                \node[bvertex, label=above:$-4$] (l3b) at (0:2*\L) {};
               
                \draw[edge] (l1b) -- (l1a) -- (c) -- (l2a) -- (l2b);
                \draw[edge] (c) -- (l3a) -- (l3b);
            \end{tikzpicture}
        }
    \end{minipage}
    \caption{\small{M-modifications of elliptic fibers (Section \ref{s3}) with only Wahl chains.}}
\label{fig:KAWA_WAHL}
\end{figure}

Figure \ref{fig:KAWA_WAHL} shows the dual graphs of the minimal resolutions of fibers that contain at least one singularity \cite[Thm. 4.2]{Kawa92}. Wahl chains represent the exceptional divisors of Wahl singularities. For \QHD surfaces, our result is as follows.

\begin{theorem} [Theorem \ref{TheoremClassM-resEllipticFiber}] \label{introthm2}
Let $W$ be a \QHD surface with $K_W$ nef. Assume that $(W\subset\mathcal{W})\to(0\in\mathbb{D})$ is a \QHD smoothing with $\kappa(W_t)=1$. Then we have an elliptic fibration $W \to B$ over some curve $B$, together with a classification of the fibers containing singularities of $W$. The dual graphs of the minimal resolutions of these fibers are shown in Figure \ref{fig:KAWA_WAHL} and in Figure \ref{fig:KAWA_QHD} plus their slidings (Definition \ref{sliding}).
\end{theorem}

\begin{figure}[ht]
    \centering
    \includegraphics[scale=0.55]{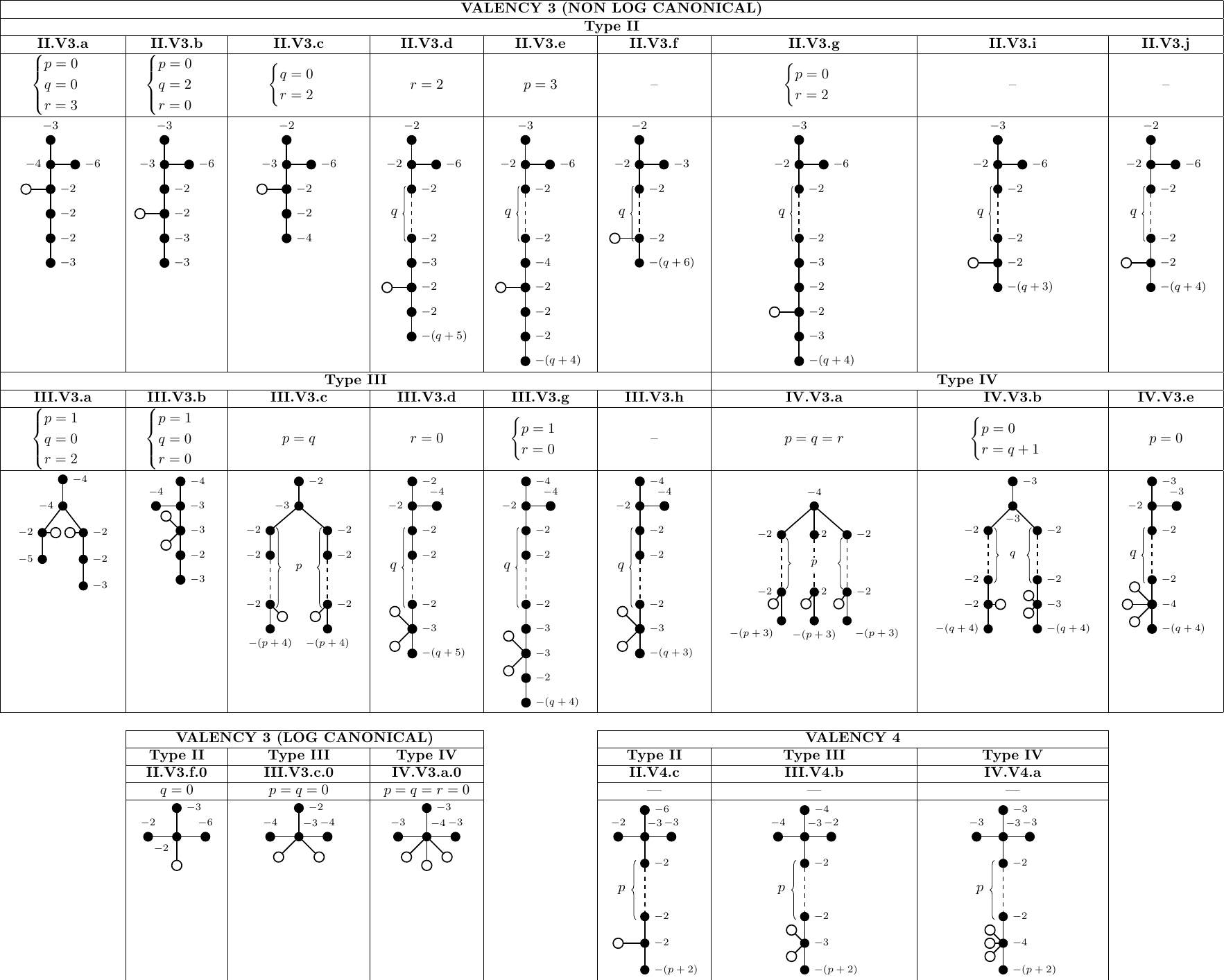}
    \caption{\small{M-modifications of elliptic fibers with a non-Wahl \QHD singularity.}}
    \label{fig:KAWA_QHD}
\end{figure}

The proof of Theorem \ref{introthm2} involves various ingredients. In Section \ref{s2}, we consider an arbitrary \QHD surface $W$ with $K_W$ nef, and compute certain intersection numbers in terms of the minimal model $S$ of the minimal resolution $X$ of $W$. This is similar in spirit to the work in \cite{FRU_23} and \cite{MNU_24}, but the analysis is more involved when the singularities are not log canonical. For instance, the exceptional divisors of $X \to S$ whose intersection with the exceptional divisor of $X \to W$ is $1$ are trivial in the log canonical case, whereas in our setting they form intricate families, which we classify in Theorems \ref{ex S1-T1} and \ref{prop. t1}. In addition, we work with M-modifications of elliptic fibers having only \QHD exceptional divisors, which serve as analogues of M-resolutions \cite{BC94} and M-modifications \cite[Sect.~6]{Kol91}, \cite{Ko24}. This is done in Section \ref{s3}, and relies on specific boundedness properties of \QHD singularities established in Section \ref{s1}. As a consequence, M-modifications involving a non-Wahl singularity can occur only over fibers of types II, III, or IV. Thus, $I_d$ fibers arise exclusively from Wahl singularities, and no other types of singular fibers appear. 

\begin{definition}
Let $a,b \in \Z_{>0}$. A Dolgachev surface $Y$ is a relatively minimal elliptic fibration $Y \to\P^1$ with $q(Y)=p_g(Y)=0$ and exactly two multiple fibers of multiplicities $a$ and $b$. The set of all of them is denoted by $D_{a,b}$.
\label{dolgachev}
\end{definition}

In this way, $D_{1,1}$ consists of rational elliptic fibrations with sections, $D_{1,b}$ consists of rational elliptic Halphen surfaces of index $b$ \cite[4.9]{CDL25}, $D_{2,2}$ consists of Enriques surfaces \cite{CDL25}, and the surfaces in $D_{a,b}$ for all other $a,b$ have Kodaira dimension $1$ \cite{D77}. In Section \ref{s5}, we study \QHD degenerations of Dolgachev surfaces and we prove the next result.

\begin{theorem} [Theorem \ref{DolgachevSmoothings}, Corollary \ref{allDegenerationsDolgachev}] \label{introthm3}
Let $a,b \geq 2$ be integers with $ab>4$. Then for each \QHD singularity from Figures \ref{fig:KAWA_WAHL} and \ref{fig:KAWA_QHD} whose $\lambda$ (in Table \ref{lambda}) has denominator $a$, we have a \QHD degeneration of surfaces in $D_{a,b}$ into a $W$ with only that \QHD singularity. Its associated surface $S$ belongs to $D_{a',b}$ for some $a'$ that divides $a$. Every \QHD degeneration of surfaces in $D_{a,b}$ with only one \QHD singularity appears that way.
\end{theorem}

For example, for $p\geq 0$, Figure \ref{exDolga} shows surfaces in $D_{p+3,b}$ degenerating to a surface $W$ with a \QHD singularity whose configuration is IV.v3.a. The associated $S$ belongs to $D_{1,b}$. The proof of Theorem \ref{introthm3} uses Theorem \ref{introthm2}, but also requires a realization of \QHD degenerations. To this end, in Section \ref{s4} we compute the relevant cohomological obstructions, which are then applied in Section \ref{s5} to a specific construction of $W$. We show that $W$ has no local-to-global obstructions, and hence admits a \QHD smoothing. Moreover, we prove that the general fiber is the desired Dolgachev surface. The argument is based on a canonical class formula for $W$ (Section \ref{s5}), and on the interpretation of the logarithmic transformation as a rational blow-down \cite{FS97} of a Wahl chain.

\begin{figure}[htbp]
\includegraphics[width=14cm]{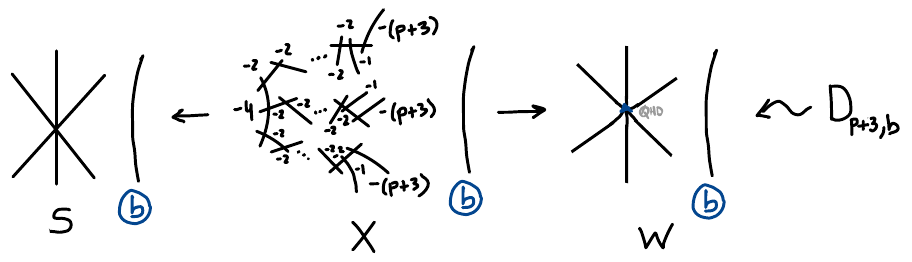}
\caption{\small{A \QHD degeneration $W$ of $D_{p+3,b}$, the minimal resolution $X \to W$ and the minimal model $X \to S$. Each admits an elliptic fibration over $\P^1$.}} \label{exDolga}
\end{figure}

The \QHD degenerations of Dolgachev surfaces in Theorem \ref{introthm3} are usually non-log canonical (e.g. Figure \ref{exDolga} when $p\neq 0$). In Section \ref{s6}, we compute their minimal semi log canonical (slc) models. We begin with the construction in Section \ref{s5}, starting from a rational elliptic surface with sections. We then consider the sliding of one of the non-log canonical \QHD singularities appearing in Figure \ref{fig:KAWA_QHD}. This produces our initial \QHD smoothing $(W \subset \W) \to (0 \in \D)$. Similarly as in \cite{Wahl_2013}, we consider its Seifert partial resolution $\W' \to \D$. We obtain an slc degeneration consisting of two surfaces $W'$ (proper transform of $W$) and $Z$ (compactified Milnor fiber) glued along a $\P^1$ with orbifold normal crossing singularities. It turns out that $\W' \to \D$ is not minimal, i.e., the canonical class $K_{\W'}$ is not nef. In Section \ref{s7}, we prove the existence of a semistable flip $\W^+ \to \D$ \cite{Kawa88,Mori02}, where the singular fiber again has two irreducible components $W^+$, $Z^+$ glued along a $\P^1$ with orbifold normal crossing singularities. In addition, there are two extra Wahl singularities, one on each component; see Figure \ref{flipfigure}.

\begin{figure}[htbp]
\includegraphics[width=11cm]{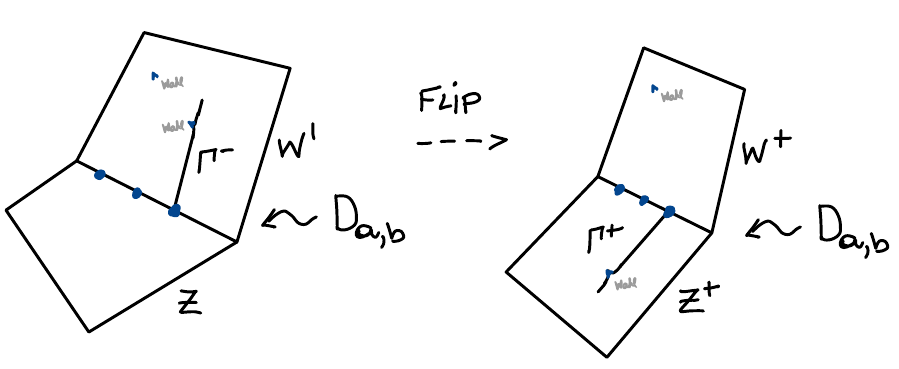}
\caption{\small{The flip of the Seifert partial resolution of $\W \to \D$.}}  
\label{flipfigure}
\end{figure}

\begin{theorem}[Theorem \ref{NefSLCLimit}] \label{introthm4}
The smoothing $\W^+ \to \D$ is a minimal slc model of the \QHD degeneration $\W \to \D$ of Dolgachev surfaces.
\end{theorem}

In the recent preprint \cite{LL25}, D. Lee and Y. Lee show that Dolgachev surfaces $D_{a,b}$ with $\gcd(a,b)=1$ degenerate into surfaces of types (1) $I_0^*/I_0^*$, (2) $II/II^*$, (3) $III/III^*$, and (4) $IV/IV^*$. These surfaces consist of two irreducible components glued along a $\P^1$ with only orbifold normal crossing singularities. One irreducible component is constructed from a surface in $D_{a,1}$ together with a fiber of type $I_0^*,II,III,IV$, while the other is constructed from a surface in $D_{1,b}$ together with a fiber of type $I_0^*,II^*,III^*,IV^*$; see details in \cite[Thm. 5.2, Cor. 5.3]{LL25}. The following result is proved in Section \ref{s7}. 

\begin{theorem}[Corollary \ref{LeeLee}] \label{introthm5}
The singular fiber of $\W^+ \to \D$ has unobstructed $\Q$-Gorenstein deformations. The $\Q$-Gorenstein smoothing of its Wahl singularities produces the Lee–Lee degenerations of types (2), (3), and (4). 
\end{theorem}

For the Lee--Lee degeneration of type (1), see Remark \ref{case (1) LeeLee}.

\subsubsection*{Acknowledgments} We are grateful to Paul Hacking, Yujiro Kawamata, Donggun Lee, Yongnam Lee, Jenia Tevelev, and Jonathan Wahl for helpful discussions. Part of the results were obtained during the authors’ visit to the Center for Complex Geometry at IBS in Daejeon during the first two weeks of November 2025; we thank them for their hospitality. The first-named author was funded by the ANID doctoral scholarship 21220497, and the second-named author by FONDECYT regular grant 1230065.

%\subsubsection*{Notation}

%\begin{itemize}
 %   \item A $(-m)$-curve $\Gamma$ in a nonsingular surface $X$ is a $\Gamma=\P^1$ such that $\Gamma^2=-m$.
    
  % \item A Du Val configuration on a nonsingular surface is the minimal (reduced) exceptional divisor of a Du Val singularity.  
    
   %\item A multisection of a fibration is an irreducible curve whose intersection with a fiber is positive. 
    
   % \item[]   
%\end{itemize}

%%%%%%%%%%%%%%%%%%%%%%%%%%%%%%%%%%%%%%%%%%%%%%%%%%%%%%%%%%%%%%%%%%%%%%%%%%%%%%%%%%%%%%%%%%%%%%%%%%%%%%%%%%%%%%%%%%%%%%%%%%%%%%%%%%%%%%%%%%%%%%%%%%%%%%%%%%
\section{Basics on \QHD singularities} \label{s1}

For us, a \textit{surface singularity} $(p\in W)$ is an analytic germ of a normal complex surface $W$ at a singular point $p$. A \textit{resolution} of $(p\in W)$ is a proper birational morphism $\phi \colon X \to W$ such that $X$ is a nonsingular surface. The pre-image of the singular point $p$, seen as a reduced divisor, is called \textit{exceptional divisor}. A resolution is \textit{minimal} if there are no $(-1)$-curves in its exceptional divisor. Every surface singularity has a minimal resolution. On the other hand, if a nonsingular projective surface $X$ contains disjoint exceptional divisors of rational singularities $p_i$, then there is a birational morphism $\phi \colon X\to W$ that precisely contracts them, and so $W$ has the singularities $p_i$ and it is projective. This is the Artin Contraction Theorem \cite[Thm. 2.3]{Art62}.

Given a resolution $\phi \colon X \to W$ of $(p\in W)$ with exceptional divisor $\sum_{i=1}^{r}{C_i}$, where $C_i$s are irreducible curves, the \textit{discrepancy} of $C_i$ is the rational number $d(C_i)$ that solves (uniquely) the system of linear equations $$K_X \cdot C_j = \sum_{i=1}^r d(C_i) \, C_i \cdot C_j.$$ This makes sense because the $r \times r$ matrix $(C_i \cdot C_j)$ is negative definite \cite[III]{BWHKPCV_2004}. We say that $(p \in W)$ is \textit{log canonical} if $d(C_i) \geq -1$ for all $i$ and every $\phi$;  we say that $(p \in W)$ is \textit{log terminal} if $d(C_i) > -1$ for all $i$ and every $\phi$. If the exceptional divisor $\sum_{i=1}^{r}{C_i}$ of $\phi$ has simple normal crossings, then it is enough to check the $d(C_i)$ of that resolution \cite[Cor. 2.32]{KM_1998}.

When $K_W$ is $\Q$-Cartier, we have the numerical equivalence $$K_X \equiv \phi^{*}(K_W) + \sum_{i}d(C_i) \, C_i.$$ If $(p \in W)$ is a rational singularity, then $K_W$ is $\Q$-Cartier. Its \textit{index} is the minimal positive integer $s$ such that $s K_W$ is Cartier. When $s=1$ we have that $(p \in W)$ is a Du Val singularity. For a rational singularity every exceptional $C_i$ must be a $\P^1$. 

Let $C=\sum_{i=1}^{r}{C_i}$ be the exceptional divisor of a rational singularity, where $C$ is a simple normal crossings divisor. We define its \textit{resolution graph} as follows:
\begin{itemize}
\item The vertices are the curves $C_i$ labeled with their self-intersection number $-e_i=C_i^2$.
\item The number of edges between two distinct curves $C_i$ and $C_j$ is given by their intersection number $C_i\cdot C_j$.
\end{itemize}

\begin{definition}
Let $0<q<m$ be coprime integers. A \textit{cyclic quotient singularity} (c.q.s.) $\frac{1}{m}(1,q)$ is the surface germ at $(0,0)$ of the quotient of $\C^2$ by $(x,y) \mapsto (\zeta x, \zeta^{q} y)$, where $\zeta$ is a $m$-th primitive root of $1$.
\label{def. cyclic quotient singularities} 
\end{definition}

Given coprime integers $0<q<m$, its \textit{Hirzebruch-Jung (HJ) continued fraction} is $$ \frac{m}{q}=  e_1- \frac{1}{e_2 - \frac{1}{\ddots - \frac{1}{e_r}}} =: [e_1, \ldots ,e_r],$$ where $e_i \geq 2$ for all $i$. The minimal resolution graph of the c.q.s. $\frac{1}{m}(1,q)$ is shown in Figure \ref{fig. res graph cqs.}. The index of $\frac{1}{m}(1,q)$ is $\frac{m}{\text{gcd}(m,q+1)}$. 

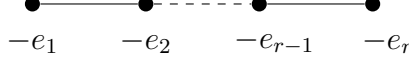
\begin{figure}[htbp]
\centering
\begin{tikzpicture}
    % Dibuja los nodos como puntos (sin estilo de nodo global)
    \node[fill=black, circle, inner sep=2pt] (A) at (0, 0) {};
    \node[fill=black, circle, inner sep=2pt] (B) at (1.5, 0) {};
    \node[fill=black, circle, inner sep=2pt] (C) at (3, 0) {};
    \node[fill=black, circle, inner sep=2pt] (D) at (4.5, 0) {};

    % Coloca los números debajo de cada nodo
    \node at (0, -0.5) {$-e_1$}; % Número debajo del nodo A
    \node at (1.5, -0.5) {$-e_2$}; % Número debajo del nodo B
    \node at (3.2, -0.5) {$-e_{r-1}$}; % Número debajo del nodo C
    \node at (4.7, -0.5) {$-e_r$}; % Número debajo del nodo D

    % Conexiones entre nodos con líneas normales
    \draw (A) -- (B);
    \draw[dashed] (B) -- (C);
    \draw (C) -- (D);
\end{tikzpicture}
\caption{Resolution graph of a c.q.s.}
\label{fig. res graph cqs.}
\end{figure}

\begin{definition}
The HJ continued fraction $\frac{m}{q}=[e_1,\ldots,e_r]$ defines the sequence of integers $ 0=\beta_{r+1} < 1=\beta_r < \ldots < q=\beta_1 < m= \beta_0 $ where $\beta_{i+1}= e_{i}\beta_i - \beta_{i-1}$. In this way, $\frac{\beta_{i-1}}{\beta_{i}}=[e_i,\ldots,e_r]$. The partial HJ continued fractions $\frac{\alpha_i}{\gamma_i} =[e_1,\ldots,e_{i-1}]$ are computed through the sequences $ 0=\alpha_0 < 1=\alpha_1 < \ldots < q^{-1}=\alpha_r < m= \alpha_{r+1},$ where $\alpha_{i+1}=e_i\alpha_{i} - \alpha_{i-1}$, and $\gamma_0=-1$, $\gamma_1=0$, $\gamma_{i+1}=e_i \gamma_i - \gamma_{i-1}$. We have $\alpha_{i+1}\gamma_i - \alpha_i \gamma_{i+1}=-1$, and $\beta_i = q \alpha_i - m \gamma_i$. 
\label{numbers}
\end{definition}

The minimal resolution of the c.q.s. $\frac{1}{m}(1,q)$ satisfies $$K_X
\equiv \phi^*(K_{W}) + \sum_{i=1}^r \big(-1 +\frac{\beta_i+\alpha_i}{m} \big) C_i,$$ and so $d(C_i)=-1 +\frac{\beta_i+\alpha_i}{m} \in ]-1,0]$. Thus, c.q.s. are log terminal.

\begin{definition}
A surface singularity $(p \in W)$ is a \emph{weighted homogeneous singularity} (\WHS) if it is an analytic neighborhood of $0$ of $\spec A$, where $A=\oplus_{i\geq 0} A_i$ is a graded Noetherian normal domain with $A_0=\C$.
\label{whs}
\end{definition}

Quotient singularities are \WHS. By \cite[p.47]{OW_1971_C-Action}, this is equivalent to an action of the multiplicative group $\C^*$ on $(p\in W)$ such that $p$ is the only fixed point. They have particular star-shaped resolutions \cite{OW_1971_C-Action}. 

\begin{definition}
The \textit{star graph} of a \WHS singularity is the graph of the exceptional divisor of the star-shaped resolution. It has a \textit{central curve} $C_0$ of genus $g$, followed by $t$ chains of nonsingular rational curves. See Figure \ref{fig. res graph whs}. For each $i \in \{1,\ldots,t\}$, we denote by $\sum_{j=1}^{r_i}{C_{i,j}}$ the corresponding chain. We call it \textit{leg}. Its HJ continued fraction is $m_i/q_i=[e_{i,1},\dots,e_{i,r_i}]$ where $e_{i,j} \geq 2$. We assume $C_{i,1} \cdot C_0=1$ for every $i$. We will only consider $g=0$, and in that case we will not draw $g$ in the graph.
\label{star graph}
\end{definition}

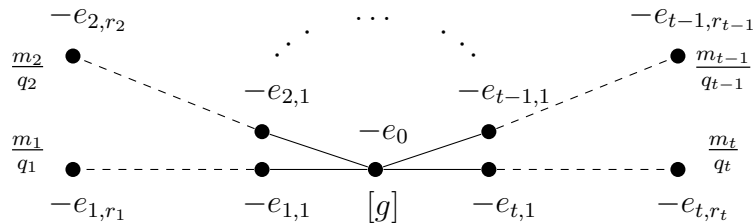
\begin{figure}[h]
\centering
    \begin{tikzpicture}
        %nodo central
        \node[fill=black, circle, inner sep=2pt] (C) at (0, 0) {};

        %nodos patas
        %pata 1
        \node[fill=black, circle, inner sep=2pt] (A11) at (-1.5, 0) {};
        \node[fill=black, circle, inner sep=2pt] (A12) at (-4, 0) {};

        %pata 2
        \node[fill=black, circle, inner sep=2pt] (A21) at (-1.5, 0.5) {};
        \node[fill=black, circle, inner sep=2pt] (A22) at (-4, 1.5) {};

        %pata 3
        \node[fill=black, circle, inner sep=2pt] (A31) at (1.5, 0.5) {};
        \node[fill=black, circle, inner sep=2pt] (A32) at (4, 1.5) {};

        %pata 4
        \node[fill=black, circle, inner sep=2pt] (A41) at (1.5, 0) {};
        \node[fill=black, circle, inner sep=2pt] (A42) at (4, 0) {};

        %LABELS
        \node at (0.1,0.5) {$-e_0$};
        \node at (0.1, -0.5) {$[g]$}; % Número debajo del nodo central
        \node at (-1.3, -0.5) {$-e_{1,1}$}; % Nodo 1 de la pata 1
        \node at (-3.8, -0.5) {$-e_{1,r_1}$};
        \node at (-1.3, 1) {$-e_{2,1}$};
        \node at (-3.8, 2) {$-e_{2,r_2}$};
        \node at (1.7, 1) {$-e_{t-1,1}$};
        \node at (4.2, 2) {$-e_{t-1,r_{t-1}}$};
        \node at (1.7, -0.5) {$-e_{t,1}$}; % Nodo 1 de la pata 1
        \node at (4.2, -0.5) {$-e_{t,r_t}$};
        \node at (-1.3, 1.5) {$\cdot$};
        \node at (-1.15, 1.65) {$\cdot$};
        \node at (-0.9, 1.8) {$\cdot$};
        \node at (1.3, 1.5) {$\cdot$};
        \node at (1.15, 1.65) {$\cdot$};
        \node at (0.9, 1.8) {$\cdot$};
        \node at (0,2) {$\dots$};

        %LABELS FRACCIONES
        \node at (-4.6,0.2) {$\frac{m_1}{q_1}$};
        \node at (-4.6,1.3) {$\frac{m_2}{q_2}$};
        \node at (4.6,0.2) {$\frac{m_{t}}{q_{t}}$};
        \node at (4.6,1.3) {$\frac{m_{t-1}}{q_{t-1}}$};
        
        %CONEXIONES
        \draw (C) -- (A11);
        \draw (C) -- (A21);
        \draw (C) -- (A31);
        \draw (C) -- (A41);

        \draw[dashed] (A11) -- (A12);
        \draw[dashed] (A21) -- (A22);
        \draw[dashed] (A31) -- (A32);
        \draw[dashed] (A41) -- (A42);
    \end{tikzpicture}
    \caption{\small{Resolution graph of a \WHS with $t$ legs.}}
    \label{fig. res graph whs}
\end{figure}

Let $\phi \colon X \to W$ be the star resolution of the \WHS $(p \in W)$. Let $\sigma \colon X \to W'$ be the contraction of the $t$ legs, and let $\phi' \colon W' \to W$ be the contraction of $C'_0=\sigma(C_0)$. The discrepancy $d(C_0)$ of $C_0$ can be easily calculated from $\sigma$ as $d(C_0)= -1 - \frac{\chi}{e}$, where $$e=e_0-\sum_{i=1}^{t} \frac{q_i}{m_i}>0 \ \ \ \text{and} \ \ \ \chi=2g-2+t-\sum_{i=1}^{t} \frac{1}{m_i}.$$ We have ${C'}_0^2=-e$ and ${C'}_0 \cdot K_{W'}=\chi+e$. Through the contraction $\sigma$ and using the fact that we already know the discrepancies for c.q.s., one can prove the following. 

\begin{proposition}
Let $(p\in W)$ be a \WHS, and consider its star resolution. Then $$d(C_0)=-1 - \frac{\chi}{e}, \ \ \ \text{and} \ \ \  d(C_{i,j})=-1 +\frac{\alpha_{i,j}}{m_i}-\frac{\chi}{e} \frac{\beta_{i,j}}{m_i}$$ for all $i,j$. Hence, $(p\in W)$ is: log terminal when $\chi<0$, log canonical when $\chi \leq 0$, and non-log canonical when $\chi>0$.
\label{prop. Sol WHS}    
\end{proposition}

\begin{corollary}
Let $(p\in W)$ be a \WHS that is non-log canonical, that is, $\chi>0$. Then we have $$d(C_{i,j+1})>d(C_{i,j})> d(C_0)$$ for every $i,j$. Moreover, if $\chi/e<1$, then $d(C_{i,r_i})>-1$ for every $i$.
    \label{decrdiscr}
\end{corollary}

\begin{proof}
    This is direct from Proposition \ref{prop. Sol WHS}, and the $\alpha_{i,j}, \beta_{i,j}$ in Definition \ref{numbers}. 
\end{proof}

For numerical explicit computations, see \href{https://colab.research.google.com/drive/1BgJJCC0qD8eTAmpa3kdrB7CG2GsvNc6d?usp=sharing#scrollTo=YHzJ42TQLYWI}{QHD discrepancies} \cite{programa}.

\begin{proposition}
Let $\phi \colon X \to W$ be the star resolution of a \WHS $(p \in W)$, where the exceptional divisor is $C=C_0 + \sum_{i,j} C_{i,j}$. Then we have $K^2_X-K_W^2+K_X \cdot C= (2g(C_0)-2+t) d(C_0)- \sum_{i=1}^t d(C_{i,r_i})$, and $$ K^2_X-K_W^2+K_X \cdot C= 2-2g(C_0) - \sum_{i=1}^t \frac{q_i^{-1}}{m_i} - \frac{\chi^2}{e},$$ where $0<q_i^{-1}<m$ is the integer satisfying $q_i q_i^{-1} \equiv 1 \ (\text{mod} \ m_i)$.
\label{prop. discrepancies whs identity}
\end{proposition}

\begin{proof}
It follows from the definition of discrepancies and the adjunction formula in $X$.
\end{proof}

\begin{definition}
Let $(p\in W)$ a normal surface singularity. A \textit{smoothing} of $(p\in W)$ is a flat surjective morphism $\pi \colon (W \subset \W) \to (0\in\D)$, where $\mathcal{W}$ is an irreducible 3-fold and $(0\in\D)$ is an analytical germ of a smooth curve, such that the fiber over $0$ is isomorphic to $W$, and the generic fiber is non-singular.

\label{smoothings}
\end{definition}

For a smoothing one can choose a local embedding $W\subset\C^N$ and a small $\epsilon$-ball $B_{\epsilon}$ so that $\W$ intersects $S_{\epsilon}=\partial B_{\epsilon}$ transversally, then the smooth 4-manifold $M=B_{\epsilon}\cap\W_t$ is independent of small $\epsilon$ and $t\neq0$, and of the embedding of $\W$ \cite{Wahl_1981}. We call $M$ the \textit{Milnor fiber} of the smoothing. The connected 4-manifold $M$ has the homotopy type of a two-dimensional CW complex, thus we have $H_i(M, \mathbb{Z})=0$ for $i > 2$. Furthermore, by Greuel--Steenbrink~\cite[Thm.~2]{GS_1983}, we have that the rank of $H_1(M, \mathbb{Z})$ is zero. We also have that $H_2(M, \mathbb{Z})$ is a finitely generated free abelian group. The \textit{Milnor number} of the smoothing is the rank of $H_2(M, \mathbb{Z})$. For example, for Du Val singularities of type $A_n$, $D_n$, or $E_n$ we have only one type of smoothing and the Milnor number is $n$.

\begin{definition}
A smoothing of $(p\in W)$ with Milnor number equal to $0$ is called a \emph{rational homology disk smoothing} (\QHDS for short). Singularities that admit a \QHDS are called \QHD singularities. 
\label{def. QHD}
\end{definition}

\begin{theorem}
The \QHD quotient singularities are precisely the c.q.s. of type $\frac{1}{n^2}(1,na-1)$ for $0<a<n$ coprime integers. We call them \textit{Wahl singularities}.
\label{Wahl}
\end{theorem}

\begin{proof}
It follows from \cite{LW_86}, \cite{KSB_1988}. They were first studied by Wahl \cite{Wahl_1981}.
\end{proof}

We also have an algorithm to find all the exceptional divisors of Wahl singularities (see \cite[5.10]{LW_86}). We call them \textit{Wahl chains}. 

\begin{proposition}\label{prop Construction Wahl Sing}
If $[e_1,\dots,e_r]$ is a Wahl chain, then $[2,e_1,\dots, e_{r-1},e_r+1]$ and $[e_1+1,e_2,\dots,e_r,2]$ are Wahl chains. Every Wahl chain can be obtained starting with $[4]$ and iterating the previous operation.
\end{proposition}

\begin{remark}
We have the following facts for an arbitrary \QHD singularity $(p \in W)$.

\begin{itemize}
\item[(1)] It is a rational singularity \cite[Prop.2.4]{SSW_2008}.

\item[(2)] Resolution graphs of all \QHD \WHS were  classified in \cite{BS_2011} after \cite{SSW_2008}. The classification distinguishes three families: Wahl singularities, valency $3$ (Figure \ref{fig QHD V3}), and valency $4$ (Figure \ref{fig QHD V4}). For valency $4$ there is a cross ratio moduli in the central curve; the one having a $\Q$HD has one fixed cross ratio \cite{Fowler13}.

\item[(3)] (Local) $K^2$ of $(0 \in X)$ is equal to (local) $K^2$ of its minimal good resolution plus the number of curves in this resolution \cite[Prop. 2.4]{SSW_2008}.

\item[(4)] The only log canonical \QHD singularities are Wahl singularities and three quotients of simple elliptic singularities (valency $3$).

\item[(5)] By Wahl \cite{Wahl_2013}, \QHD smoothings occur over  $1$-dimensional smoothing components of the deformation space of $(p \in W)$, which we call \QHD components. There can be one or two \QHD components, the precise information is in \cite[Section 7.2]{Wahl_21}, and mainly follows from \cite{Fowler13}. 

\item[(6)]  Let $(W \subset \W) \to (0 \in \D)$ be a \QHD smoothing. By \cite{Wahl_2013}, $\W$ is $\Q$-factorial and log terminal. \QHD smoothings are $\Q$-Gorenstein, i.e., induced by a deformation of an index $1$ cover.

\item[(7)] For any \QHD \WHS we have $\chi/e<1$ \cite[Lem. 2.4]{Wahl_2013}. Thus we have $d(C_{i,j})>d(C_0)>-2$ for all $i,j$ by Corollary \ref{decrdiscr}. 

\end{itemize}
\label{factsQHD}
\end{remark}

\begin{figure}[h!]
    \centering
    \includegraphics[scale=0.53]{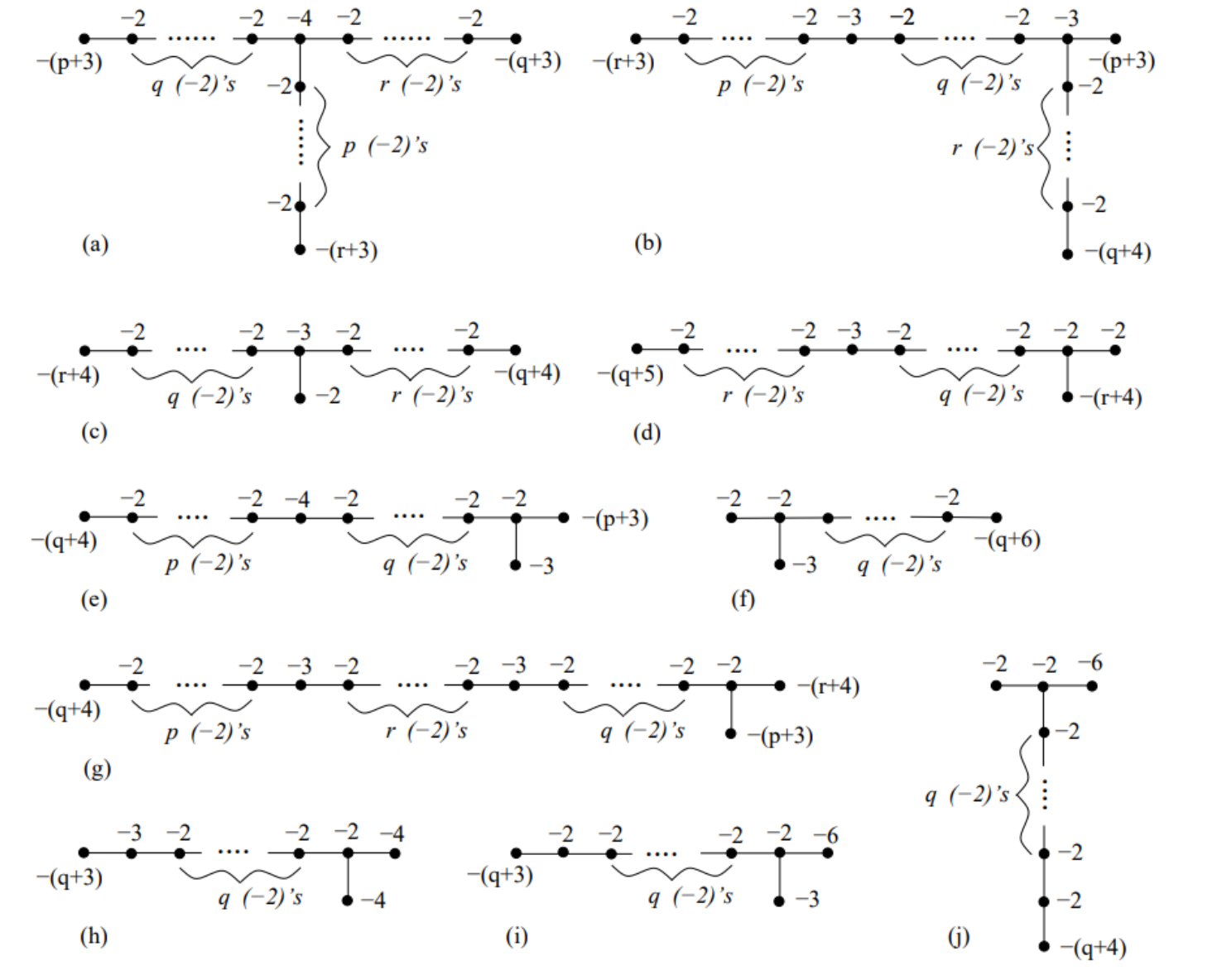}  
    \caption{\small{Star graphs of \QHD singularities of valency 3 from \cite{BS_2011}.}}
    \label{fig QHD V3}
\end{figure}

\begin{figure}[htbp]
    \centering
    \includegraphics[scale=0.35]{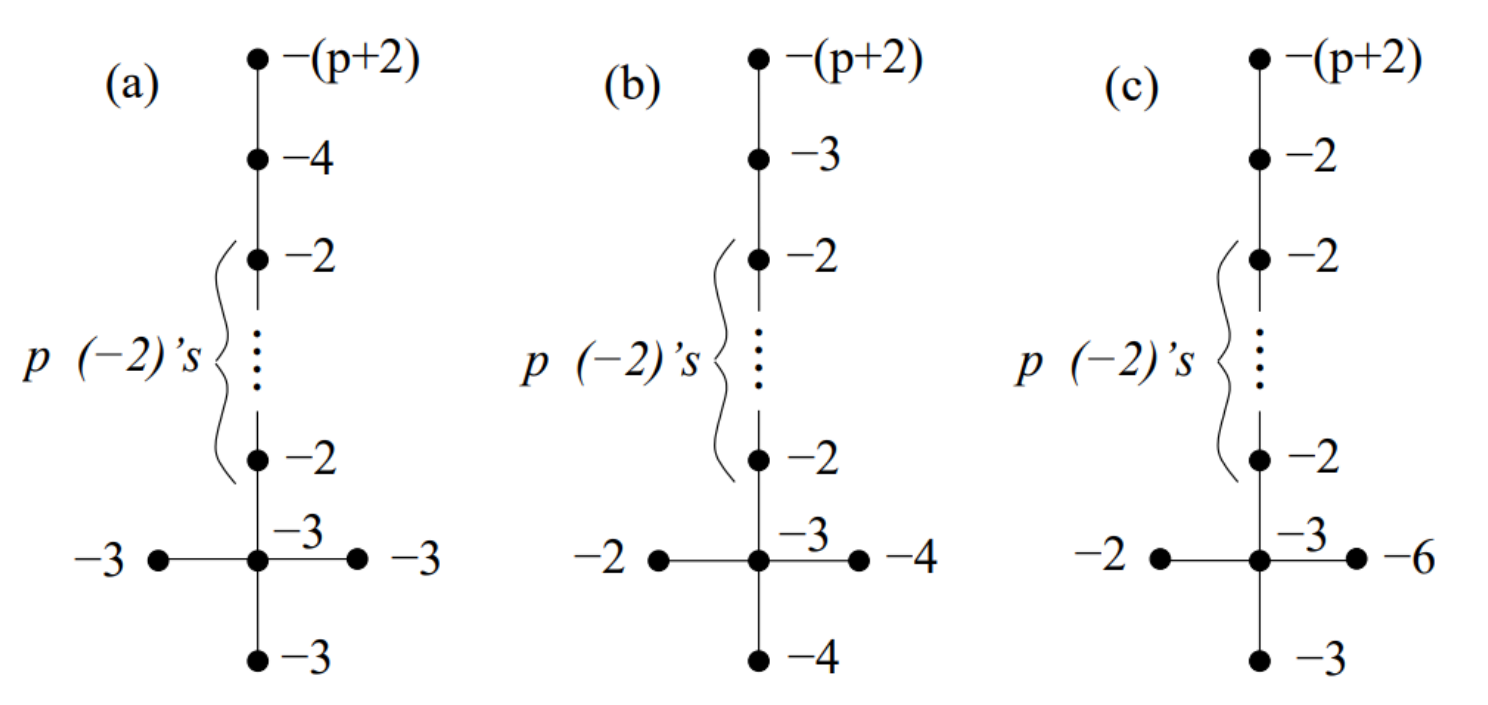}
    \caption{\small{Star graphs of \QHD singularities of valency 4 from \cite{BS_2011}.}}
    \label{fig QHD V4}
\end{figure}

The following conjecture appears in \cite{Wahl_2011}, see also \cite{Wahl_21}. We call it Wahl's conjecture.

\begin{conjecture}
Every \QHD singularity is \WHS.
\label{conj. Wahl Conjecture}
\end{conjecture}

\begin{notation}
In what follows, we will only consider weighted homogeneous \QHD singularities. We will simply call them \QHD singularities.
\end{notation}

\begin{proposition}\label{prop WQHD}
Let $(p\in W)$ a \QHD singularity, and let $C=C_0+\sum_{i=1}^{t}\sum_{j=1}^{r_i}{C_{i,j}}$ be the exceptional divisor associated to the star graph. Let $|C|$ be the number of irreducible curves in $C$. Then $K_X \cdot C = 1+|C|$ for Wahl and valency $3$ singularities, $K_X \cdot C =|C|$ for valency $4$ singularities, and $$K_{X}^{2}-K_{W}^{2}=-|C|.$$ We also have:

\begin{itemize}
    \item For a valency 3 \QHD singularity $\sum_{i=1}^3 d(C_{i,r_i})-d(C_0)=-1$, and so $|d(C_{i,j})-d(C_{i,k})|<1$ for all $i,j,k$.

    \item For a valency 4 \QHD singularity $d(C_{4,r_4})-d(C_0)=1$, where $C_4$ is the long leg. Thus $|d(C_{4,j})-d(C_{4,i})|<1$ for all $i,j$.
\end{itemize}
\end{proposition}

\begin{proof}
    The last part follows directly from Proposition \ref{prop. discrepancies whs identity}.
\end{proof}

%%% NO SE SI ES NECESARIO ESTO ULTIMO

We end this section with some useful facts.

\begin{definition}
Let $(p \in W)$ be a $\Q$-factorial normal surface singularity. An \textit{M-modification} is a proper birational morphism $\sigma \colon Z \to W$ such that $Z$ is normal, and $K_Z$ is relatively nef. It is said to be an \textit{M-resolution} if in addition $Z$ has only Wahl singularities \cite{KSB_1988,BC94,Ko24}.
\label{M-modification}
\end{definition}

\begin{proposition}
Let $(p \in W)$ be a normal singularity, and let $C$ be a $\Q$-Cartier $\Q$-effective divisor on $W$. Let $\sigma \colon Y \to W$ be a proper birational morphism with $Y$ nonsingular. Then $\sigma^*(C)$ is an effective $\Q$-divisor on $Y$.
\label{eff}
\end{proposition}

\begin{proof}
Let $\sigma^*(C)=C' + \sum \alpha_i \Gamma_i$ where Exc$(\sigma)=\sum_i \Gamma_i$ and $C'$ is the proper transform of $C$. We have that the intersection matrix $(\Gamma_i \cdot \Gamma_j)$ is negative definite, and $-(\sum \alpha_i \Gamma_i) \cdot \Gamma_i= C' \cdot \Gamma_i \geq 0$ for all $i$. By \cite[Lem. 3.41]{KM_1998} we have $\alpha_i \geq 0$ for all $i$.   \end{proof}

\begin{proposition}
Let $\sigma \colon Z \to W$ be an M-modification. Then $K_Z \equiv \sigma^*(K_W) - \sum_i \alpha_i \Gamma_i$ with $\alpha_i \geq 0$ for every $i$.  
\label{M-modPos}
\end{proposition}

\begin{proof}
Write $K_Z \equiv \sigma^*(K_W) - \sum_i \alpha_i \Gamma_i$. Let $\phi \colon Y \to Z$ be the minimal resolution of $Z$. Consider $\phi^*(-\sum_i \alpha_i \Gamma_i)=-(\sum_i \beta_i D_i)$ where $\sum_i D_i$ is the exceptional divisor of $\sigma \circ \phi$. Since $K_Z$ is nef, we have $-(\sum_i \beta_i D_i) \cdot D_j \geq 0$ for all $j$. Therefore, by \cite[Lem. 3.41]{KM_1998}, $\beta_i \geq 0$, and in particular $\alpha_i \geq 0$ for every $i$, since among the $D_i$s we have the proper transforms of the $\Gamma_j$s. 
\end{proof}

\begin{proposition}
Let $(p \in W)$ be a log canonical (log terminal) singularity. An M-modification $\sigma \colon Z \to W$ can only have log canonical (log terminal) singularities. In particular, an M-modification with only \QHD singularities of a log terminal singularity is an M-resolution. Moreover, the exceptional divisor of the minimal resolution of an M-modification of a c.q.s. is a chain of nonsingular rational curves.
\label{M-modLogCan}
\end{proposition}

\begin{proof}
We have that $\sigma \colon Z \to W$ satisfies $K_Z \equiv \sigma^*(K_W) -\sum_i \alpha_i \Gamma_i$ where $\sum_i \Gamma_i$ is the exceptional divisor of $\sigma$, and $\alpha_i \geq 0$ for all $i$ by Proposition \ref{M-modPos}. Then by Proposition \ref{eff} we have that $\phi^*(\sum_i \alpha_i \Gamma_i)$ is an effective $\Q$-divisor, where $\phi \colon Y \to Z$ is the minimal resolution. But then there is a proper birational morphism $\pi \colon Y \to W_{min}$ that completes a commutative diagram of morphisms, where $W_{min} \to W$ is the minimal resolution. By definition of log canonical and log terminal singularities we have what we want. For the last part, we use the fact that the \QHD singularities that are log terminal are precisely Wahl singularities. Finally, M-resolutions are described through the P-resolutions of \cite{KSB_1988}. The key to knowing that their minimal resolution is a chain comes from the fact that P-resolutions are dominated by the maximal resolution of the c.q.s. This is \cite[Lem. 3.14]{KSB_1988}.
\end{proof}

%%%%%%%%%%%%%%%%%%%%%%%%%%%%%%%%%%%%%%%%%%%%%%%%%%%%%%%%%%%%%%%%%%%%%%%%%%%%%%%%%%%%%%%%%%%%%%%%%%%%%%%%%%%%%%%%%%%%%%%%%%%%%%%%%%%%%%%%%%%%%%%%%%%%%%%%%%
\section{Intersection numbers from surfaces with only \QHD singularities} \label{s2}

\begin{definition}
We say that $W$ is a \QHD surface if it is a projective surface with only (weighted homogeneous) \QHD singularities.
\label{WQHDsurf}
\end{definition}

The following set-up resembles \cite{FRU_23} and \cite{MNU_24}. Let $W$ be a \QHD surface with $l$ singularities. For $i=3,4$, let $v_i$ be the number of valency $i$ \QHD singularities on $W$. Thus the number of Wahl singularities on $W$ is $l-v_3-v_4$. Consider the diagram
$$
     \xymatrix{  & X  \ar[ld]_{\pi} \ar[rd]^{\phi} &  \\ S &  & W}
$$
where the morphism $\phi$ is the minimal resolution of $W$, and $\pi$ is a composition of $m$ blow-ups such that $S$ has no $(-1)$-curves. Let $E_i$ be the pull-back divisor in $X$ of the $i$-th point blown-up through $\pi$. Thus, $E_i$ is a connected, possibly non-reduced tree of $\P^1$s, $E_i^2=-1$, and $E_i\cdot E_j=0$ for $i\neq j$. Let $E=\sum_{i=1}^{m}{E_i}$ and $C=\sum_{i=1}^l C_i$ be the exceptional (reduced) divisor of $\phi$, where the $C_i$ are the exceptional (reduced) divisors over each \QHD singularity. Let $|C_i|$ be the number of curves in $C_i$. We have $$K_X \sim \pi^{*}(K_S)+E.$$ By intersecting with $C$ and using Proposition \ref{prop WQHD}, we obtain $$K_{W}^{2}-K_{S}^{2}=\pi^{*}K_{S}\cdot C+E\cdot C -l+v_4-m=\sum_{i=1}^{l}|C_i|-m.$$ 

We can also express this formula using essentially the arithmetic genus of $E+C$. 

\begin{definition}
For any divisor $D$ on a nonsingular projective surface $Y$, we define $$q(D):=-\frac{1}{2}(D^2+ D\cdot K_Y).$$ If $D$ is effective, then $q(D)=\chi(\O_D)$. For any $D,D'$ we have $$q(D+D')=q(D)+q(D')-D \cdot D',$$ and $q(nD)=nq(D)-{n \choose 2} D^2$. We also have $q(D)=q(\sigma^*(D))$ for any composition of blow-ups $\sigma$. In fact, if $\sigma \colon Y' \to Y$ is a composition of blow-ups, $E_i$ are the pull-back of the $i$-th point blown-up through $\sigma$, and $\sigma^*(D)=D'+\sum_i k_i E_i$ for some $k_i$s, then $$q(D')=q(D)+\sum_i \binom{k_i}{2},$$ where $\binom{N}{2} := \frac{N(N-1)}{2}$. Note that $E_i \cdot D'=k_i$.
\label{defq}
\end{definition}

\begin{definition}
For any given $E_i$, we define the divisors 
$$C_{out,i}=\{ \Gamma \in C \colon \Gamma \nsubseteq E_i \} \ \ \ \text{and} \ \ \ C_{in,i}=\{ \Gamma \in C \colon \Gamma \subseteq E_i \}.$$ Let $S_h$ be the set of $E_i$s such that $E_i \cdot C_{out,i}=h$. Let $T_h$ be the set of $E_i$s in $S_h$ satisfying $E_i \cdot C_{in,i}=0$. Let $s_h=|S_h|$ and $t_h=|T_h|$ be the corresponding cardinalities.
\label{def ShTh}
\end{definition}

\begin{remark}\label{q(E+C)}
As in Definition \ref{defq}, we note that for the divisors $D=\pi(C)$ and $D'=E+C$, we can write $\pi^*(\pi(C))=C+E+\sum_i k_i E_i$ for some $k_i$s, and so we have
\begin{align*}
q(E+C)&=q(\pi(C))+\sum_{i}{\binom{E_i\cdot (C+E)}{2}}
=q(\pi(C))+\sum_{i}{\binom{E_i\cdot C-1}{2}}\\
&=q(\pi(C))+\sum_{h\geq1}\left[\sum_{E_i\in T_h}{\binom{E_i\cdot C-1}{2}}+\sum_{E_i\in S_h\setminus T_h}{\binom{E_i\cdot C-1}{2}}\right]\\
&=q(\pi(C))+\sum_{h\geq1}{t_h\binom{h-1}{2}}+\sum_{h\geq1}{(s_h-t_h)\binom{h-2}{2}}
\end{align*}
Thus, the value $q(E+C)$ depends only on the numbers $t_h$ for $h\geq 3$, $s_h-t_h$ for $h\neq 2,3$, and $q(\pi(C))$.
\end{remark}

\begin{proposition}
We have $K_W^2-K_S^2= K_S \cdot \pi(C) - q(E+C) +v_4$.
    \label{q}
\end{proposition}

\begin{proof}
We have $q(E+C)=q(E)+q(C)-E \cdot C$, $q(E)=m$ and $q(C)=l$. Thus, we obtain the formula by direct evaluation. 
\end{proof}

We will use the point of view of $q$ later to study possible $\pi(C)$.

\begin{definition}
Let us consider the following partition of $C$: 
\begin{itemize}
    \item $M:= \{ \Gamma \in C \ \text{contracted by} \ \pi \}$,
    \item $J:= \{ \Gamma \in C \ \text{such that} \ \pi(\Gamma) \cdot K_S \neq 0 \}$, and 
    \item $J^c:= \{ \Gamma \in C \ \text{such that} \ \pi(\Gamma) \cdot K_S = 0 \ \text{and} \ \pi(\Gamma) \ \text{a curve in } S \}$. 
\end{itemize}
\label{def MJJ}
\end{definition}

Let $E_i$ be the pull-back of a $(-1)$-curve $\Gamma$ in the composition of blow-ups $\pi$. We must have $E_i \cdot C_{in,i}=-1$ or $0$, and $E_i \cdot C_{in,i}=-1$ if and only if $C$ contains the proper transform of $\Gamma$.
Therefore, the cardinality $|M|$ of $M$ is $\sum_{h\geq0}{(s_h-t_h)}$. We also have $m=\sum_{h\geq 0} s_h$.

\begin{proposition}
We have $K_W^2-K_S^2= |\pi(C)| - \sum_{h\geq0}{t_h} $, where $|\pi(C)|$ is the number of irreducible curves in $\pi(C)$. 
    \label{pi(C)-T}
\end{proposition}

\begin{proof}
We have $K_{W}^{2}-K_{S}^{2}=\sum_{i=1}^{l}|C_i|-m=\sum_{i=1}^{l}|C_i|-|M|-\sum_{h\geq0}{t_h}=|\pi(C)|-\sum_{h\geq0}{t_h}$.
\end{proof}

\begin{proposition}
If $K_W$ is nef, then $s_0=0$. In particular, $\pi(C)$ must contain some curve.
\label{s0=0}
\end{proposition}

\begin{proof}
Let $E_k\in S_0$. Then, for every $\Gamma \in C$ we have $\Gamma \cdot E_k =0$ or $-1$. Then 
$$ \phi^{*}(K_W) \cdot E_k=K_X \cdot E_k - \sum_{\Gamma \in C} d(\Gamma) \Gamma \cdot E_k = -1- \sum_{\Gamma \in C} d(\Gamma) \Gamma \cdot E_k <0,$$ but this contradicts the nef hypothesis on $K_W$.
\end{proof}

\begin{proposition}
We have $K_W^2-K_S^2=\phi^*(K_W) \cdot E-\pi^{*}(K_S) \cdot \sum_{\Gamma\in C}d(\Gamma) \Gamma$.
\label{keyEqual}    
\end{proposition}

\begin{proof}
By definition of $E$ we have $K_X \cdot E=E^2$. Then $$\phi^*(K_W) \cdot E + \sum_{\Gamma\in C}d(\Gamma) \Gamma \cdot E=K_X \cdot E = E^2= K_W^2 -K_S^2 + \Big(\sum_{\Gamma\in C} d(\Gamma) \Gamma \Big)^2.$$ But $\big(\sum_{\Gamma\in C} d(\Gamma) \Gamma \big)^2=(\pi^*(K_S) + E) \cdot (\sum_{\Gamma\in C} d(\Gamma) \Gamma)$.
\end{proof}

\begin{corollary}
If $K_W$ is nef, then $K_W^2 \geq K_S^2$ and $\pi(C)$ cannot be a (disjoint) collection of Du Val configurations. In particular, if $K_S$ is big and nef, then $K_W^2 > K_S^2$.
\label{kw>ks}
\end{corollary}

\begin{proof}
If $\pi(C)$ is a collection of disjoint Du Val configurations (exceptional divisors of minimal resolutions of Du Val singularities), then in $X$ we should have minimal resolutions of M-resolutions of them by Proposition \ref{M-modLogCan}, but Du Val singularities have none. 

Say that $S$ has big and nef canonical class, and so $S$ is of general type. Then the configurations of curves that are intersection trivial with $K_S$ are precisely Du Val configurations. On the other hand, by Proposition \ref{keyEqual},  
$$K_W^2-K_S^2=\phi^*(K_W) \cdot E-\pi^{*}(K_S)\cdot \sum_{\Gamma\in C}d(\Gamma) \Gamma,$$ and $d(\Gamma)<0$ for all $\Gamma$ and $\phi^*K_W\cdot E \geq 0$. By the projection formula $\pi^*(K_S) \cdot \Gamma =K_S \cdot \pi(\Gamma)$. As there must be curves in $\pi(C)$ and it cannot be a disjoint collection of Du Val configurations, then $K_S \cdot \pi(\Gamma)>0$ for some $\Gamma \in C$. Thus, $K_W^2 > K_S^2$.  
\end{proof}

\begin{proposition}
If $W$ is a \QHD surface and $K_W$ is big and nef, then the minimal model $S$ of the minimal resolution $X$ of $W$ is one of the following: it is rational; it is either a K3 or an Enriques surface; it has Kodaira dimension $1$ and $q(S)=0$; or it is of general type, with $q(S)=0$ and $K_S^2<K_W^2$.
\label{kodairatype}
\end{proposition}

\begin{proof}
If $S$ is ruled, then we have a generically $\P^1$ fibration $f \colon X \to C$, where $C$ is some curve. If the exceptional divisor of the minimal resolution $X \to W$ has some transversal curve to the fibration, then $C=\P^1$ and $X$ is rational. Otherwise, they are in the fibers of $f$, but then $K_W$ cannot be nef. If the Kodaira dimension of $S$ is $0$, then we proceed as in \cite[Prop. 2.3]{RU19} since $\pi(C)$ has at least a curve and it is not a Du Val configuration by Proposition \ref{s0=0}. For the claim $q(S)=0$ in Kodaira dimensions $1$ and $2$, we can use the same proof of \cite[Prop. 4.17]{U25}. We have $K_W^2 > K_S^2$ by Corollary \ref{kw>ks}. 
\end{proof}

We now study the set $S_1$ introduced in Definition \ref{def ShTh}. Recall that $S_0=\emptyset$, but for \QHD singularities we may have elements in $S_1$. It will be crucial to know about them.  

\begin{proposition}
Let $E_k \in S_1$. We have one of the following.
\begin{itemize}
    \item[(i)] If $E_k \notin T_1$, then there are $A \in C_{out,k}$ and $B \in C_{in,k}$ such that $-d(A) \geq 1-d(B)$, $A \cdot E_k=1$, and $B \cdot E_k=-1$. For all other curves $\Gamma \in C$ we have $\Gamma \cdot E_k=0$. Thus, the curve $A$ belongs to the star graph of a \QHD singularity (which is not Wahl), and it is not the ending curve of a leg.
    
    \item[(ii)] If $E_k \in T_1$, then there is $A \in C_{out,k}$ with $A \cdot E_k=1$ and $-d(A) \geq 1$. For all other curves $\Gamma \in C$ we have $\Gamma \cdot E_k=0$.  
\end{itemize}

\label{DiscrS1}
\end{proposition}

\begin{proof}
Since $E_k \in S_1$, there is one curve $A \in C_{out,k}$ such that $E_k \cdot A=1$. In case (i) we have one curve $B \in C_{in,k}$ such that $E_k \cdot B=-1$; in case (ii) we have $E_k \cdot B=0$ for every $B \in C_{in,k}$. In both cases, for all other curves $\Gamma \in C$ we have $\Gamma \cdot E_k=0$.  As in the proof of Proposition \ref{s0=0}, we have $\phi^{*}(K_W) \cdot E_k= -1 -d(A)+d(B) \geq 0$ for (i), and $\phi^{*}(K_W) \cdot E_k= -1 -d(A) \geq 0$ for (ii). In case (i), we have $d(A)<-1$, and so it must belong to a star graph of a non Wahl singularity. Since $0<\chi/e<1$, we have that $A$ is not an ending curve of a leg by Corollary \ref{decrdiscr}. 
\end{proof}

\begin{corollary} [Lem.~3.12 \cite{FRU_23}]
If $K_W$ is nef and $W$ has only Wahl singularities, then $s_1=0$.
\end{corollary}

\begin{lemma}\label{lema S1T1}
Assume that $K_S$ is nef, and let $E_k \in S_1 \setminus T_1$. Let $A$ and $B$ be the curves in Proposition \ref{DiscrS1}. Then $A$ must be the central curve of the star graph of a valency $4$ \QHD singularity.
\end{lemma}

\begin{proof}
Let $X \to S'$ be the contraction of all curves in $E_k$ except $B$. Then $B$ is a $(-1)$-curve in $S'$ and intersects $A$ transversally at one point $P$. Let $S'' \to S'$ be the first blow-up in the composition $X \to S'$, with $G$ as the exceptional $(-1)$-curve. Let $E_{k+1}$ be the corresponding divisor in $X$. If this blow-up is not at $P$, then by the projection formula we have $$0\leq \phi^*(K_W) \cdot E_{k+1}= -1+ d(G)-d(B),$$ because $K_W$ is nef. (If $G$ is not part of $C$, then we take $d(G)=0$.) But then $d(B)\leq -1$. By Proposition \ref{DiscrS1} we have $d(A) \leq -1+d(B) \leq -2$, a contradiction to Remark \ref{factsQHD} part (7). Then the first blow-up must be at $P$. Therefore, as $K_S$ is nef, we must have $A^2 \leq -3$ in $S''$. By Proposition \ref{DiscrS1}, the curve $A$ in $X$ is in the star graph of a \QHD singularity, and it is not an ending curve of a leg. By the classification of \QHD singularities, we must have $A^2 \geq -4$ in $X$. Thus, $A^2$ could be $-3$ or $-4$. 

\bigskip 
\noindent 
\textbf{Claim.} The curve $A$ must be the central curve of the star graph of a \QHD singularity. 

\begin{proof}
Assume that $A$ is not the central curve. Then $B$ cannot belong to the same star graph as $A$ since in that case they must be in the same leg, and so $|d(A)-d(B)|<1$ by Proposition \ref{prop WQHD}. But we know that $d(B) - d(A) \geq 1$ by Proposition \ref{DiscrS1}, which is a contradiction. By the same proposition we know that $A$ is not at the end of a leg.

Therefore we have that $A$ and $B$ are not in the same star graph, and so $B$ belongs to a Wahl chain. This is a direct application of Proposition \ref{M-modLogCan}.

Let $A^2=-3$ in $X$. In $S''$ we have the initial chain $B,G,A$ with $B^2=-2$, $G^2=-1$, and $A^2=-3$. If the proper transform of $G$ in $X$ does not belong to Exc$(\phi)=C$, then in $X$ we have an M-resolution over the curve $B$. By Proposition \ref{M-modLogCan} this must be the curve $B$ itself, but then $B$ is not in $C$, a contradiction. As $A^2=-3$, we must have that $G,A$ belong to the same star graph. Note that there must be a unique curve $A'$ next to $A$ in the star graph of $A$. Then after contracting $B$ and $G$ from $S''$, we obtain the $(-1)$-curve $A$ and the curve $A'$, and so we obtain in $X$ an $E_{k-1} \in S_1 \setminus T_1$ where $B':=A$ and $A'$ are the key curves in Proposition \ref{DiscrS1}, but then $d(B')=d(A)<-1$ and $d(A')<-1$, a contradiction. We conclude that $A^2=-3$ is impossible when $A$ is not the central curve.  

Let us say $A^2=-4$. By the above discussion we know that $G$ is on the same star graph as $A$. Then $B$ and $G$ must separate, and so there are blow-ups over $B \cap G$ creating the chain $$ B,D_1,\ldots,D_{i-1},G', D_i,\ldots,D_s,D_{s+1}=G,A$$ where $G'$ is a $(-1)$-curve not on the same star graph as $A$, but $D_i,\ldots,D_s,D_{s+1},A$ are. We can assume $D_i^2 \leq -2$ for all $i$. Since $A^2=-4$ and $A$ is not the central curve, in $X$ the possibilities for the self intersections of $D_i,\ldots,D_s,D_{s+1},A$ are  

\begin{itemize}
\item Valency 3 of type (e): $[\alpha+4,\underbrace{2,\ldots,2}_{\beta},4]$ for some $\alpha, \beta \geq 0$. Then $B,D_1,\ldots,D_{i-1}$ must be of the form $[\beta+3,\underbrace{2,\dots,2}_{\alpha+2}]$ because of the ending curve $D_i$,  we cannot do further blow-ups. (Otherwise, we would increase self-intersections in the leg, or we would contradict the hypothesis $K_W$ nef, since the ending curve of the leg must have discrepancy $>-1$.) The only M-resolution over the chain $[\beta+3,2,\dots,2]$ that could work is the chain itself being a Wahl chain (see \cite[Prop. 6.5]{MNU_24}), and so $\beta =\alpha+3$. The discrepancy of $D_{i-1}$ is equal to $d(D_{i-1})=-1+\frac{\alpha+3}{\alpha+4}$, and $d(D_i)=-1+\frac{3+2(\alpha+3)}{12+6(\alpha+3)+3\alpha+2\alpha(\alpha+3)}$, and so 
$$K_W \cdot G' = -1-d(D_{i-1})-d(D_i)=-\frac{2(a+3)}{(a+4)(30+15a+2a^2)}<0,$$ a contradiction as $K_W$ is nef.
    
\item Valency 4 of type (a) $[p+2,4]$. Then $B,D_1,\ldots,D_{i-1}$ must be of the form $[3,\underbrace{2,\dots,2}_{\alpha+2}]$ and there are no M-resolutions. Then again $G'$ would be negative for $K_W$, a contradiction.
\end{itemize}
\end{proof} 

Thus, $A$ is the central curve of a (no Wahl) star graph. By Proposition \ref{M-modLogCan}, the curve $A$ cannot be in a valency 3 star graph, since in that case we would obtain an M-resolution over a c.q.s. with a (non Wahl) \QHD singularity. 
\end{proof}

\begin{theorem}
Let $E_k\in S_1\setminus T_1$, and $A,B$ as in Proposition \ref{DiscrS1}. Let $SG$ be the star graph containing $A$. Then we have the following possibilities:
\begin{itemize}
    \item[(i)] $B \in SG$, and it is the ending $(-p-2)$-curve of a leg in Figure \ref{fig S1-T1}.
    \item[(ii)] $B \notin SG$, and it is the left ending $(-p-4)$-curve in Figure \ref{fig out S1-T1 2}, Figure \ref{fig out S1-T1 3}, or Figure \ref{fig out S1-T1 4}.
\end{itemize}
\label{ex S1-T1}
\end{theorem}

\begin{proof}
Let us consider (i), i.e., $B \in SG$. Let $C_{4,r_4},C_{4,r_4},\ldots,C_{4,1}$ be the leg that contains $B$ (it must be the long leg of a valency 4 $SG$). One can compute that $d(C_{4,r_4-1})=-1$, and so by Corollary \ref{decrdiscr}, and since $d(B)-d(A) \geq 1$, we have that $B$ is the ending curve of the leg.
We now need to see the possibilities to contract the leg $[p+2,b,\underbrace{2,\dots,2}_p]$ where $B$ is the ending $(-p-2)$-curve, and $b=2,3,4$. The curve $B$ is the last curve contracted in this leg, as the leg is inside $E_k$. We can exclude:
\begin{itemize}
    \item the possibility to contract the first $(-2)$-curve of the configuration, because this implies that the curve $A$ will be contracted first than $B$.
    \item the possibility to do blow-downs over the curve $B$ since $-1<d_B<0$ and this will contradict the nef hypothesis over $K_W$.
\end{itemize}

Then, any possible blow-down must occur on the curve $C_{4,r_4-1}$. The possibilities of this are shown in Figure \ref{fig S1-T1}.

\bigskip
Let us consider (ii), i.e., $B \notin SG$. As in Lemma \ref{lema S1T1} we have a configuration of curves $B,D_1,\ldots,D_{i-1}$ with no $(-1)$-curves. Since $A$ must be contracted after $B$, there are no blow-downs over the leg $[p+2,b,\underbrace{2,\dots,2}_{p}]$ of $SG$. Therefore, the possibilities for the configuration $B,D_1,\ldots,D_{i-1}$ are:
$[p+4,\underbrace{2,\dots,2}_p]$, $[p+3,3,\underbrace{2,\dots,2}_p]$, or $[p+3,2,3,\underbrace{2,\dots,2}_p]$. This precisely gives the situations in Figure \ref{fig out S1-T1 4}, Figure \ref{fig out S1-T1 3}, or Figure \ref{fig out S1-T1 2}.
\end{proof}

\begin{figure}[htbp]
\centering
\includegraphics[width=13cm]{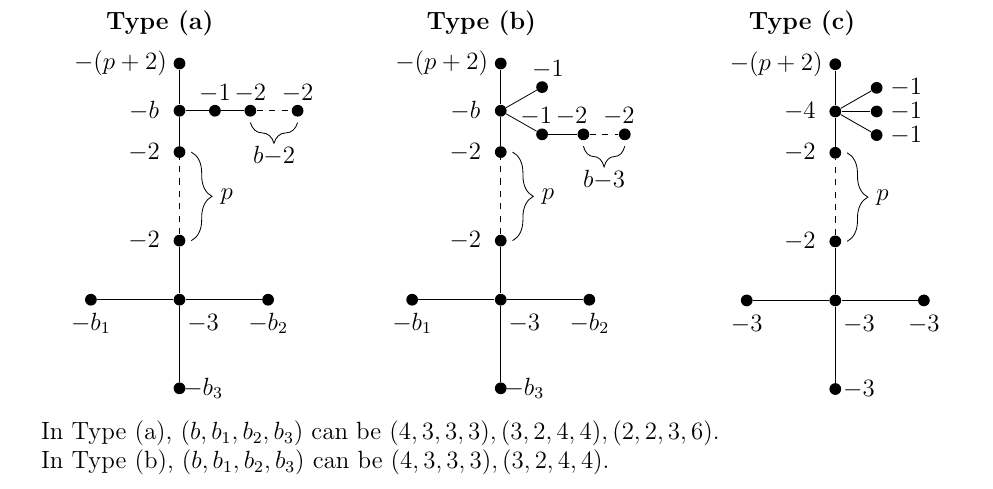}
\caption{\small{The $E_k \in S_1 \setminus T_1$ with $A,B$ in the same star graph.}}
\label{fig S1-T1}
\end{figure}

\begin{figure}[htbp]
\centering
\includegraphics[width=11.5cm]{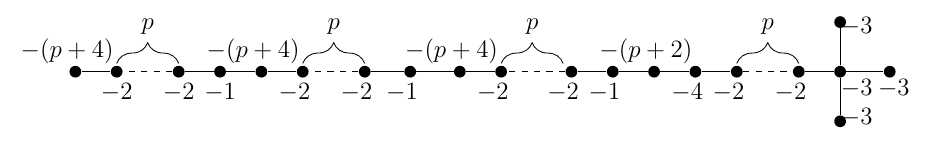}
\caption{\small{The $E_k \in S_1 \setminus T_1$ with $A$ in valency 4 (a) and $B$ out of it.}}
\label{fig out S1-T1 4}
\end{figure}

\begin{figure}[htbp]
\centering
\includegraphics[width=10cm]{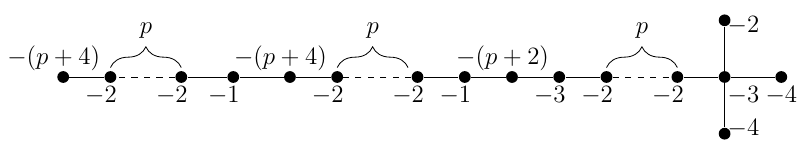}
\caption{\small{The $E_k \in S_1 \setminus T_1$ with $A$ in valency 4 (b) and $B$ out of it.}} 
\label{fig out S1-T1 3}
\end{figure}

\begin{figure}[htbp]
\centering
\includegraphics[width=9cm]{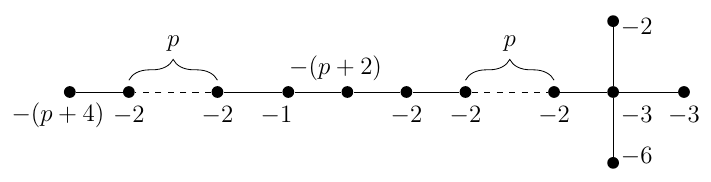}
\caption{\small{The $E_k \in S_1 \setminus T_1$ with $A$ in valency 4 (c) and $B$ out of it.}}
\label{fig out S1-T1 2}
\end{figure}

\begin{theorem}\label{prop. t1}
Assume that $K_S$ is nef. If $E_k \in T_1$, then $E_k$ has the form $[1,\underbrace{2,\dots,2}_s]$ for some $s \geq 0$.
\end{theorem}

\begin{proof}
By Proposition \ref{DiscrS1}, there is a curve $A \in C_{out,k}$ with $d_A\leq -1$. In particular, it is not at the end of any leg. Let $E_k$ the pullback of a $(-1)$-curve $G$, and so $G \notin C$. Therefore, the only blow-ups over $G$ occur over $G \cap A$. (Otherwise, we would contradict the nef hypothesis on $K_W$. Let us do one blow-up, obtaining a $(-1)$-curve $G'$. If $G'$ is in $C$, then its pullback will be in $S_1\setminus T_1$. By our classification in Lemma \ref{ex S1-T1}, we would have that $A$ must be contracted after $G'$ is, but then $G$ would become a negative curve for $K_S$, a contradiction. Thus, $G'$ is not in $C$. We continue inductively, obtaining that $E_k$ has the form $[1,\underbrace{2,\dots,2}_s]$. 
\end{proof}

We end this section by defining an operation that will help transit between generic and nongeneric degenerations. We call it \textit{sliding}, and it was already used in \cite{UZ25}.

\begin{definition}
Let $W$ be a \QHD surface, and let $p \in W$ be a singularity. Let $\phi \colon X \to W$ be the minimal resolution of $W$. Assume there is $(-1)$-curve $\Gamma$ in $X$ intersecting one curve of the star graph over $p$ transversally at one point. Let $X \to Z$ be the contraction of $\Gamma$. We have two situations (See Figure \ref{fig. res graph whs}):
\begin{itemize}
    \item $\Gamma$ intersects the central curve of the star graph, and so we choose a leg $[e_{i,1}, \ldots,e_{i,r_i}]$.
    \item $\Gamma$ intersects a curve in the leg $[e_{i,1}, \ldots,e_{i,r_i}]$ with position $j<r_i$.
\end{itemize}
The \textit{sliding of $\Gamma$} with respect to the leg $[e_{i,1}, \ldots,e_{i,r_i}]$ is the existence of a nonsingular surface $X'$ containing a chain of nonsingular rational curves $$[\text{Wahl chain},1,e_{i,r_i}, \ldots,e_{i,1}]$$ such that the contraction of the $(-1)$-curve and all consecutive ones from this chain is $Z$. We say that the result of the sliding is the surface $W'$, which is defined as the contraction in $X'$ of all star graphs coming from $W$ and the new Wahl chain. The $(-1)$-curve between the Wahl chain and the star graph will be denoted by $\Gamma'$. We call the inverse process of going from $W'$ to $W$ an \textit{unslide}. 
\label{sliding}
\end{definition}

\begin{remark}
In the definition of sliding, we will also consider the following special situation which we call \textit{Aslide}. As in the previous definition, let us now consider two disjoint $(-1)$-curves intersecting transversally at one point each one curve $D$ in a star graph. Then we contract them both, and then we blow-up twice over some point of $D$ which does not belong to any other curve of the star graph. The result is a chain of two $\P^1$s with self-intersections $-1$ and $-2$ in a surface $X'$, and we recover the self-intersection of $D$ in the original star graph. The Asliding is the surface $W'$, which is defined as the contraction in $X'$ of all star graphs coming from $W$ and the new $(-2)$-curve. We will also consider a sequence of Aslidings which produce from a surface with $s$ disjoint $(-1)$-curves intersecting one curve in a star graph, a surface $W'$ with the same \QHD singularities as $W$ and an $A_{s-1}$ Du Val singularity. We call the inverse process of going from $W'$ to $W$ an \textit{unAslide}. 
\label{extra sliding}
\end{remark}

By \cite[Prop. 2.9]{UZ25}, slidings always exist and are unique. In particular, the Wahl chain associated with the leg (and the position of the intersection with $\Gamma$) is unique. Of course, slidings exists over any \WHS. It is easy to verify examples starting with $X$ as in Figure \ref{fig S1-T1}, and $X'$ as in Figures \ref{fig out S1-T1 4}, \ref{fig out S1-T1 3} , \ref{fig out S1-T1 2}.

\begin{lemma}
Let $W$ be a \QHD surface, and let $W'$ be the result of sliding a $(-1)$-curve $\Gamma$ in the minimal resolution $X$ of $W$. Denote by $s_h(W),t_h(W)$ the corresponding cardinalities of the sets $S_h$ and $T_h$ in $X$ as in Definition \ref{def ShTh}, and by $s_h(W'),t_h(W')$ the corresponding in the minimal resolution $X'$ of $W'$. Then
$t_1(W')=t_1(W)-1$, $t_2(W')=t_2(W)+1$, $s_1(W')-t_1(W')=s_1(W)-t_1(W)$, and $s_h(W')=s_h(W)$ for $h>2$.
\label{lemma:sliding property}
\end{lemma}

\begin{proof}
By the definition of sliding, we have that $X'$ is obtained from $X$ by contracting $\Gamma$ and doing blow-ups over a nodal point of a leg $[e_{i,1}, \ldots,e_{i,r_i}]$. This implies that the only change in the sets $S_h$ and $T_h$ is that we lose one element in $T_1$, which is the contracted curve $\Gamma$, and we gain one element in $T_2$, which is the new $(-1)$-curve $\Gamma'$. The other cardinalities remain the same.
\end{proof}

\begin{corollary}\label{slided surface}
Let $W$ be a \QHD surface with $t_1>0$. Then after applying unAslides and slides, we obtain a \QHD surface $W'$ with
$t_1(W')=0$, $t_2(W')=t_2(W)+t_1(W)$, $s_1(W')-t_1(W')=s_1(W)-t_1(W)$, $s_h(W')=s_h(W)$ for $h>2$.
\end{corollary}

\begin{proof}
Notice that an $E_k$ in $T_1$ has always the form $[1,2,\ldots,2]$ by Theorem \ref{prop. t1}. We first unAslide every such $E_k$, and so every $E_k \in T_1$ is a $(-1)$-curve intersecting (transversally at one point) one curve of a star graph, obtaining a surface $W_0$. This curve cannot be an ending leg curve, since in that case we would obtain a negative curve for the canonical class of the corresponding singular surface. 

We can now apply slides on $W_0$ for every $E_k \in T_1$, and so we do it $t_1(W)$ times. Each time we reduce $t_1$ by $1$, and increase $t_2$ by $1$. We then reach a surface $W'$ with $t_1(W')=0$. The relations between the cardinalities follows from Lemma \ref{lemma:sliding property}.
\end{proof}

%%%%%%%%%%%%%%%%%%%%%%%%%%%%%%%%%%%%%%%%%%%%%%%%%%%%%%%%%%%%%%%%%%%%%%%%%%%%%%%%%%%%%%%%%%%%%%%%%%%%%%%%%%%%%%%%%%%%%%%%%%%%%%%%%%%%%%%%%%%%%%%%%%%%%%%%%%
\section{M-modifications of elliptic fibers and \QHD degenerations} \label{s3}

Let $W$ be a \QHD surface. Assume that it admits a minimal elliptic fibration $W \to B$ for some nonsingular projective curve $B$. Hence, up to finitely many points in $B$, the fiber is a nonsingular projective curve of genus $1$, and $K_W$ is nef. We consider the same diagram of morphisms 
     \[\xymatrix{  X  \ar[d]_{\pi} \ar[r]^{\phi} & W \ar[d] \\ S \ar[r] & B}\] 
as in Section \ref{s2}, with the same notation, but now $S \to B$ is a relatively minimal elliptic fibration. The diagram is commutative. The image of the exceptional divisor $C$ of $\phi$ is supported on some fibers of $S \to B$. Let $F$ be one of these fibers.

\begin{definition}
We call 
    \[ \pi \colon (\pi^{-1}(F) \subset X) \to (F \subset S)\] 
an \textit{M-modification} of the elliptic fiber $F$. Note that for every curve $\Gamma$ in $\pi^{-1}(F)$ we have $\phi^*(K_W) \cdot \Gamma \geq 0$, since $K_W$ is nef.
    \label{Mmodification}
\end{definition}

In this way, $\pi \colon X \to S$ is a collection of M-modifications over some fibers of $S \to B$. From now on, we only work over a fixed fiber $F$, and so we consider $E_k$s, $C_i \subset C$, and any other data only over $F$. We also consider $\pi(C)$ as a reduced divisor in that fixed fiber.

\begin{remark}\label{facts QHD elliptic surface}
We have the following facts:
\begin{itemize}
    \item[(a)] The divisor $\phi^*(K_W)$ is effectively supported in fibers of the induced elliptic fibration $X \to B$ since \[ \pi^*(K_S) + E \equiv \phi^*(K_W)+ \sum_{\Gamma \in C} d(\Gamma) \Gamma.\] Therefore $K_W^2 \leq 0$. As $K_W$ nef, we have $K_W^2=0$. By \cite[III, (8.2) Lemma]{BWHKPCV_2004}, we have that ${\phi}^{*}(K_{W})$ is a positive rational multiple of the fiber $\pi^*(F)$ in $X$, and so $\Gamma \cdot K_W=0$ for any $\Gamma \subset \phi(\pi^{-1}(F))$.
    
    \item[(b)] Since $\pi(C)$ is supported on a fiber, we have $S_h=\emptyset$ for $h \geq 4$, and $T_h=\emptyset$ for $h \geq 3$.
    
    \item[(c)] By Propositions \ref{q}, \ref{pi(C)-T} and \ref{keyEqual}, we get $v_4=\chi(\O_{\pi(C)})$, $|\pi(C)|=\sum_h t_h$, $K_S \cdot \pi(C)=\phi^*(K_W) \cdot E=0$, and so \[t_1+t_2=
    \begin{cases}
        1, & \text{if } \pi(C) \text{ is of type }II \\
        2, & \text{if } \pi(C) \text{ is of type }III \\
        3, & \text{if } \pi(C) \text{ is of type }IV
    \end{cases}\]
\end{itemize}
\end{remark}

\begin{proposition}
The sliding of an M-modification of an elliptic fiber is also an M-modification.
\label{M-modfiber}
\end{proposition}

\begin{proof}
Let $W'$ be the result of the sliding of a $(-1)$-curve in the M-modification of an elliptic fiber, and let $W$ be the corresponding contraction with $K_W$ nef. Denote by $\phi \colon X\to W$ and $\phi'\colon X'\to W'$ their respective minimal resolutions, and by $\pi\colon X\to S$ and $\pi'\colon X'\to S$ the corresponding contractions to the minimal model $S$. Let $E, E'$ and $C, C'$ be the corresponding exceptional divisors. Note that we have $\pi(C)=\pi'(C)$ by the definition of sliding. 

By Lemma \ref{lemma:sliding property}, it follows that $t_h(W)=t_h(W')$ for $h\geq 3$ and $s_h(W)-t_h(W)=s_h(W')-t_h(W')$ for $h\neq 2$. Thus, by Remark \ref{q(E+C)}, we obtain $q(E+C) = q(E'+C')$. By the formula for $K_W^2$ in Proposition \ref{q}, we have ${\phi'}^{*}K_{W'}^2 = K_{W}^2$ and so $K_{W'}^2=0$. On the other hand, as in Remark \ref{facts QHD elliptic surface}, we have that ${\phi'}^{*}(K_{W'})$ is a $\Q$-effective divisor with support over a fiber. As before, by \cite[III, (8.2) Lemma]{BWHKPCV_2004}, we have that ${\phi'}^{*}(K_{W'})$ is a positive rational multiple of the fiber $\pi^*(F)$ in $X$. Thus $K_{W'}$ is nef.
\end{proof}

\begin{proposition}
We have that $\pi(C)=F_{red}$, and it must be of type $I_n$ for some $n$, $II$, $III$, or $IV$. 
    \label{Nostar}
\end{proposition}

\begin{proof}
If $\pi(C)$ is contained in a fiber, then it is an ADE configuration of rational curves, but this is not possible by Corollary \ref{kw>ks}. Let us now assume that $\pi(C)$ is a fiber of type $I_n^*$ for some $n$, $II^*$, $III^*$, or $IV^*$. Then by Remark \ref{facts QHD elliptic surface} part (d) we have $v_4=\chi(\O_{\pi(C)})=1$ as $\pi(C)$ is reduced and simply-connected. Note that $S_1=T_1$ in this case, because otherwise $\pi(C)$ would be of type $II$, $III$, or $IV$ by our classification. Therefore, we cannot create all the curves of a leg using only blow-ups, and so $\pi(C)$ must contain the central curve and at least one curve of each leg from the valency $4$ star graph. Then $\pi(C)$ could only be $I_0^*$. As the central curve in $X$ must have self-intersection $-3$, we can do only one blow-up over the central curve of $I_0^*$. But we also have $3$ possible triples of self-intersections of "short" legs: $(-3,-3,-3)$, $(-2,-3,-6)$ or $(-2,-4,-4)$. Thus we need to do blow-ups over more than one "leg" of $I_0^*$, this produces a contradiction. 
\end{proof}

\begin{proposition}
If $\pi(C)=I_n$, then we have only Wahl chains in the M-modification.
    \label{In}
\end{proposition}

\begin{proof}
By Remark \ref{facts QHD elliptic surface} part (d) we have $v_4=\chi(\O_{\pi(C)})=0$. On the other hand, if there is $E_k \in S_1 \setminus T_1$, then by our classification Theorem \ref{ex S1-T1} we must have $v_4 \neq 0$. Thus $S_1=T_1$, and in the M-resolution we have only blow-ups at the intersection points of the curves. But then it is impossible to form a valency $3$ star graph, and so it is only possible to obtain Wahl chains.
\end{proof}

We now recall Kawamata's result, i.e., the case of a $W$ with only Wahl singularities.

\begin{notation}
In the next tables, we will represent the dual graphs of $\pi^{-1}(F)$. Black dots $\bullet$ represent curves in Exc$(\phi)$, and white dots $\circ$ represent curves not in Exc$(\phi)$. White dots with no number attached represent $(-1)$-curves.   
\end{notation}

\begin{theorem} \label{kawamataQHD}
Let $W$ be a \QHD surface with only Wahl singularities and $K_W$ nef. Assume that $(W\subset\mathcal{W})\to(0\in\mathbb{D})$ is a \QHD smoothing with $\kappa(W_t)=0,1$. Then we have an elliptic fibration $W \to B$ over some curve $B$, together with a classification of the fibers containing singularities of $W$ (see Figure \ref{fig:KAWA_WAHL}). If $\kappa(W_t)=0$, we have that $W_t$ is an Enriques surface and $W$ is rational with only $\frac{1}{4}(1,1)$ singularities (at most $10$ of them).           
\end{theorem}

\begin{proof}
For $\kappa(W_t)=0$, this is \cite[Thm. 4.1]{Kawa92} and \cite[Section 9.1]{DK25}, as the minimal resolution of $W$ must be a Coble surface of K3 type. The rest is \cite[Thm. 4.2]{Kawa92}.  
\end{proof}

\begin{lemma}\label{Mres II-III-IV}
 If $\pi(C)$ is not of type $I_n$, then its corresponding M-resolutions with a non-Wahl \QHD singularity are given by the dual graphs in Figure \ref{fig:KAWA_QHD} and their slidings (Definition \ref{sliding} and Remark \ref{extra sliding}).   
\end{lemma}

%\begin{figure}[ht]
%    \centering
%    \includegraphics[scale=0.54]{KAWA_TABLE.pdf}
%    \caption{The M-modifications of fibers with a non-Wahl \QHD singularity.}
%   \label{fig:KAWA_QHD}
%\end{figure}

\begin{proof}
First, it is a tedious routine computation to verify that the configurations in Table \ref{fig:KAWA_QHD} are indeed M-modifications. To do that one needs to compute the discrepancy of the curve intersecting a $(-1)$-curve using the \ref{app}, and also check that the configuration contracts to the type of the elliptic fiber $F$.

Let \[(\pi^{-1}(F) \subset X) \to (F \subset S)\] be an M-modification over $\pi(C)=F$, which must be of type $II$, $III$ or $IV$ by Proposition \ref{Nostar}. For each of these fiber types, we have a minimal SNC resolution $(F' \subset S') \to (F \subset S)$ so that $F'$ consists of one $(-1)$-curve $\Gamma_0$ and 3 other disjoint $\P^1$s, each intersecting $\Gamma_0$ transversally at one point (see Figure \ref{s.n.c config}).

Note that $(\pi^{-1}(F) \subset X) \to (F \subset S)$ must factor through $F' \to F$. Otherwise, we would have the first two M-modifications with only Wahl chains in Theorem \ref{kawamataQHD} (type II).

By Remark \ref{facts QHD elliptic surface} part (b), we have $T_3=\emptyset$. Thus the pull-back of $\Gamma_0$ under $\pi^{-1}(F) \to F'$ belongs to $S_3 \setminus T_3$, and $\Gamma_0$ must be the central curve of the exceptional divisor of a valency 3 or valency 4 singularity. We will now analyze all possible M-resolutions over $F'$ taking into account that blow-ups produce $E_k$s only in $S_1$ and $S_2$ (Remark \ref{facts QHD elliptic surface} part (b)).

\begin{figure}[ht]
    \centering
    \scalebox{0.8}{
    \begin{tikzpicture}[thick, scale=0.5]
    % --- Configuración ---
    \coordinate (A) at (-4, 0);
    \coordinate (B) at (0, 0);
    \coordinate (C) at (4, 0);
    \def\vlen{2.5} 

    % --- Dibujo ---
    
    % 1. La línea central horizontal
    % El nodo [right] pone el "-1" al final de la línea a la derecha
    \draw (-6, 0) -- (6, 0) node[right] {$-1$}; 

    % 2. Las tres líneas verticales
    % Los nodos [above] ponen las etiquetas encima de las líneas verticales
    \draw (A) ++(0, -\vlen) -- ++(0, 2*\vlen) node[above] {$-a$};
    \draw (B) ++(0, -\vlen) -- ++(0, 2*\vlen) node[above] {$-b$};
    \draw (C) ++(0, -\vlen) -- ++(0, 2*\vlen) node[above] {$-c$};

    % 3. Marcar los puntos de intersección
    \filldraw (A) circle (2pt) node[below left] {A};
    \filldraw (B) circle (2pt) node[below left] {B};
    \filldraw (C) circle (2pt) node[below left] {C};

\end{tikzpicture}
    }
    \caption{\small{The fiber $F'$ with $(a,b,c) \in \{(2,3,6), (2,4,4), (3,3,3)\}$.}}
    \label{s.n.c config}
\end{figure}
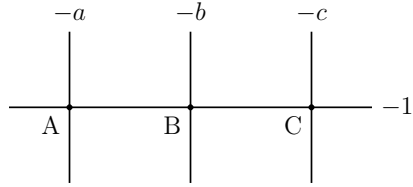

If $S_1\setminus T_1 \neq \emptyset$, then Theorem \ref{ex S1-T1} gives all the M-modifications with valency 4 singularities in Figure \ref{fig:KAWA_QHD} and their slidings. 

Say that $S_1 \setminus T_1 = \emptyset$. By Proposition \ref{M-modfiber}, the sliding of an M-modification is an M-modification. Hence, by Corollary \ref{slided surface}, we can assume that $T_1 = \emptyset$, and we will compute all the possible M-modifications with only blow-ups of type $S_2$ over $F'$.

Assume that $(\pi^{-1}(F) \subset X) \to (F \subset S)$ is an M-resolution of $F$ with $s_1=t_1=0$ and some non-Wahl \QHD star graph. Let us consider the intersection points $A,B,C$ in $F'$ (Figure \ref{s.n.c config}). As $s_h=0$ for $h \neq 2$, we can have blow-ups only over $A,B,C$, which produce $E_k \in S_2$. Every $E_k$ divisor supports some $(-1)$-curve in $X$, which must be in $T_2$. By Remark \ref{facts QHD elliptic surface} part (c), we have at most $t_2=|\pi(C)|$ blow-ups over the points $A, B, C$.

Moreover, we note that the number of Wahl chains in the M-modification will be exactly $t_2=|\pi(C)|$, since $0 = q(\pi(C)) = q(E+C) = l+m-E \cdot C = l-t_2-1$,
where $l$ is the number of Wahl chains in the M-modification plus $1$. The non-central curves in $F'$ that have a point that is blown-up are called \textbf{nonfixed curves}; the rest are the \textbf{fixed curves}. Non-fixed curves are at most $|\pi(C)|$ by the previous observation, and so the fixed curves are at least $3-|\pi(C)|$.

The computation of all M-modifications of $F$ is described in Tables \ref{Tab II}, \ref{Tab III} and \ref{Tab IV}. The steps to compute them are the following.

\begin{itemize}
    \item[1.] Select at least $3-|\pi(C)|$ fixed curves. The first column contains that choice. Note that the self-intersection of (the proper transform of) $\Gamma_0$ in $X$ is at most $-(1+|\pi(C)|)$, and it is the central curve of a non-Wahl star graph of valency 3.
    
    \item[2.] Using the explicit table of valency 3 \QHD singularities, we choose the families of \QHD star graphs which satisfy the numerical restrictions imposed in the first step. This gives the type of the valency 3. 
    
    \item[3.] By doing blow-ups of type $S_2$ only over $F'$, we construct the star graphs of the chosen families in the previous step. The result is the chosen valency 3 star graph and some chain of curves, which represents the minimal resolution of some c.q.s. It is connected to the valency 3 star graph through a $(-1)$-curve at the end of a leg. The c.q.s. found is in a column of the table.

    \item[4.] We determine whether the c.q.s. attached to these legs admit M-resolutions (only Wahl chains and positivity for $K_W$). We will prove below that the admissible M-resolutions for us are very restrictive, and so we will almost never have further constraints. If we do, then we write them in the corresponding column.   

    \item[5.]  To finish the proof, we verify that the potential M-modification in the previous step are slides or Aslides of the M-modifications in Figure \ref{fig:KAWA_QHD} (and so indeed an M-modification by Proposition \ref{M-modfiber}). This is the content of the last column, the correspondence between the obtained M-modifications with the M-modifications in Figure \ref{fig:KAWA_QHD}.
\end{itemize}

As an illustration of how the classification works, consider $F'$ for type $II$:

\vspace{0.2 cm}
\noindent 
{1.} Select the $(-2)$-curve and the $(-3)$-curve as fixed curves. This is denoted by $(2,3)$ in the "fixed curves" column of Table \ref{Tab II}. By definition, we will not perform any blow-ups on these curves. Note that the self-intersection of the exceptional divisor of the central curve must be at most $-2$.

\noindent 
{2.} The unique family of valency 3 star graphs that contains both a $(-2)$-curve and a $(-3)$-curve attached to the central curve is of type (f). 

\noindent 
{3.} We proceed by performing blow-ups of type $S_2$ on $F'$ until the family (f) is constructed. This process is outlined in Figure \ref{fig:Construction_f}.

\begin{figure}[ht!]
    \centering
    \includegraphics[scale=0.5]{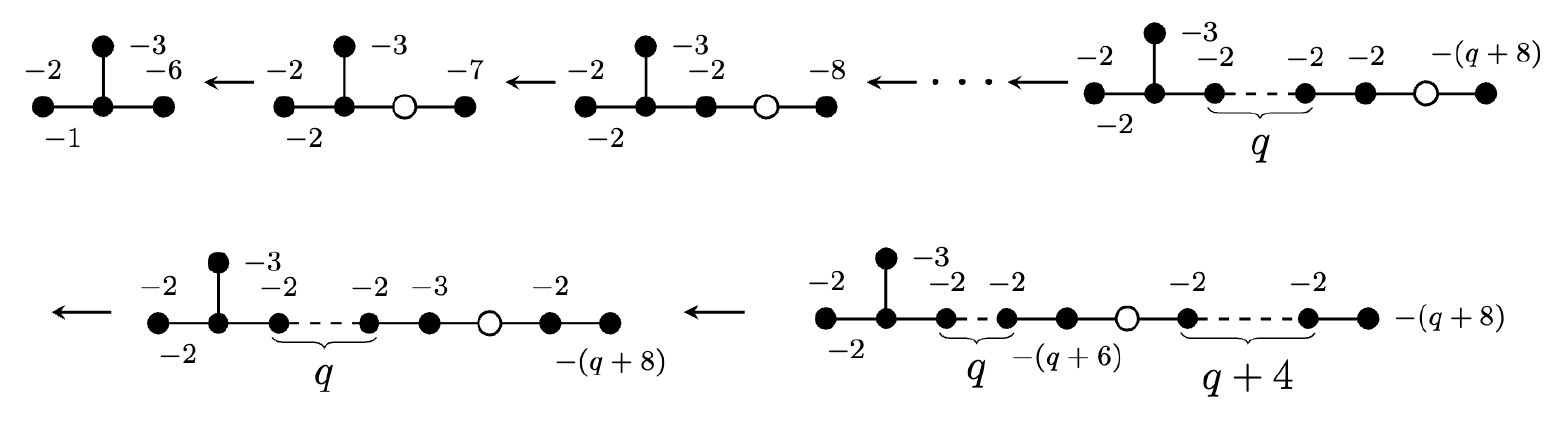}
    \caption{\small{Construction of the family (f) over a fiber of type $II$.}}
    \label{fig:Construction_f}
\end{figure}

\noindent 
{4.} We obtain the chain of rational curves $[\underbrace{2,\dotsc,2}_{q+4},q+8]$, which is, in fact, a Wahl chain. In this case, as we will see below, the chain coincides with the only possible M-resolution. Thus, there are no additional constraints.

\noindent 
{5.} We verify that the configuration obtained is the sliding of the $(-1)$-curve in the M-modification II.v3.f presented in Table \ref{fig:KAWA_QHD}. 

We now consider each fiber type II, III, and IV separately, starting with an explanation about the potential M-resolutions of the c.q.s. that shows up in step 3. They are very limited. A key is the following observation. For any M-modification of $F$ we must have $K_W^2=0$, and so the number of blow ups from $F \subset S$ to $X$ is equal to the number of curves contracted. Then the number of components for $\phi(\pi^{-1}(F))$ and $F$ is equal.

\vspace{0.2 cm}
\noindent 
\textbf{Fiber of type II.} As the number of components of $\phi(\pi^{-1}(F))$ must be equal to $1$ and we already have a $(-1)$-curve in $\pi^{-1}(F)$, the number of Wahl chains for an M-resolution of the c.q.s. in step 3 is $1$. This will be satisfied automatically in our analysis, except for the case $(3,6)$ with family (g), where the parameter $q$ must be zero.
\begin{table}[ht!]
\centering
\small
\renewcommand{\arraystretch}{1.6} % Espacio vertical para que los corchetes respiren
\scalebox{0.7}{
\begin{tabular}{|c|c|c|c|c|c|c|}
\hline
\textbf{Fiber} & \textbf{\makecell{Fixed\\Curves}} & \textbf{v3} & \textbf{\makecell{Parameters\\of the family}} & \textbf{\makecell{Obtained \\C.Q.S.}} & \textbf{\makecell{Constraints\\for the C.Q.S.}} & \textbf{\makecell{Unslided\\ M-modification}}\\
\hline

% --- FIBER II ---
% Total de filas: 1 (caso 2,3) + 4 (caso 2,6) + 6 (caso 3,6) = 11 filas
\multirow{11}{*}{$II$} 
 % CASO (2,3)
 & $(2,3)$ & (f) & $q\geq0$ & $[\underbrace{2,\dotsc ,2}_{q+4} ,q+8]$ & --- & \makecell{II.v3.f for $q>0$\\ II.v3.f.0 for $q=0$}\\ \cline{2-7}
 
 % CASO (2,6) - 4 Variables
 & \multirow{4}{*}{$(2,6)$} 
   & (c) & $r=0,q=2$ & $[2,2,5,4]$ & --- & II.v3.c \\ \cline{3-7}
 & & (d) & $r=2$     & $[\underbrace{2,\dotsc ,2}_{q+3} ,5,q+5]$ & ---  & II.v3.d \\ \cline{3-7}
 & & (f) & $q=0$     & $[2,5]$ & --- & II.v3.f.0\\ \cline{3-7}
 & & (j) & $q\geq0$     & $[\underbrace{2,\dotsc ,2}_{q+2} ,q+6]$ & --- & II.v3.j \\ \cline{2-7}
 
 % CASO (3,6) - 6 Variables
 & \multirow{6}{*}{$(3,6)$} 
   & (a) & $p=0,q=0,r=3$ & $[2,6,2,3]$ & --- & II.v3.a \\ \cline{3-7}
 & & (b) & $p=0,q=2,r=0$ & $[2,3,5,3]$ & --- & II.v3.b\\ \cline{3-7}
 & & (e) & $p=3$         & $[\underbrace{2,\dotsc ,2}_{q+2} ,6,2,q+4]$ & --- & II.v3.e\\ \cline{3-7}
 & & (f) & $q=0$         & $[4]$ & --- & II.v3.f.0 \\ \cline{3-7}
 & & (g) & $p=0,r=2$     & $[2,2,3,5,q+4]$ & $q=0$ & II.v3.g \\ \cline{3-7}
 & & (i) & $q\geq0$          & $[\underbrace{2,\dotsc ,2}_{q+1} ,q+5]$ & --- & II.v3.i\\
\hline
\end{tabular}
}
\caption{\small{Classification for type II.}}
\label{Tab II}
\end{table}

\vspace{0.2 cm}
\noindent 
\textbf{Fiber of type III.} As the number of components of $\phi(\pi^{-1}(F))$ must be $2$ and we already have a $(-1)$-curve in $\pi^{-1}(F)$, the number of Wahl chains for an M-resolution of the c.q.s. in step 3 is $1$ or $2$. Additionally, there are either one or two fixed curves. In the first case, we must have one Wahl chain connected to each of the two legs. This leads to the necessary constraints presented in Table \ref{Tab III}. In the second case, we notice that by Proposition \ref{facts QHD elliptic surface} part (a), the M-resolution of the c.q.s. must be an M-resolution of a T-singularity with one curve and two equal Wahl chains. Particularly, the construction of the family of type (a) has four possibilities, which are captured by two systems with $(\alpha,\beta) \in \{(2,4), (4,2)\}$. The unique case resulting in two Wahl singularities occurs when $\alpha=2$, $\beta=4$, and $r=2$. Also, the family of type (b) has no solutions for $q, r > 0$. Note that if these values are zero, we recover the family of type (b) with two fixed curves $(4,4)$, which was previously classified. 

\begin{table}[ht!]
\centering
\small
\renewcommand{\arraystretch}{1.6} 
\scalebox{0.7}{
\begin{tabular}{|c|c|c|c|c|c|c|}
\hline
\textbf{Fiber} & \textbf{\makecell{Fixed\\Curves}} & \textbf{v3} & \textbf{\makecell{Parameters\\of the family}} & \textbf{\makecell{Obtained\\C.Q.S.}} & \textbf{\makecell{Constraints\\for the C.Q.S.}}&\textbf{\makecell{Unslided\\ M-modification}} \\
\hline

% --- FIBER III ---
% Total filas: 9
\multirow{9}{*}{$III$} 
 
 % (2,4) - 2 filas
 & \multirow{2}{*}{$(2,4)$} 
   & (c) & $r=q=0$ & $[2,2,3,5]$ & --- & III.v3.c.0\\ \cline{3-7}
 & & (d) & $r=0$   & $[\underbrace{2,\dotsc ,2}_{q+3} ,3,q+6]$ & --- & III.v3.d\\ \cline{2-7}
 
 % (4,4) - 4 filas
 & \multirow{4}{*}{$(4,4)$} 
   & (b) & $p=1,q=r=0$ & $[2,4,3,3]$ & --- & III.v3.b\\ \cline{3-7}
 & & (c) & $r=q=0$     & $[3,3]$ & --- & III.v3.c.0\\ \cline{3-7}
 & & (g) & $p=1,r=0$   & $[\underbrace{2,\dotsc ,2}_{q+2} ,4,3,q+4]$ & ---& III.v3.g \\ \cline{3-7}
 & & (h) & $q\geq0$        & $[\underbrace{2,\dotsc ,2}_{q+1} ,3,q+4]$ & --- & III.v3.h\\ \cline{2-7}
 
 % (2) - 1 fila
 & $(2)$ & (c) & \makecell{$r\geq0$ \\ $q\geq0$}
 & $\begin{cases}
    [\underbrace{2,\dotsc ,2}_{r+2} ,q+6]\\
    [\underbrace{2,\dotsc ,2}_{q+2} ,r+6]
 \end{cases}$
 & $r=q$ & \makecell{III.v3.c for $r>0$\\III.v3.c.0 for $r=0$}\\ \cline{2-7}
 
 % (4) - 2 filas
 & \multirow{3}{*}{$(4)$} 
   & \multirow{2}{*}{(a)} & \multirow{2}{*}{\makecell{$p=1$\\$q=0$\\$r>0$\\[0.5cm]$(\alpha,\beta)\in\{(2,4),(4,2)\}$}} 
   & 
    $\begin{cases}
        [\underbrace{2,\dotsc,2}_{r+1},4,\alpha+1]\\
        [2,\beta+r+2]
    \end{cases}$ & \makecell[l]{The first c.q.s. of the system \\does not admit any M-resolutions \\with only Wahl singularities}& ---\\\cline{5-7}
    &&&& $\begin{cases}
        [2,r+3,\alpha+1]\\
        [\underbrace{2,\dotsc,2}_{r+1},\beta+3]
    \end{cases}$ 
   & \makecell{$\alpha=2$\\ $\beta=4$\\ $r=2$}& III.v3.a \\ \cline{3-7}
 & & (b) & \makecell{$p=1$\\$q>0$\\$r>0$} 
   & \makecell[l]{$\begin{cases}[\underbrace{2,\dotsc ,2}_{q+2} ,\alpha +r+2]\\ [\underbrace{2,\dotsc ,2}_{r+1} ,4,\beta +q+2]\end{cases}$} 
   & \makecell[l]{The second c.q.s. of the system \\does not admit any M-resolutions \\with only Wahl singularities}& --- \\
\hline
\end{tabular}
}
\caption{\small{Classification for type III.}}
\label{Tab III}
\end{table}

\vspace{0.2 cm}
\noindent 
\textbf{Fiber of type IV.} As the number of components of $\phi(\pi^{-1}(F))$ must be equal to $3$ and we already have a $(-1)$-curve in $\pi^{-1}(F)$, the number of Wahl chains for an M-resolution of the c.q.s. in step 3 is $1$ or $2$ or $3$. By Proposition \ref{facts QHD elliptic surface} part (a), the M-resolution of the c.q.s. in step 3 must be an M-resolution of a T-singularity with one curve or two curves over a leg, and two or three equal Wahl chains. In addition, we have three possibilities: there are two fixed curves of type $(3,3)$, one fixed curve of type $(3)$, or no fixed curves at all.

\begin{table}[ht!]
\centering
\small
\renewcommand{\arraystretch}{1.6} 
\scalebox{0.7}{
\begin{tabular}{|c|c|c|c|c|c|c|}
\hline
\textbf{Fiber} & \textbf{\makecell{Fixed\\Curves}} & \textbf{v3} & \textbf{\makecell{Parameters\\of the family}} & \textbf{\makecell{Obtained\\C.Q.S.}} & \textbf{\makecell{Constraints\\for the C.Q.S.}}&\textbf{\makecell{Unslided\\ M-modification}}\\
\hline

% --- FIBER IV ---
% Total filas: 5
\multirow{5}{*}{$IV$} 
 
 % (3,3) - 2 filas
 & \multirow{2}{*}{$(3,3)$} 
   & (a) & $p=q=r=0$ & $[2,3,2,4]$ & --- & \makecell{IV.v3.a for $p>0$\\IV.v3.a.0 for $p=0$}\\ \cline{3-7}
 & & (e) & $p=0$     & $[\underbrace{2,\dotsc ,2}_{q+2} ,3,2,q+5]$ & --- & IV.v3.e \\ \cline{2-7}
 
 % (3) - 2 filas
 & \multirow{3}{*}{$(3)$} 
   & \multirow{2}{*}{(a)} & \multirow{2}{*}{$p=q=0$} 
   & $\begin{cases}
    [\underbrace{2,\dotsc,2}_{r+1},3,4]\\
    [2,r+5]
   \end{cases}$
   & $r=0$& \makecell{IV.v3.a for $p>0$\\ IV.v3.a.0 for $p=0$}\\ \cline{5-7}
   &&&& $\begin{cases}
    [2,r+3,4]\\
    [\underbrace{2,\dotsc,2}_{r+1},5]
   \end{cases}$ & $r=0$& \makecell{IV.v3.a for $p>0$\\ IV.v3.a.0 for $p=0$}\\ \cline{3-7}
 & & (b) & $p=0$   & $\begin{cases}
    [\underbrace{2,\dotsc,2}_{r+1},3,q+5]\\
    [\underbrace{2,\dotsc,2}_{q+2},r+5]
   \end{cases}$ & $r=q+1$ & IV.v3.b \\ \cline{2-7}
 
 % Empty set - 1 fila
 & $\emptyset$ & (a) & \makecell{$p\geq0$\\$q\geq0$\\$r\geq0$} 
 & \makecell[l]{$\begin{cases}[\underbrace{2,\dotsc ,2}_{p+1} ,q+5]\\[0.5ex][\underbrace{2,\dotsc ,2}_{q+1} ,r+5]\\[0.5ex][\underbrace{2,\dotsc ,2}_{r+1} ,p+5]\end{cases}$} 
 & $p=q=r$ &\makecell{IV.v3.a for $p>0$\\ IV.v3.a.0 for $p=0$}\\
\hline
\end{tabular}
}
\caption{\small{Classification for type IV.}}
\label{Tab IV}
\end{table}
\end{proof}

\begin{theorem} \label{TheoremClassM-resEllipticFiber}
Let $W$ be a \QHD surface with $K_W$ nef. Assume that $(W\subset\mathcal{W})\to(0\in\mathbb{D})$ is a \QHD smoothing with $\kappa(W_t)=1$. Then there is an elliptic fibration $W \to B$ for some curve $B$.  %If $\kappa(W_t)=0$, we have that $W_t$ is an Enriques surface and $W$ has only Wahl singularities. 
Hence the possible configurations of \QHD singularities in elliptic fibers of $W \to B$ are classified by the dual graphs in Theorem \ref{kawamataQHD} and Lemma \ref{Mres II-III-IV}. 
\end{theorem}

\begin{proof}
By \cite[Main Theorem]{Wahl_2013}, we have that $(W\subset\mathcal{W})\to(0\in\mathbb{D})$ is a $\Q$-Gorenstein smoothing. This means that it has canonical class $\Q$-Cartier, and it is the quotient of its index 1 cover. Moreover, the singularity of the $3$-fold $\mathcal{W}$ is log terminal. As $K_W$ is nef and $W_t$ have nonnegative Kodaira dimension, we can apply the proof of \cite[Thm. 4.2]{Kawa92}. This means that the deformation factors through a surface $\B$ so that $\W \to \B$ is an elliptic fibration (\cite[Thm. 6.3]{Naka85}). Its restriction to the fibers of $\B$, which are nonsingular curves, defines an elliptic fibration $W \to B$, where $B$ is the corresponding fiber below $W$. Thus we must have M-modifications over fibers on an elliptic fibration $S \to B$ which were classified in Theorem \ref{kawamataQHD} and Lemma \ref{Mres II-III-IV}.
\end{proof}

%%%%%%%%%%%%%%%%%%%%%%%%%%%%%%%%%%%%%%%%%%%%%%%%%%%%%%%%%%%%%%%%%%%%%%%%%%%%%%%%%%%%%%%%%%%%%%%%%%%%%%%%%%%%%%%%%%%%%%%%%%%%%%%%%%%%%%%%%%%%%%%%%%%%%%%%%%
\section{Local-to-global obstructions to deform} \label{s4}

Let $W$ be a \QHD surface. We recall that $\Omega_W^1$ denotes the sheaf of differentials on $W$ \cite[II.8]{Har77}. Its dual $\mathcal{H}\text{om}_{\O_W}(\Omega_W^1,\O_W)$ is the tangent sheaf $T_W$ of $W$ \cite[II.8]{Har77}. Locally at each \QHD singularity, we will only consider its \QHD smoothings as in \cite[Thm. (1) and (2)]{Wahl_2013}. Since our singularities are isolated, to construct a global \QHD smoothing for $W$ over a disk $\D$ that glues these local \QHD smoothings, it is enough to have $H^2(W,T_W)=0$. This is because it implies the smoothness of the natural morphism from global deformations of $W$ to the product of local deformations of its singularities; see, for example, \cite[Prop. 6.4]{Wahl_1981}, \cite[Lem. 1]{Man91}, or \cite[Lem. 7.2]{Hack13}.

Let $\phi \colon X \to W$ be the minimal resolution of $W$ with (reduced) exceptional divisor $C$. The logarithmic tangent sheaf $T_X(-\log C)$ is the dual of the locally free sheaf $\Omega_X^1(\log C)$ of differentials with poles along $C$. The Leray spectral sequence gives the exact sequence $$ 0 \to H^1(W,\phi_*(T_X(-\log C))) \to H^1(X,T_X(-\log C)) \to  \ \ \ \ \ \ \ \ \ \ \ \ \ \ \ \ \ \ \ \ \ \ \ \ \ \ \ \ \ \ \ \ \ \ \ \ \ \ \ \ \ \ \ \ \ \ \ \ \ $$ $$  \ \ \ \ \ \ \ \ \ \ \ \ \ \ \ \ \ H^0(W,R^1 \phi_*(T_X(-\log C))) \to H^2(W,\phi_*(T_X(-\log C))) \to H^2(X,T_X(-\log C)).$$ The morphism $\alpha \colon H^1(X,T_X(-\log C)) \to H^0(W,R^1 \phi_*(T_X(-\log C)))$ represents the map from global to local deformations which preserve all the curves in the exceptional divisor.

\begin{theorem}
If $\alpha$ is onto and $H^2(X,T_X(-\log C))=0$, then we have no local-to-global obstructions to deform $W$. In particular, we have \QHD smoothings of $W$.
\label{ObstrThm}
\end{theorem}

\begin{proof}
As $\alpha$ is onto and $H^2(X,T_X(-\log C))=0$, we have $H^2(W,\phi_*(T_X(-\log C)))=0$ by the previous exact sequence. As in the proof of \cite[Thm. 2]{LP_2007}, we have the short exact sequence $$0 \to \phi_*(T_X(-\log C)) \to \phi_*(T_X)=T_W \to \Delta \to 0 $$ where $\Delta$ is supported at the singular points of $W$. Therefore, $H^2(W,\phi_*(T_X(-\log C)))=0$ implies $H^2(W,T_W)=0$.
\end{proof}

\begin{remark}
The sheaf $R^1 \phi_*(T_X(-\log C))$ is supported at the singular points of $W$. If a singular point $(p \in W)$ is taut in the sense of Laufer \cite{Laufer_1973}, then we can apply the proof of \cite[Lem. 1]{LP_2007} to show that $R^1 \phi_*(T_X(-\log C))_p=0$. By \cite{Laufer_1973}, we observe that Wahl and valency $3$ \QHD singularities are taut, and so over them $\alpha$ is automatically onto. We note that for valency $4$ \QHD singularities, we have in principle a cross ratio among the $4$ points that fixes the analytic type of the singularity. But, a valency $4$ \QHD singularity has a unique cross ratio by  \cite{Fowler13}.
\label{ObstrRem}
\end{remark}

We now work out a particular global \QHD smoothing of a \QHD singularity that will define its compactified Milnor fiber. We begin with a general birational operation over any \WHS singularity that will be used on singular surfaces and on their deformations ($3$-folds). We follow \cite[Section 1]{Wahl_2013}.

Let $(y \in Y)$ be the germ of a weighted homogeneous singularity $\spec A$, where $A=\oplus_{i\geq 0} A_i$ is a graded Noetherian domain with $A_0=\C$. If $Y=\spec A$, then it can be compactified as $\overline Y=\text{Proj} \, A[u]$ by adding the divisor $E_{\infty}=\text{Proj} \, A$. Let $I_s =\mathop{\oplus}\limits_{k\ge s} A_k$ be the weight filtration and consider the Rees algebra (graded by powers of $u$) $$R(A)=\mathop{\oplus}\limits_{s=0}^\infty I_su^s\subset A[u].$$

\begin{definition}
The blow-up of the weight filtration
$$\sigma \colon Z = \text{Proj} \, R(A)\to Y =\spec A$$
of the corresponding modification for a germ $(y \in Y)$ is called the \emph{Seifert partial resolution}.
\label{seifertDef}
\end{definition}
    
It is a birational morphism and the exceptional divisor is the irreducible curve $\sigma^{-1}(0)=C_0 \simeq E_\infty$. We can also compactify $Z \subset \overline Z$ by gluing $Z$ and $\overline Y \setminus \{y\}$. If $\dim Y=2$, then $C_0$ is a nonsingular curve and $\overline Y$ has $n$ c.q.s. ${1\over m_i}(1,q_i)$ along $C_0$ and $n$ dual singularities ${1\over m_i}(1,m_i-q_i)$ along $E_\infty$. We restrict ourselves to the case $C_0 \simeq \P^1$, because our singularity $(y \in Y)$ will be \QHD (they are rational). Then $\overline Y$ can be constructed from a Hirzebruch surface $\F_b$ by blowing up as shown in Figure \ref{CompactMilnorFig} and then contracting the $a$ and the $b$ chains of $\P^1$s to conjugate cyclic quotient singularities. The central $\P^1 \simeq C_0$ (represented by the $-b$ circle) is a negative section and maintains its self-intersection. We choose a section with positive self-intersection $b$ and denote by $E_{\infty}$ its proper transform on the blown-up surface, which we denote by $X$. Notice that $E_{\infty}^2=b-n$, where $n$ is the number of legs of the \WHS, as shown in the figure. We have contractions $X \to \overline Z \to \overline Y$.

\begin{figure}[htbp]
\includegraphics[width=13cm]{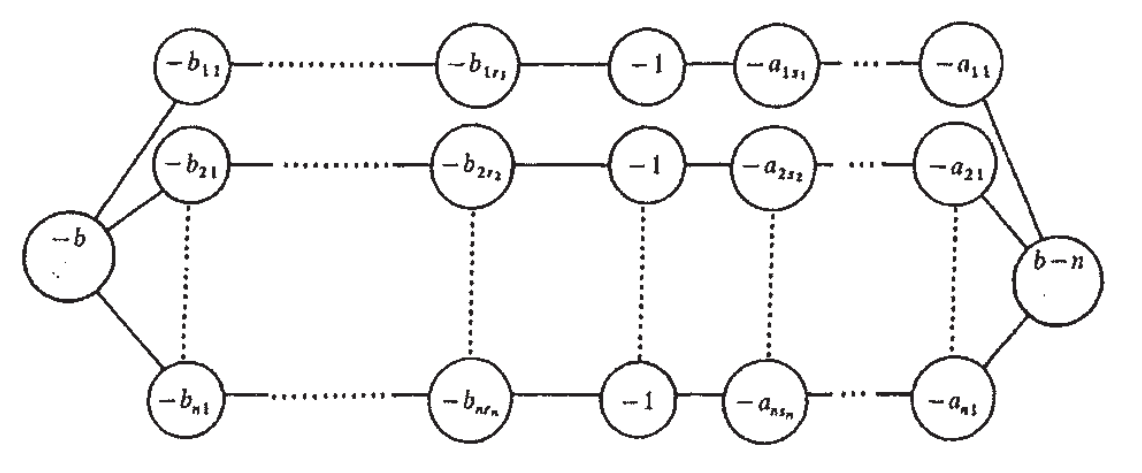}
\caption{\small{Minimal resolution $X$ of $\overline Y$; picture from \cite[Section 6.3]{P78}.}}  \label{CompactMilnorFig}
\end{figure}

\begin{lemma}
Let us construct $\overline{Y}$ for a \QHD singularity $(y \in Y)$. Then any local deformation of $(y \in Y)$ can be globalized to a deformation $(\overline{Y} \subset \overline{\Y}) \to (0 \in \D)$ that is trivial on $E_{\infty}$.
\label{CMF}
\end{lemma}

\begin{proof}
Let $D$ be the configuration of curves in Figure \ref{CompactMilnorFig}. We will prove that $$H^2(X,T_X(-\text{log}\, D))=0.$$ By Serre duality, $H^2(\F_b,T_{\F_b}(-\text{log}(C_0 + E_{\infty})))=H^0(\F_b,\Omega_{\F_b}^1(\text{log}(C_0+E_{\infty})) \otimes \Omega_{\F_b}^2)$, and we have $\Omega_{\F_b}^2 \sim -2C_0 -(2+b)F$, where $F$ is the class of a fiber. Using the residue sequence for $\Omega_{\F_b}^1\big(\text{log}(C_0+E_{\infty})\big)$, together with the fact that $C_0$ and $E_{\infty}$ are numerically independent, we conclude that $H^0(\F_b,\Omega_{\F_b}^1(\text{log}(C_0+E_{\infty})))$ vanishes. Hence, by the above, $H^2(\F_b,T_{\F_b}(-\text{log}(C_0 + E_{\infty})))=0$. We now apply the add-delete principle for $(-1)$-curves \cite[Section 4]{PSU13} to show that $H^2(X,T_X(-\text{log} \, D))=0$. (If we blow up and add the exceptional divisor to the configuration, we preserve the cohomology $H^2$. If we add a $(-1)$-curve intersecting the divisor with normal crossings, then we also preserve $H^2$.)

We can now apply Theorem \ref{ObstrThm} to $\overline{Y}$, because we can freely choose the fibers over the central curve $C_0$ to construct the star graph of the \QHD singularity. Thus, we have global deformations keeping the c.q.s. over $E_{\infty}$, although we may not preserve $E_{\infty}$. We address this by repeatedly blowing up the intersection points between $E_{\infty}$ and the $D_i$'s (this takes place over $X$) until $E_{\infty}$ becomes sufficiently negative to allow us to blow down the configurations $D_0$ (\QHD original star graph), $D_1, \ldots, D_n$ (chains to c.q.s.s), and $E_{\infty}$ (to a single c.q.s.). If $\overline{Y}'$ denotes the resulting singular surface, one can prove that Theorem \ref{ObstrThm} applies to it. Therefore, we consider an arbitrary deformation of the \QHD singularity, together with trivial deformations for the remaining singularities, which glue to a deformation of $\overline{Y}'$. This deformation can be resolved simultaneously for the cyclic quotient singularities. Finally, we contract any possible $(-1)$-curves (in this family over $\D$), and then consider the blow-down deformation for the cyclic quotient singularities coming from the original $D_1, \ldots, D_n$ (those from the original $X$).
\end{proof}

\begin{definition}
Consider in Lemma \ref{CMF} a \QHD smoothing for $(y \in Y)$. The general fiber of $(\overline{Y} \subset \overline{\Y}) \to (0 \in \D)$ will be called the \textit{compactified Milnor fiber} of the \QHD smoothing for $(y \in Y)$. 
\label{CompactMilnorFiberDef}
\end{definition}

The list of minimal resolutions of compactified Milnor fibers of \QHD singularities is summarized in \cite[Section 7]{Wahl_21}. In general, it can be proved that the compactified Milnor fiber is a rational surface with Picard number equal to $1$. The converse is a consequence of Pinkham's result \cite[6.7]{P78}. This is due to Wahl \cite[Thm. 8.1]{SSW_2008}.  

\begin{theorem}
Let $\overline{Y}$ be the compactification of a rational \WHS. Let $\tilde Z$ be a nonsingular rational surface, and let $D \subset \tilde Z$ be a union of nonsingular rational curves isomorphic to the complement of $Y$ in the minimal resolution of $\overline Y$ at all singular points along $E_{\infty}$
(i.e., the right configuration in Figure \ref{CompactMilnorFig}).
Let $\tilde Z\to Z$ be a contraction of all $a$-chains to c.q.s.
Assume that $b_2(D)=b_2(\tilde Z)$, i.e., that $Z$ has Picard number $1$. Then there is a \QHD smoothing of $Y$
with compactified Milnor fiber isomorphic to $Z$. Moreover, this smoothing has
``negative weight'', that is, given by $\Y=\spec R \to \spec \C[t]$ where
$R= \oplus_{i\geq 0} R_i$ is a positively graded ring with $R_0=\C$, $t \in R_1$, and $\spec R/tR = Y$. Every \QHD smoothing appears that way.
\label{PinkWahl}
\end{theorem}

Let $(Y \subset \Y) \to (0 \in \D)$ be a \QHD smoothing of negative weight as in the theorem. It is proved in \cite[Section 2]{Wahl_2013} that $\Y$ is log terminal and the smoothing is $\Q$-Gorenstein. We consider the Seifert partial resolution (Definition \ref{seifertDef}) $$\sigma \colon \Y' \to \Y $$ 
where the exceptional divisor is $\overline{M}=\text{Proj}(R)$. Notice that $\overline{M}$ is isomorphic to the compactified Milnor fiber \cite[6.4]{P78}.
Let $\Y' \to \D$ be the induced smoothing. The central fiber is reduced and composed of two surfaces, one is $\overline{M}$ and the other is the Seifert partial resolution $Z$ of $Y$. They intersect along a $\P^1=\text{Proj}(R/tR)$.
The analytic $3$-fold $\Y'$ is $\Q$-factorial, terminal, and $\Q$-Gorenstein.
The central fiber has orbifold normal crossing singularities $(xy=0)\subset {1\over m}(q,-q,1)$ at the terminal singularities of $\Y'$,
matching cyclic quotient singularities of $Z$ with conjugate singularities of $\overline{M}$.

\bigskip 

To conclude this section, we construct a \QHD smoothing of a \QHD surface $W$ such that the smooth fiber is a simply connected surface of general type with $p_g = 0$ and $K^2 = 1$. The surface $W$ has a single singularity of valency $3$ and type (c) with $q = r = 1$. We can also obtain numerous examples with non-Wahl \QHD singularities of $p_g = 0$ surfaces of general type with $K^2 = 1,2,3,4$ \cite{programaReyes}. These will be presented in subsequent work focused on KSBA surfaces.

\begin{example}
We start with an elliptic rational surface $S$ with sections, and $I_3+9I_1$ as singular fibers. Choose an $I_1$ and denote $A+B+C$ the components of $I_3$ as in Figure \ref{GenTypeEx}. We also consider a double section $M$ as shown. We blow-up $5$ times $X \to S$, see the figure. Then $$K_X \sim -F_{\text{gen}} + E_1 + 2E_2+ E_3+ E_4 +E_5$$ where $F_{gen}$ is the general fiber and (in this example) $E_i$ are exceptional curves. Let $\phi \colon X \to W$ be the contraction of the valency 3 type (c) star graph. 

\begin{figure}[htbp]
\includegraphics[width=11cm]{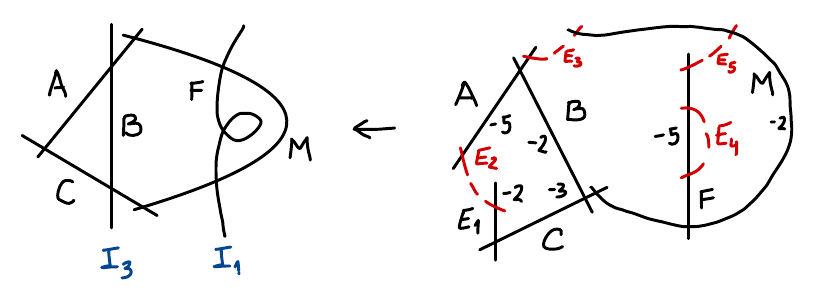}
\caption{\small{$W$ with a valency 3 type (c) ($q=r=1$) singularity and $K_W^2=1$.}}  
\label{GenTypeFig}
\end{figure}

By using the same strategy as in \cite{LP_2007} and that valency 3 \QHD singularities are taut, we can prove via Theorem \ref{ObstrThm} that there are no obstructions to deform $W$, and so we consider a global \QHD smoothing. One can show that $K_W$ is ample and $K_W^2=1$. The computation of the fundamental group, which is trivial, follows Lee-Park strategy \cite{LP_2007}, and the computation of Mumford \cite{Mumford_1961}; we omit it. Let $W' \to W$ be the Seifert partial resolution of $[5-2,2,5-2;3]$. The central curve is $C$. Let $\psi \colon X \to W'$ be the minimal resolution. Then $$\psi^*(K_{W'}+C) = K_X + C + \frac{8}{9} B + \frac{7}{9} A + \frac{8}{9} M + \frac{7}{9} F + \frac{1}{2} E_1.$$ Then we can do Lee-Park trick $F_{gen} \sim F + 2E_4 + E_5 \sim A+B+C+2E_1+3E_2$, and so by just intersecting we have that $K_{W'}+C$ is ample. Then one can prove that the Seifert partial resolution of the \QHD smoothing is in fact the KSBA limit of this family.
\label{GenTypeEx}
\end{example}

%%%%%%%%%%%%%%%%%%%%%%%%%%%%%%%%%%%%%%%%%%%%%%%%%%%%%%%%%%%%%%%%%%%%%%%%%%%%%%%%%%%%%%%%%%%%%%%%%%%%%%%%%%%%%%%%%%%%%%%%%%%%%%%%%%%%%%%%%%%%%%%%%%%%%%%%%%
\section{\QHD degenerations of Dolgachev surfaces} \label{s5}

Let $W$ be a \QHD surface. As in Section \ref{s3}, assume that it admits a minimal elliptic fibration $g_W \colon W \to B$ for some nonsingular projective curve $B$. We have the diagram of morphisms 
     \[\xymatrix{  X  \ar[d]_{\pi} \ar[r]^{\phi} & W \ar[d]^{g_W} \\ S \ar[r]^{g_S} & B}\] 
where $g_S \colon S \to B$ is a relatively minimal elliptic fibration. We have the numerical canonical formula $$K_S \equiv \bar \lambda F$$ where $\bar \lambda \in \Q$ and $F$ is a fiber. Let $F_1,\ldots,F_s$ be the fibers of $g_S$ over which $\pi$ blows-up. By Proposition \ref{Nostar}, we have $\pi(C)=F_1 \cup F_2 \cup \ldots \cup F_s$, and $F_i$ can be of type $I_n$ for some $n$, $II$, $III$, or $IV$. The possibilities for $f^{-1}(F_i)$ are classified in Theorem \ref{kawamataQHD} and Lemma \ref{Mres II-III-IV}, see Figures \ref{fig:KAWA_WAHL} and \ref{fig:KAWA_QHD}. 

For an effective divisor $D$, let $\mu_D(\Gamma)$ be the multiplicity of $\Gamma$ in the support of $D$.

\begin{lemma} 
We have 
$$K_W \equiv \big( \bar \lambda + \sum_{i=1}^s \lambda_i \big) F,$$ where $F$ is a fiber of $g_W$, $\lambda_i=\frac{\mu_E(\Gamma_i)}{\mu_{f^*(F_i)}(\Gamma_i)} \in \Q_{>0}$, and $\Gamma_i$ is any component of $f^{-1}(F_i)$ not contracted by $\phi$.
\label{CanonicalClassFormula}
\end{lemma}

\begin{proof}
We have $K_S \equiv \bar \lambda F$, and so $K_X \equiv \bar \lambda F + E$. As in Remark \ref{facts QHD elliptic surface} (a), we have that $E-\sum_{\Gamma \in C} d(\Gamma) \, \Gamma \equiv \sum_{i=1}^s \lambda_i \pi^*(F_i)$ for some $\lambda_i$s, and so $\phi^*(K_W) \equiv \big( \bar \lambda + \sum_{i=1}^s \lambda_i \big) F$, where $F$ is a fiber of $g_W \circ \phi$. As $\phi$ is a contraction into rational singularities, we have $K_W \equiv \big( \bar \lambda + \sum_{i=1}^s \lambda_i \big) F$ where $F$ is a fiber of $g_W$. To compute the $\lambda_i$, we note that $E-\sum_{\Gamma \in C} d(\Gamma) \, \Gamma \equiv \lambda_i \pi^*(F_i)$ when restricted to $F_i$, and so for a component $\Gamma_i$ of $\pi^*(F_i)$ not in $C$ we must have $\mu_E(\Gamma_i)=\lambda_i \mu_{f^*(F_i)}(\Gamma_i)$ by \cite[III, (8.2) Lemma]{BWHKPCV_2004}.
\end{proof}

In \cite[Thm. 4.4]{Kawa92}, Kawamata writes a similar formula in the case of $W$ with only Wahl singularities. Of course, we can compute $\lambda_i$ using any component of $\pi^*(F_i)$ with different formulas. In the next corollary, we conveniently use the central curve of the corresponding star graph to compute $\lambda_i$. 

%\red{XXX Corolario: Un QHD smoothing de $W$ construido a través de $S$ produce siempre alguna fibra multiple (and so for Kodaira 1 we have only Enriques) XXX}

\begin{corollary} 
The $\lambda_i$ for $W$ and the corresponding $\lambda'_i$ for a sliding $W'$ of $W$ are equal.
\label{sameLambdai}
\end{corollary} 

\begin{proof}
The sliding can only be performed over fibers $F_i$ of type $II$, $III$, and $IV$. The M-modification of $F_i$ contains the central curve $\Gamma_0$ of the corresponding unique star graph. By the same proof of the previous lemma, we have $\mu_E(\Gamma_0)-d(\Gamma_0)=\lambda_i \mu_{f^*(F_i)}(\Gamma_0)$, where $\Gamma_0 \in C$ now. On the other hand, a sliding over $F_i$ will define surfaces $X'$ and $W'$ over the same elliptic fibration $g_S \colon S \to B$ and with the same $\mu_E(\Gamma_0)$, $d(\Gamma_0)$ and $\mu_{f^*(F_i)}(\Gamma_0)$. Therefore, $\lambda'_i$ for the $F_i$ below $X'$ is equal to $\lambda_i$.
\end{proof}

%In the next corollary, we will only consider the ``general" M-modifications of fibers in Figures \ref{fig:KAWA_WAHL} and \ref{fig:KAWA_QHD}. This is to simplify to work with different multiplicities in the components of a multiple fiber of $g_W$.

%\begin{corollary}
%Let $\bar F_1,\ldots, \bar F_s$ be the reduced images of $\pi^*(F_i)$ under $\phi$, and let $\bar F_{s+1}, \ldots \bar F_{s+r}$ be the reduced images under $\phi$ of the other (possible) multiple fibers of $X \to B$. Let $\mu_i$ ($\bar \mu_i$) be the multiplicities of the corresponding fibers $F_i$ ($\bar F_i$) of $g_S$ ($g_W$). Then we have the numerical canonical formula $$K_W \equiv (\chi(\O_S)-2\chi(\O_B))F + \sum_{i=1}^s \big( \big(\frac{\mu_i-1}{\mu_i} \big)\bar \mu_i +\lambda_i \bar \mu_i \big) \bar F_i + \sum_{i=s+1}^r  (\mu_i-1) \bar F_i.$$ We have $\bar \mu_i = \mu_{f^*(F_i)}(\Gamma_i)$ and $\mu_i$ divides $\bar \mu_i$. 
%\label{K_W}
%\end{corollary}

%\begin{proof}
%By Lemma \ref{CanonicalClassFormula} and the canonical class formula \cite[V (12.3)]{BWHKPCV_2004}, we obtain the given formula except for the $\bar F_i$ with $i=1,\ldots,s$. We now note that for all M-modifications in Figures \ref{fig:KAWA_WAHL} and \ref{fig:KAWA_QHD}, we have that every component of the fiber corresponding to $\bar F_i$ has the same multiplicity $\bar \mu_i = \mu_{f^*(F_i)}(\Gamma_i)$. Then the computation of the coefficient for $\bar F_i$ for $i=1,\ldots s$ is $\big(\frac{\mu_i-1}{\mu_i} \big)\bar \mu_i +\lambda_i \bar \mu_i$. We now go case by case to prove that this coefficient is $\bar \mu_i-1$. \end{proof}

\begin{table}[!htbp]
    \centering
    \small % Tamaño de fuente reducido para acomodar 8 columnas
    \renewcommand{\arraystretch}{1.3} % Espaciado vertical para las fracciones
    \begin{tabular}{|c|c|c|c|c|c|c|c|}
    \hline
    M-mod & $\lambda$ & M-mod & $\lambda$ & M-mod & $\lambda$ & M-mod & $\lambda$ \\
    \hline
    II.v3.a & $6/7$ & II.v3.i & $\frac{8+3q}{9+3q}$ & III.v3.g & $\frac{6+2q}{7+2q}$ & $y \text{I}_d(n,a)$ & $\frac{n-1}{yn}$ \\
    \hline
    II.v3.b & $7/8$ & II.v3.j & $\frac{7+2q}{8+2q}$ & III.v3.h & $\frac{5+2q}{6+2q}$ & II(2) & $1/2$ \\
    \hline
    II.v3.c & $6/7$ & II.v4.c & $\frac{11+6p}{12+6p}$ & III.v4.b & $\frac{7+4p}{8+4p}$ & II(3) & $2/3$ \\
    \hline
    II.v3.d & $\frac{8+2q}{9+2q}$ & III.v3.a & $4/5$ & IV.v3.a & $\frac{2+p}{3+p}$ & II(4) & $3/4$ \\
    \hline
    II.v3.e & $\frac{9+3q}{10+3q}$ & III.v3.b & $4/5$ & IV.v3.b & $\frac{3+q}{4+q}$ & II(5) & $4/5$ \\
    \hline
    II.v3.f.0 & $5/6$ & III.v3.c.0 & $3/4$ & IV.v3.c & $\frac{3+q}{4+q}$ & III(2) & $1/2$ \\
    \hline
    II.v3.f & $\frac{5+q}{6+q}$ & III.v3.c & $\frac{3+p}{4+p}$ & IV.v4.a.0 & $5/6$ & III(3) & $2/3$ \\
    \hline
    II.v3.g & $\frac{10+3q}{11+3q}$ & III.v3.d & $\frac{4+q}{5+q}$ & IV.v4.a & $\frac{5+3p}{6+3p}$ & IV(2) & $1/2$ \\
    \hline
    \end{tabular}
    \caption{\small{The $\lambda$ in Lemma \ref{CanonicalClassFormula} for all M-modifications of fibers.}} 
    \label{lambda}
\end{table}

We will now construct \QHD degenerations of surfaces in $D_{a,b}$ with $a,b \geq 2$ but not both equal to $2$. This is similar to \cite[Section 4]{U16a} for the case of Wahl singularities. We start with a relatively minimal rational elliptic fibration $g_S \colon S \to \P^1$. In this way, it might have (exactly) one multiple fiber or none. We will only use a multiple fiber when we construct the M-modification $yI_1(x,z)$. We assume that $g_S$ has a fiber $F_1$ of type $yI_1$ for some $y \geq 1$, $II$, $III$, or $IV$, and a fiber $F_2=I_1$.

Let us choose one M-modification over $F_1$ of any type from Figures \ref{fig:KAWA_WAHL} or \ref{fig:KAWA_QHD}, and the M-modification with Wahl chain $[b+2,2,\ldots,2]$ ($b-2$ $2$s) over $F_2$. Then we have the corresponding composition of blow-ups $\pi \colon X \to S$. Let $\phi \colon X \to W$ be the contraction of the \QHD star graph and the Wahl singularity (ies) in the M-modifications. 

Let $a$ be the denominator of $\lambda$ in Table \ref{lambda} for the M-modification of $F_1$. 

\begin{corollary}
The canonical divisor $K_W$ is nef but not numerically trivial for $a,b \geq 2$ with $ab>4$. In fact, $K_W \equiv (1- \frac{1}{a}-\frac{1}{b})F$. 
\label{Knef}
\end{corollary}

\begin{proof}
By Lemma \ref{CanonicalClassFormula} we have $K_W \equiv (\bar \lambda+\lambda_1+\lambda_2)F$. Consider two cases: if $g_S$ has a multiple fiber (which is $F_1$), then $\bar \lambda=-1+\frac{y-1}{y}$ and $\lambda_1=\frac{n-1}{ny}$; otherwise, $\bar \lambda=-1$. In both cases, we obtain $$\bar \lambda+\lambda_1+\lambda_2 = -1 + \frac{a-1}{a} + \frac{b-1}{b},$$ since the $\lambda_i$ in Figure \ref{lambda} have this form. Hence $ab>4$ implies that $K_W$ is nef and not numerically trivial.
\end{proof}

\begin{lemma}
The surface $W$ has no local-to-global obstructions to deform. In particular, we have \QHD smoothings for $W$.
\label{ObstrMainConstr}
\end{lemma}

\begin{proof}
In the case of valency 4 M-modification, we can freely choose the fourth point to blow up on the central curve of the star graph. Following Theorem \ref{ObstrThm}, this implies, just as in the proof of Lemma \ref{CMF}, that we only need to compute $H^2(X,T_X(- \log C))=0$ to prove that $W$ is unobstructed to deform.

For this computation, we will use the strategy in \cite[Section 4]{PSU13}. To make it work, however, we need to compute the initial vanishing $H^2(S',\Omega_{S'}^1(\log (T_1+T_2))(K_{S'}))=0$, as is done in the proof of \cite[Thm. 2.1]{PSU13}. Here $S'$ is obtained by blowing up the node of $F_2$ and the points over $F_1$ needed to obtain the minimal simple normal crossings fiber $F'$, as in Figure \ref{s.n.c config} (including the case of $yI_1$, where we blow up once at the node). The $T_i$s are the corresponding reduced fibers minus the $(-1)$-curve. Let $E_1$ be the $(-1)$-curve over $F_1$, and let $E_2$ be the $(-1)$-curve over $F_2$. We proceed case by case on $F_1$ to show that the critical step in the proof of \cite[Thm. 2.1]{PSU13} holds:

\begin{itemize}
    \item[I.] In this case we have a multiple fiber $yI_1$. We have the exact same proof but with $K_S = -F +(y-1){F_1}_{red}$. If $A$ is the $(-4)$-curve, then $K_{S'}+A=E_2-(2y-1)E_1$. Then the proof follows as in \cite[Thm. 2.1]{PSU13}.
    
    \item[II.]: Let $T_1=A+B+C$ where $A$ is the $(-6)$-curve, $B$ is the $(-3)$-curve, and $C$ is the $(-2)$-curve. Then $K_{S'}+A+B+C=E_1-2E_2$, and the proof follows as in \cite[Thm. 2.1]{PSU13}.

    \item[III]: Let $T_1=A+B+C$ where $A$ is a $(-4)$-curve, $B$ is the other $(-4)$-curve, and $C$ is the $(-2)$-curve. Then $K_{S'}+A+B+C=E_1-2E_2$, and the proof follows as in \cite[Thm. 2.1]{PSU13}. 

    \item[IV]: Let $T_1=A+B+C$ where $A,B,C$ are the $(-3)$-curves. Then $K_{S'}+A+B+C=E_1-2E_2$, and the proof follows as in \cite[Thm. 2.1]{PSU13}.  
\end{itemize}
\end{proof}

\begin{theorem}
Assume that $a,b \geq 2$ with $ab>4$. Then there are \QHD degenerations of Dolgachev surfaces in $D_{a,b}$ into $W$. In addition, the \QHD smoothing of only the Wahl singularity $\frac{1}{b^2}(1,b-1)$ in $W$ is a surface with one \QHD singularity whose relative minimal model of the minimal resolution is in $D_{1,b}$ (a Halphen elliptic surface of index $b$).
\label{DolgachevSmoothings}
\end{theorem}

\begin{proof}
By Lemma \ref{ObstrMainConstr}, we can consider a \QHD smoothing $(W\subset\mathcal{W})\to(0\in\mathbb{D})$ with nonsingular fibers $W_t$. Since $K_W$ is nef with $K_W^2=0$ (Corollary \ref{Knef}), we have $K_{W_t}$ nef, $K_{W_t}^2=0$, and $K_{W_t}$ is not numerically trivial as $a,b \geq 2$ but not both equal to $2$. Then the Kodaira dimension of $W_t$ is equal to $1$.

As in \cite[Thm. 4.2]{Kawa92}, we have by \cite[Thm. 6.3]{Naka85} that $(W\subset\mathcal{W})\to(0\in\mathbb{D})$ factors through a surface $\B$ so that $\W \to \B$ is an elliptic fibration. Its restriction to the fibers of $\B$, which are $\P^1$s, gives the elliptic fibration $g_W \colon W \to \P^1$. By construction, this elliptic fibration has two multiple fibers $\bar F_1$ and $\bar F_2$, corresponding to $F_1$ and $F_2$ in $S$.

By \cite[V (12.3)]{BWHKPCV_2004} and the invariance of $\chi(\O_{W_t})=\chi(\O_W)$, we can write its canonical class as $K_{W_t} \equiv (-1 + \sum_{i=1}^s \frac{\mu_i-1}{\mu_i})F_t$, where $\mu_i$ are the multiplicities of the $s$ multiple fibers of the elliptic fibrations $g_{W_t} \colon W_t \to \P^1$, and $F_t$ is the class of a fiber. (Note that by \cite[Prop. 7.1]{FM94}, the number of multiple fibers and multiplicities is constant in smooth families.) We also note that $F_t$ degenerates into the class of the fiber $F$ of $g_W$.

As the canonical class $K_{W_t}$ degenerates into $K_W$, by Lemma \ref{CanonicalClassFormula} we have $$\bar \lambda + \lambda_1 + \lambda_2 = -1 + \sum_{i=1}^s \frac{\mu_i-1}{\mu_i}.$$ By Corollary \ref{Knef}, we have $ -1 + \frac{a-1}{a} + \frac{b-1}{b} = -1+\sum_{i=1}^s \frac{\mu_i-1}{\mu_i}$.

For the \QHD smoothing of the Wahl singularity from the Wahl chain $[b+2,2,\ldots,2]$, we know that it is the same as a logarithmic transformation of the original $F_2$ fiber of index $b$; This is in the original rational blowdown work of Fintushel-Stern \cite{FS97}. Therefore, in $W_t$ we have a multiple fiber of multiplicity $\mu_1=b$ and it is the only multiple fiber degenerating to $\bar F_2$.

Then by the equation above, we have $\frac{a-1}{a} =s-1-\sum_{i=2}^s \frac{1}{\mu_i} \geq \frac{s-1}{2}$, and so $s=2$ and $a=\mu_2$. The surface $W_t$ is then a Dolgachev surface $D_{a,b}$. 

For the last statement, as $W$ has no obstructions to deform, we consider the \QHD smoothing of the Wahl singularity only, being trivial on the \QHD singularity. We now resolve in the family this \QHD singularity. Then we can contract all $(-1)$-curves in the special fiber over the \QHD singularity, producing divisorial contractions in the family (by stability of $(-1)$-curves, see \cite[IV Section 4]{BWHKPCV_2004}), until we arrive to the original fiber of type II, III, or IV. Now, this is precisely the situation of \cite[Thm. 4.2 ($-\infty$)]{U16a}, and so it must be a Halphen surface of index $b$.
\end{proof}

%--- Corollario: The same happens for all slidings of the \QHD star graphs.

\begin{corollary}
Let $W$ be a \QHD surface with one \QHD singularity. Let $(W\subset\mathcal{W})\to(0\in\mathbb{D})$ be a \QHD smoothing with $W_t \in D_{a,b}$ and $\kappa(W_t)=1$. Then, up to the choice of $a$ or $b$, the associated $\lambda$ is $\frac{a-1}{a}$ and $S \in D_{a',b}$ for some divisor $a'$ of $a$.
\label{allDegenerationsDolgachev}
\end{corollary}

\begin{proof}
By the classification Theorems \ref{kawamataQHD} and \ref{TheoremClassM-resEllipticFiber}, the \QHD singularity lies on an elliptic fiber of a fibration $W \to \P^1$, and it is one of the possibilities listed in the corresponding tables. As in the proof of the previous theorem, one of the multiple fibers of $W_t$, say the one with multiplicity $a$, deforms to the multiple fiber containing the \QHD singularity, and $\lambda=\frac{a-1}{a}$. Since this is a \QHD smoothing, we have that $p_g=q=0$ remain constant in the family. Therefore, by the canonical formula in Lemma \ref{CanonicalClassFormula}, the contribution of $\sum_{i=1}^r \frac{b_i-1}{b_i}$ from the possible additional multiple fibers in $W \to \P^1$ must equal $\frac{b-1}{b}$. It follows that $r=1$ and $b_1=b$. Therefore, the minimal model $S$ of the minimal resolution of $W$ is a Dolgachev surface $D_{a',b}$. If the \QHD singularity is not Wahl, then $a'=1$. Otherwise, $a'n=a$, where $n$ is the index of the Wahl singularity, and hence $a'$ divides $a$. 
\end{proof}

\begin{remark}
The smallest possible $a,b$ are $2,3$. In this case, Dolgachev surfaces in $D_{2,3}$ are simply connected, something that happens whenever $\gcd(a,b)=1$. By inspecting Table \ref{lambda}, we have II(2), III(2), and IV(2) with $a=2$, so we take $b=3$. We also have IV.v3.a with $p=0$ (that is, IV.v3.a.0, which is log canonical), II(3), and III(3) with $a=3$, so we take $b=2$. Thus, in total we have $6$ possibilities for the combinations of M-modifications for $D_{2,3}$, and only one non-Wahl \QHD singularity appears, and it is log canonical. As soon as we increase the indices, we see more and more non-Wahl \QHD singularities appearing.
\label{rmk1-section6}
\end{remark}

\begin{remark}
In the case of valency 4 \QHD singularities, we sometimes have two distinct \QHD smoothings (see Remark \ref{factsQHD} (5)), but they produce Dolgachev surfaces with the same multiplicities. For example, this is the case for IV.v3.a for any $p$. It has $\lambda=\frac{2+p}{3+p}$, and thus it appears in all Dolgachev surfaces $D_{p+3,b}$ with Kodaira dimension $1$.
\label{rmk2-section6}
\end{remark}

\begin{remark}
We can perform the same computation of obstructions for the slidings and Aslidings of the M-modifications in Figure \ref{fig:KAWA_QHD}, obtaining that any such $W$ is unobstructed to deform. The \QHD smoothing of the new Wahl singularity(ies) recovers the original M-modifications in Figure \ref{fig:KAWA_QHD}. Thus, by Theorem \ref{DolgachevSmoothings}, Dolgachev surfaces also degenerate to these slidings. We recall that the $\lambda$s are equal by Corollary \ref{sameLambdai}. This \QHD boundary for the surfaces in $D_{a,b}$ will be modified to a nef limit with only slc singularities in the next section.
\label{rmk3-section6}
\end{remark}

%%%%%%%%%%%%%%%%%%%%%%%%%%%%%%%%%%%%%%%%%%%%%%%%%%%%%%%%%%%%%%%%%%%%%%%%%%%%%%%%%%%%%%%%%%%%%%%%%%%%%%%%%%%%%%%%%%%%%%%%%%%%%%%%%%%%%%%%%%%%%%%%%%%%%%%%%%
\section{Minimal slc limits of \QHD degenerations} \label{s6}

Consider a \QHD smoothing $(W \subset \mathcal{W}) \to (0 \in \mathbb{D})$ of a surface $W$ as in Theorem \ref{DolgachevSmoothings}, which has a non-log canonical \QHD singularity. In this section, we construct a semi-log canonical replacement of this degeneration. The process consists of two steps: first, we consider the Seifert partial resolution of the \QHD smoothing, as explained after Theorem \ref{PinkWahl}; then, we perform flips to obtain a minimal slc model.

\begin{remark}
To simplify the computations, we will consider only \QHD singularities constructed from an elliptic fiber of \textbf{type II} in Figure \ref{fig:KAWA_QHD}. We will always use the slide of the unique $(-1)$-curve in this fiber; consequently, the special fiber contains one \QHD singularity and one Wahl singularity. It will be clear that cases III and IV can be treated in a similar way. 
\label{asumir}
\end{remark}

Let $W$ be the surface constructed for Theorem \ref{DolgachevSmoothings} with a non-log canonical \QHD singularity, its sliding Wahl singularity, and the Wahl singularity $\frac{1}{b^2}(1,b-1)$. We have an elliptic fibration $g_W \colon W \to \P^1$ with two multiple fibers $\bar F_1$ and $\bar F_2$ of multiplicities $a$ and $b$, respectively, where $a$ is the denominator of the associated $\lambda$. As there are no obstructions by Lemma \ref{ObstrMainConstr}, we choose local \QHD smoothings as in Theorem \ref{PinkWahl} for all singularities and globalize them to obtain a \QHD smoothing $(W\subset\mathcal{W})\to(0\in\mathbb{D})$. We now consider the Seifert partial resolution (as described right after Theorem \ref{PinkWahl}) $$(W' \cup Z) \subset \mathcal{W'} \to (W\subset\mathcal{W})\to(0\in\mathbb{D})$$  where

\begin{itemize}
    \item The curve $\Delta:= W' \cap Z$ is a $\P^1$.
    
    \item The restriction $(\Delta \subset W') \to (P \in W)$ is the partial resolution of the \QHD singularity $(P \in W)$, which only extracts the central curve $\Delta$ of the star graph. The surface $W'$ has $3$ or $4$ c.q.s. $\frac{1}{m_i}(1,q_i)$ that correspond to its legs.     
    
    \item The surface $Z$ is the compactified Milnor fiber of the \QHD smoothing of $(P \in W)$, and it has $3$ or $4$ c.q.s. $\frac{1}{m_i}(1,m_i-q_i)$ in $\Delta$ at the same points of $W'$. Thus, the fiber $W' \cup Z$ has orbifold normal crossing singularities as in \cite[Section 5]{Hack12}.     
\end{itemize}

Wahl \cite{Wahl_2013} proves that the $3$-fold $\mathcal{W}'$ has log terminal singularities and the smoothing is $\Q$-Gorenstein. We now run MMP relative to $\D$ through the steps:
\bigskip

\textbf{(1)} Let $\Gamma^{-}$ be the strict transform of the only component of $\bar F_1$ in $W'$. We have that $\Gamma^-$ is a negative for $K_{\W'}$. Indeed, we can show that $$\Gamma^- \cdot K_{\W'}=\Gamma^- \cdot (K_{W'}+\Delta|_{W'})=- \frac{\chi}{e \, m_1},$$ where $\frac{1}{m_1}(1,q_1)$ is the c.q.s. in $\Gamma^{-}$, and $\frac{\chi}{e}$ is the invariant of $(P \in W)$ (see Section \ref{s1}). As the \QHD singularity is not log canonical, we have $\Gamma^-.K_{\W'}=- \frac{\chi}{e \, m_1}<0$. We recall that for the non-log canonical case we have $0< \frac{\chi}{e} <1$ (see Remark \ref{factsQHD} part (7)).
\bigskip

\textbf{(2)} We also have that ${\Gamma^{-}}^2<0$ in $W'$, and so it is an extremal ray for $K_{W'}+\Delta|_{W'}$. Therefore, by the Cone theorem \cite[Section 3.3]{KM_1998} for the log terminal pair $(W',\Delta)$, there is $H$ nef divisor on $W'$ so that $nH-(K_{W'}+\Delta|_{W'})$ is ample for $n>>0$. Let us fix such an $n$. We have that $|mH|$ induces the contraction (birational morphism) of $\Gamma^{-} \subset W'$ for any $m\geq n$.   
\bigskip

%(3) We have $\Delta|_{W'} \cdot (K_{W'}+\Delta|_{W'})=\Delta|_Z \cdot (K_Z+\Delta|_Z)$ because of the global canonical sheaf $K_\mathcal{W'}|_{W'+Z}$.
%\bigskip

\textbf{(3)} On the projective surface $Z$, we take a very ample line bundle $J$. Let us choose $u,v>0$ integers such that $v\Delta|_Z \cdot J=u\Delta|_{W'} \cdot (nH)$. We also choose it so that $vJ-(K_Z+\Delta|_Z)$ is ample.
\bigskip

\textbf{(4)} In this way, we have a line bundle $M$ on $W' \cup Z$ with ${M}|_{W'}=unH$ and ${M}|_{Z}=vJ$ (see \cite[Lem. 7.3]{Hack13}) as $\Delta=\P^1$. By \cite[Prop. 4.2]{Hack13}, since $p_g(W')=p_g(Z)=0$ and $q(W')=q(Z)=0$, we have that the line bundle $M$ lifts to a unique line bundle $\M$ on $\W'$.
\bigskip

\textbf{(5)} We note that $\M-K_{\W'}$ is an ample line bundle on $\W'$ over $\D$ possibly by shrinking $\D$. Indeed, the restriction to $W'$ and to $Z$ is ample by construction, and so it is for $W' \cup Z$. By \cite[Prop. 1.4]{Naka85}, we have that $\M-K_{\W'}$ is ample over a possibly smaller disk $\D$. This implies that $\M$ is nef relative to $\D$.   
\bigskip

\textbf{(6)} We now apply \cite[Thm. 4.8]{Naka85} to show the existence of a proper surjective morphism $\W' \to \Y$ into a normal $3$-fold $\Y$ over $\D$, which must be birational since it is given by some power of $\M$, and this defines a birational morphism over the singular fiber. Moreover, it only contracts the curve $\Gamma^{-}$. This defines an extremal neighborhood $(\Gamma^{-} \subset \W') \to (Q \in \Y)$ of flipping type in the terminology of \cite{KM92} (see also \cite{Mori02}).   
\bigskip

\textbf{(7)} In addition, our extremal neighborhood $(\Gamma^{-} \subset \W') \to (Q \in \Y)$ is exactly as in \cite[Thm. 10.6]{Kawa88} (see its proof, case (2)). In that theorem, Kawamata proves that there is a section of $\O_{\Y}(-K_{\Y})$ whose divisor has a rational point of type A at $Q \in \Y$. Thus, by definition, it is semistable of type k1A or k2A \cite{KM92}. As we have two singularities, it is of type k2A, and they were classified by Mori in \cite{Mori02}. Mori proves that we have a flip $(\Gamma^{+} \subset \W^+) \to (Q \in \Y)$. The flipped curve $\Gamma^+$ must be over the central fiber, which is formed by two surfaces $W^+$ and $Z^+$, which correspond to proper transforms of $W'$ and $Z$. On $\Delta^+=W^+ \cap Z^+$ we have only orbifold normal crossing singularities due to the description of a k2A flip in \cite{Mori02}. We claim that $\Gamma^+$ is not contained in $W^+$. Otherwise:

\begin{itemize}
    \item If $\Gamma^+ \subset W^+ \cap Z^+$, then $\Delta^+=\Delta+\Gamma^+$ and $\Gamma^+ \cdot K_{\W^+}=\Gamma^+ \cdot (K_{W^+}+\Delta^+)=-\frac{1}{m'_1}-\frac{1}{m'_2}+\Gamma^+ \cdot \Delta$, where $\Delta$ is the proper transform of $\Delta \subset \W'$, and $m'_i$ are the indices of the new terminal singularities of $\W^+$. But $\Gamma^+ \cdot \Delta=\frac{1}{m'_i}$ for some $i$, and so $\Gamma^+ \cdot K_{\W^+}<0$, a contradiction. 
    
    \item If $\Gamma^+ \subset W^+$ but not in $Z^+$, then the c.q.s. in $Z^+$ and $Z$ over $Q$ are equal. As in $W^+$ we must have its dual c.q.s., then $W^+$ and $W'$ share the same c.q.s. As the contraction of each $\Gamma^+ \subset W^+$ and $\Gamma^- \subset W^{-}$ give the same c.q.s., then the other singularity (which must be a Wahl singularity) on $\Gamma^+$ must be equal to the singularity on $\Gamma^-$. But this is not possible, because after a k2A flip at least one of the indices must drop. 
\end{itemize}

Therefore, the flipped curve $\Gamma^+$ is in $Z^+$ and not in $W^+$. To compute it explicitly, we contract $\Gamma^- \subset W'$ obtaining a c.q.s. $\frac{1}{A}(1,B)$. Then in $Z^+$ we must have a Wahl singularity followed by $\Gamma^+$ and then $\frac{1}{A}(1,A-B)$ which contracts to the original singularity in $Z$. 

\begin{example}
Let us do the computation from the example in Figure \ref{fig:Construction_f} in terms of continued fractions:
\begin{itemize}
\item Over the minimal resolution of $W'$ we have $[\underbrace{2,\ldots,2}_q,q+6]-(1)-[\underbrace{2,\ldots,2}_{q+4},q+8]$. The $(1)$ represents the proper transform of $\Gamma^{-}$. The chain contracts to $[6]$, which contracts to $\frac{1}{6}(1,1)$ in $W^+$.

\item Over the minimal resolution of $Z^+$ we have $[\underbrace{2,\ldots,2}_{q+4},q+8]-(1)-[2,2,2,2,2]$. The $(1)$ represents the proper transform of $\Gamma^{+}$. The chain contracts to $[\underbrace{2,\ldots,2}_{q+4},q+2]$, whose contraction is the singularity in $Z$, dual to the one in $W'$.
\end{itemize}  
\label{exampleStep(8)}
\end{example}

\textbf{(9)} Let us identify the surface $Z^+$. As the leg of the \QHD singularity together with the sliding Wahl chain is contracted to $\frac{1}{2}(1,1)$, $\frac{1}{3}(1,1)$ or $\frac{1}{6}(1,1)$ (valency 3), or to a nonsingular point in the central curve (valency 4), then on $Z^+$ we have singularities $\frac{1}{2}(1,1)$, $\frac{1}{3}(1,2)$ or $\frac{1}{6}(1,5)$ over $\Delta$. In the minimal resolution $\hat{Z}^+$ of $Z^+$, $\Delta$ is a $(-2)$-curve, and we have an elliptic fibration $\hat{Z}^+ \to \P^1$ with a fiber of type $II^*$. Its central $(-2)$-curve is $\Delta$. On the other hand, the new Wahl singularity in $Z^+$ is resolved to a type $yI_1(n,a)$ configuration as in Figure \ref{fig:KAWA_WAHL}. One can check that, depending on the leg used in the construction of the original \QHD singularity, the multiplicity $y$ is $1$, $2$, $3$, $4$, or $6$. These are the multiplicities of the central component or extremal components for the elliptic fibers $II^*$, $III^*$, and $IV^*$. The number $ny$ beautifully coincides with the corresponding denominators of the $\lambda$s after computing the Wahl singularity of the flip. See Figure \ref{flipfigure}. 

\begin{theorem}
For each given \QHD degeneration from Theorem \ref{DolgachevSmoothings} we have:

\begin{itemize}
    \item[(1)] If the singularities of $W$ are all log canonical, then the \QHD degeneration is a minimal slc degeneration.

    \item[(2)] If there is one non-log canonical \QHD singularity, then the computed flip $\W^+ \to \D$ is a minimal slc degeneration.
\end{itemize}
        \label{NefSLCLimit}
\end{theorem}

\begin{proof}
Part (1) is obvious by construction. For part (2) we need to check that $K_{W^+} + \Delta$ and $K_{Z^+}+\Delta$ are nef. If that is true, then $K_{\W^+}$ is nef as it is for every fiber.

The calculation is very similar to the one in \cite[Thm. 4.2]{U16a}. Let us start with $Z^+$. Let $\phi \colon Y \to Z^+$ be its minimal resolution. We have an elliptic fibration on $Y$, one fiber contains the Wahl exceptional chain and another fiber is $II^*$ with central curve equal $\Delta$. Let $\pi \colon Y \to S$ be the contraction of all $(-1)$-curves to get a relative minimal elliptic fibration $S \to \P^1$. It has a multiple fiber $yI_1=yF_1$ (over which we have the Wahl chain) and $II^*=F_2$. Let $\pi^*(F_1)=\sum_{i=1}^{s+1} \alpha_i G_i$ be its prime decomposition, where $G_{s+1}$ is the $(-1)$-curve, $G_1$ is the proper transform of $F_1$, and $s$ is the length of the Wahl chain. Then $$K_Y \sim -F + (y-1)F_1+ \sum_{i=1}^{s+1} (\alpha_i-1)G_i.$$ We also have that the discrepancies of the curves $G_i$ in the Wahl chain are $d(G_i)=-1+\frac{\alpha_i}{n}$ by \cite[Lem. 4.1]{U16a}, where $n$ is the index of the Wahl singularity. Then $$-F + (y-1)F_1 + \sum_{i=1}^{s+1} (\alpha_i-1)G_i \equiv \phi^*(K_{Z^+})- \sum_{i=1}^{s} \Big(1-\frac{\alpha_i}{n} \Big) G_i,$$ where $F$ is the general fiber, and so $F \sim \sum_{i=1}^{s+1} y\alpha_i G_i$. On the other hand $\phi^*(\Delta) \equiv \frac{1}{6} F$ and we obtain $$\phi^*(K_{Z^+}+\Delta) \equiv \frac{yn-6}{6yn}F .$$ But $yn$ is exactly the numerator of the corresponding $\lambda$ for the original \QHD singularity. By inspection on the table in Figure \ref{lambda}, we have that $ny > 6$ for type II non-log canonical. Thus, $K_{Z^+}+\Delta$ is nef. 

If now $\phi \colon Y \to W^+$ is the minimal resolution of $W^+$, a similar calculation shows that for $W^+$ we have $$ \phi^*(K_{W^+}+\Delta) \equiv \frac{5n-6}{6n}F,$$ and so it is nef. Note that $y=1$ in this case, because we started with a fibration with sections $g_S \colon S \to \P^1$. 
\end{proof}

\begin{remark}
In the calculation above, we assumed that the \QHD singularity was of type $II$. We now summarize it for all types $II$, $III$ or $IV$, where $\phi$ is the corresponding minimal resolution:
\begin{itemize}
    \item[(II)]$\phi^*(K_{W^+}+\Delta) \equiv \frac{5n-6}{6n}F$ and $\phi^*(K_{Z^+}+\Delta) \equiv \frac{yn-6}{6yn} F$,
    \item[(III)]$\phi^*(K_{W^+}+\Delta) \equiv \frac{3n-4}{4n}F$ and $\phi^*(K_{Z^+}+\Delta) \equiv \frac{yn-4}{4yn} F$,
    \item[(IV)]$\phi^*(K_{W^+}+\Delta) \equiv \frac{2n-3}{3n}F$ and $\phi^*(K_{Z^+}+\Delta) \equiv \frac{yn-3}{3yn} F$,
\end{itemize}
and so they are all nef for the case non-log canonical \QHD singularity. But for types III and IV we need 1 and 2 extra flips.
\label{K+Delta}
\end{remark}

%--- Are any of these nef slc limits equal for different \QHD singularities? Make a remark on this. 

%%%%%%%%%%%%%%%%%%%%%%%%%%%%%%%%%%%%%%%%%%%%%%%%%%%%%%%%%%%%%%%%%%%%%%%%%%%%%%%%%%%%%%%%%%%%%%%%%%%%%%%%%%%%%%%%%%%%%%%%%%%%%%%%%%%%%%%%%%%%%%%%%%%%%%%%%%
\section{Minimal slc limits are unobstructed and deform to Lee--Lee degenerations} \label{s7}

In this section, we will show that the singular fiber of the minimal slc models in Theorem \ref{NefSLCLimit} is unobstructed. As a consequence, we prove that the examples in \cite[Thm. 5.2]{LL25} degenerate into our minimal slc models.
\bigskip 

For the general set-up, we need to work with a rational elliptic fibration $g_S \colon S \to \P^1$ that has a multiple fiber $F_1=y I_1$ and a fiber $F_2$ of type $II$, $II^*$, $III$, $III^*$, $IV$, or $IV^*$. We denote the latter by $F_2$ or $F_2^*$ accordingly. We consider the composition of blow-ups $S' \to S$ to obtain the minimal simple normal crossing fiber $F'_2$ of $F_2$. This is the situation in Figure \ref{s.n.c config}. We denote the central curve of $F'_2$ by $\Delta$, and the rest $A,B,C$ with self-intersections $-a,-b,-c$ respectively, for the triples $-2,-3,-6$ (II), $-4,-2,-4$ (III), and $-3,-3,-3$ (IV). In the case of $F^*_2$ we do not blow-up, and we denote the central curve by $\Delta$. Let $S' \to Y_0$ be the contraction of the curves over $F_2$ except for $\Delta$.

\begin{lemma}
We have $-(K_{Y_0}+\Delta)= \big(\frac{1}{y}-m \big) F$ where $F$ is a fiber of the induced fibration, and $m=\frac{1}{6}$ for $II^*$, $m=\frac{1}{4}$ for $III^*$, $m=\frac{1}{3}$ for $IV^*$, $m=\frac{5}{6}$ for $II$, $m=\frac{3}{4}$ for $III$, and $m=\frac{2}{3}$ for $IV$. 
\label{calculito}
\end{lemma}

\begin{proof}
This is a straightforward computation, simpler than the ones in Remark \ref{K+Delta}.
\end{proof}

\begin{corollary}
We have that $-(K_{Y_0}+\Delta)$ is effective when $y=1$ for $II$, $III$, $IV$; when $y\leq 6$ for $II^*$; when $y \leq 4$ for $III^*$; and when $y \leq 3$ for $IV^*$.
    \label{effective}
\end{corollary}

Let us consider $Y_0$ as in Corollary \ref{effective}. We have elliptic fibration $g_{Y_0} \colon Y_0 \to \P^1$. Let us blow up the node of $F_1$ through the morphism $\pi_1 \colon Y'_0 \to Y_0$. Let $\pi_1^*(F_1)=\Gamma_1+ 2G_1$ where $\Gamma_1^2=-4$ and $G_1^2=-1$. Let $T_{Y'_0}(-\log (\Gamma_1))$ be the tangent sheaf of $Y'_0$ with simple poles along $\Gamma_1$.

\begin{lemma}
We have that $H^2(Y'_0,T_{Y'_0}(-\log (\Gamma_1))(-\Delta))=0$.
\label{computeObstr0}
\end{lemma}

\begin{proof}
This follows the strategy in \cite[Lem. 9.4]{Hack04}. By Serre duality, $$H^2(Y'_0,T_{Y'_0}(-\log (\Gamma_1))(-\Delta))= \text{Hom}(T_{Y'_0}(-\log (\Gamma_1)),(K_{Y'_0}+\Delta))^{\vee}.$$ As $-(K_{Y'_0}+\Delta+G_1)$ is effective ($Y_0$ as in Corollary \ref{effective}), we have $$\text{Hom}(T_{Y'_0}(-\log (\Gamma_1)),(K_{Y'_0}+\Delta)) \hookrightarrow \text{Hom}(T_{Y'_0}(-\log (\Gamma_1))(-G_1), \O_{Y_0}),$$ and $\text{Hom}(T_{Y'_0}(-\log (\Gamma_1))(-G_1), \O_{Y'_0})=H^0(Y'_0, {\Omega_{Y'_0}^1}(\log(\Gamma_1))^{\vee \vee}(G_1))$. We now follow the strategy in the proof of \cite[Thm. 2.1]{PSU13}. We have the short exact sequence $$ 0\to {\Omega_{Y'_0}^1}(\log(\Gamma_1+G_1))^{\vee \vee} \to {\Omega_{Y'_0}^1}(\log(\Gamma_1))^{\vee \vee}(G_1) \to \Omega_{G_1}^1(\Gamma_1+G_1)|_{G_1} \simeq \O_{\P^1}(-1) \to 0, $$ and so  $H^0({\Omega_{Y'_0}^1}(\log(\Gamma_1+G_1))^{\vee \vee})=H^0({\Omega_{Y'_0}^1}(\log(\Gamma_1))^{\vee \vee}(G_1))$. We now take the residue sequence $$ 0 \to {\Omega_{Y'_0}^1}^{\vee \vee} \to {\Omega_{Y'_0}^1}(\log (\Gamma_1 + G_1))^{\vee \vee} \to \O_{\Gamma_1} \oplus \O_{G_1} \to 0.$$ We have $h^0({\Omega_{Y_0}^1}^{\vee \vee})=h^0(\Omega_{S}^1)=0$ (see \cite[Lem. 9.4]{Hack04}). On the other hand, the map $h^0(\O_{\Gamma_1}) \oplus h^0(\O_{G_1}) \to H^1({\Omega_{Y'_0}^1}^{\vee \vee}) \hookrightarrow H^1(\Omega_{S''}^1)$ is the Chern class map, where $S'' \to S'$ is the blow-up of $F_1$ at the node. But $\Gamma_1$ and $G_1$ are linearly independent, and so $0=H^0({\Omega_{Y'_0}^1}(\log(\Gamma_1+G_1))^{\vee \vee})=H^0({\Omega_{Y'_0}^1}(\log(\Gamma_1))^{\vee \vee}(G_1))$.
\end{proof}

\begin{proposition}
Let $Y \to Y'_0$ be the blow-ups over a node of $\Gamma_1+G_1$ so that we form a Wahl chain $D$ plus an extra $(-1)$-curve. Then $H^2(Y,T_{Y}(-\log (D))(-\Delta))=0$.  
\label{erase/add}
\end{proposition}

\begin{proof}
This is true because of the same principles of adding and erasing $(-1)$-curves to a SNC divisor and keeping the relevant second cohomology. This is explained in \cite[Section 4]{PSU13}. Note that all birational operations happen away from the singularities of the surfaces involved.
\end{proof}

\begin{proposition}
Let $Y \to U$ be the contraction of the Wahl chain $D$. Then $$H^2(Y,T_{Y}(-\log (D))(-\Delta))=H^2(U,T_U(-\Delta)).$$
\label{LP07}
\end{proposition}
    
\begin{proof}
This follows from the same proofs of \cite[Lem. 1 and Theorem 2]{LP_2007}.
\end{proof}

\begin{theorem}
Let us consider the surface $U_1 \cup U_2$, where $U_i$s are log terminal surfaces, such that:   
\begin{itemize}
\item[(1)] $\P^1=U_1 \cap U_2=\Delta$, and along it we have orbifold normal crossing points only.
\item[(2)] If $\Delta_i:=\Delta|_{U_i}$, then $\Delta_1^2+\Delta_2^2=0 \ \text{or} \ 1$.
\item[(3)] $H^2(U_i,T_{U_i}(-\Delta_i))=0$ for $i=1,2$.
\end{itemize}
Then $U_1 \cup U_2$ has unobstructed $\Q$-Gorenstein deformations.
\label{LPnoObstructions}
\end{theorem}

\begin{proof}
This is the same proof of \cite[Thm. 9.1]{Hack01} together with \cite[Lem. 9.4]{Hack04}. Let $U:=U_1 \cup U_2$. Consider the sheaves $\T_{QG,U}^i$ as in \cite[Section 3]{Hack04} for $i=0,1,2$ that control $\Q$-Gorenstein deformations of $U$. We need to show that $H^j(U,\T_{QG,U}^i)=0$ for $i+j=2$. We have $H^0(U,\T_{QG,U}^2)=0$ because the local canonical cover along $\Delta$ is a hypersurface. We also have $\T_{QG,U}^1 \simeq \O_{\Delta}(\Delta_1^2+\Delta_2^2)$ along $\Delta=\P^1$, and so under (2) we obtain $H^1(U,\T_{QG,U}^1)=0$. Thus, we only need to check $H^2(U,\T_{QG,U}^0)=0$, but this follows \cite[Lem. 9.4]{Hack04} directly from (3). 
\end{proof}

In \cite{LL25}, D. Lee and Y. Lee give the following examples of degenerations of Dolgachev surfaces $D_{a,b}$ with gcd$(a,b)=1$. Essentially, they consider the situation of two Jacobian rational elliptic fibrations $\hat{U}_1/\hat{U}_2$ with at least one singular fiber pair (1) $I_0^*/I_0^*$, (2) $II/II^*$, (3) $III/III^*$, and (4) $IV/IV^*$. For the star fibers they contract the $(-2)$-curves except for the central curve $\Delta$. For the fibers $II$, $III$, and $IV$ they consider the minimal SNC resolution with curves $\Delta$ (central), $A,B,C$ as before, and contract $A,B,C$. In both cases we obtain singular surfaces $U_1,U_2$. Then they glue them along $\Delta$, obtaining a surface $U=U_1 \cup U_2$ with two components and orbifold normal crossing singularities. In \cite[Thm. 5.2]{LL25}, they prove that $U$ has unobstructed $\Q$-Gorenstein smoothings into Jacobian rational elliptic fibrations. Then they perform logarithmic transformations over two fibers of the smoothing, each landing on each one $U_i$. The result $U_{a,b}$ is a $\Q$-Gorenstein degeneration of surfaces $D_{a,b}$ \cite[Cor. 5.3]{LL25}. We call them \textit{Lee--Lee examples} of type (1), (2), (3), (4). We do not care if gcd$(a,b)$ is $1$ or not $1$.      

\begin{corollary}
The nef slc limit in Theorem \ref{NefSLCLimit} is unobstructed, and they are degenerations of the Lee--Lee examples of type (2), (3), (4).
\label{LeeLee}
\end{corollary}

\begin{proof}
Our minimal slc surface $U=U_1 \cup U_2$ will satisfy the hypothesis of Theorem \ref{LPnoObstructions}. We construct each component $U_i$ from a surface $Y_0$ as in Corollary \ref{effective}. Then by Lemma \ref{computeObstr0}, and Propositions \ref{erase/add} and \ref{LP07}, we have hypothesis (3). Hypothesis (1) is for free by construction. For hypothesis (2), we compute $$\Delta_1^2+\Delta_2^2=-c_1-c_2+\#,$$ $\#$ is the number of singular points on $\Delta$, and $-c_i$ is the self-intersection of the proper transform of $\Delta_i$ in the minimal resolutions of $U_i$. It is an easy check that $c_1=2$, $c_2=1$, and $\#=3$ (valency $3$) or $4$ (valency $4$). Thus Theorem \ref{LPnoObstructions} holds for $U$.

We now choose to $\Q$-Gorenstein smoothing of the Wahl singularity in $U_1$ and in $U_2$, keeping the double curve $\Delta$. Then we obtain $U_t=U'_1 \cup U'_2$ where $U_i$ are Halphen pencils of index equal to $a$ (which is the denominator of $\lambda$ of the starting \QHD singularity) and $b$. This is then a Lee--Lee example.
\end{proof}

\begin{remark}
A direct computation showing that the obstruction vanishes for the Lee--Lee examples does not work within our strategy. This is because, in the case of a star singular fiber, Lemma \ref{effective} does not guarantee that $-(K+\Delta)$ is effective for an arbitrary multiplicity $y$.
\label{notGuarantee}
\end{remark}

\begin{remark}
Lee--Lee examples of type (1) do not come from \QHD singularities. But one could consider $\Q$-Gorenstein smoothable singularities $[n;[2]^4]$ with $n \leq 6$ as in \cite[Table 2, $1^{\bullet}$]{Prokho19}. We can construct its star graph from a $I_0^*$ by blowing-up general points on the central curve. We do sliding and get Wahl singularities $\frac{1}{4}(1,1)$. We now consider a Seifert partial resolution of the $\Q$-Gorenstein smoothing. The proper transforms of the $(-1)$-curves in $W'$, the proper transform of $W$, would give flopping curves. Then the corresponding Atiyah flops would give a minimal slc model with no obstructions, and the smoothing of Wahl singularities would give Lee--Lee examples of type (1). However, the corresponding Dolgachev surfaces would be of type $D_{2,b}$ only. Would it be possible to improve this hypothetical example to get any $D_{a,b}$?
\label{case (1) LeeLee}
\end{remark}

%%%%%%%%%%%%%%%%%%%%%%%%%%%%%%%%%%%%%%%%%%%%%%%%%%%%%
\bibliographystyle{abbrv}

\bibliography{Bibliography}

%---------------------------------------
\appsection{Explicit computation of discrepancies} \label{app}

%\subsection{Discrepancies of W\QHD Singularities}
In the following singularities, we enumerate the legs counterclockwise. We computed the invariants associated to every singularity. Denote by $r_i$ the total number of curves in the $i$-th leg. Also, denote by $d_{i,k}$ the discrepancy of the $k$-th curve in the $i$-th leg. Computations can be done here \href{https://colab.research.google.com/drive/1BgJJCC0qD8eTAmpa3kdrB7CG2GsvNc6d?usp=sharing#scrollTo=YHzJ42TQLYWI}{QHD discrepancies} \cite{programa}.

\bigskip 
\begin{center}
\textbf{Valency 3 singularities}    
\end{center}

\bigskip 
% --- Mini Columna 1---
\begin{minipage}[t]{0.45\textwidth}
\textbf{Family (a)}\\[0.1cm]
{\small
Discrepancies of ending curves:
\begin{align*}
    \chi/e &= \frac{(p+1)(q+1)(r+1)-1}{(p+2)(q+2)(r+2)+1} \\
    1+d_{1,r_1} &= \frac{(q+1)(r+2)+1}{(p+2)(q+2)(r+2)+1} \\[0.2cm]
    1+d_{2,r_2} &= \frac{(p+1)(q+2)+1}{(p+2)(q+2)(r+2)+1} \\[0.2cm]
    1+d_{3,r_3} &= \frac{(r+1)(p+2)+1}{(p+2)(q+2)(r+2)+1}
\end{align*}
Discrepancies of legs:\\
For $ 1\le k\le q+1$:
\[ 1+d_{1,k}= k(1+d_{3,r_3})-\frac{\chi}{e} \]
For $1\le k\le p+1$:
\[ 1+d_{2,k}= k(1+d_{1,r_1})-\frac{\chi}{e}\]
For $1\le k\le r+1$:
\[ 1+d_{3,k}= k(1+d_{2,r_21})-\frac{\chi}{e}\]
Special cases:\\
If $p=q=r$, then:
\[
1+d_{i,k}=\frac{k-p}{p+3}
\]
for $1\le k\le p+1$ and $i=1,2,3$.\\
}
\end{minipage}
\hfill
% ======= COLUMN 2 =======
\begin{minipage}[t]{0.45\textwidth}
\textbf{Family (b)}\\[0.1cm]
{\small
Discrepancies of ending curves:
\begin{align*}
\chi/e &= \frac{(p+2)(q+2)(r+1)-(r+2)}
               {(p+2)(q+3)(r+2)+(q+2)}\\
1+d_{1,r_1} &= \frac{(p+2)(q+3)-1}
               {(p+2)(q+3)(r+2)+(q+2)}\\[0.2cm]
1+d_{2,r_2} &= \frac{(p+2)(r+1)+1}
               {(p+2)(q+3)(r+2)+(q+2)}\\[0.2cm]
1+d_{3,r_3} &= \frac{4+q+r}{(p+2)(q+3)(r+2)+(q+2)}
\end{align*}
Discrepancies of legs:\\
For $1\le k\le q+1$:
\[
1+d_{1,k}= k(1+d_{2,r_2})-\frac{\chi}{e}
\]
For $q+2\le k\le p+q+2$:
\[
1+d_{1,k}= (k-(q+1))(1+d_{3,r_3})+(1+d_{1,q+1})
\]
For $1\le k\le r+1$:
\[
1+d_{2,k}= k(1+d_{1,r_1})-\frac{\chi}{e}
\]
Special cases:\\
If $p=0,\ r=q+1$, then:
\[
1+d_{1,k}=1+d_{2,k}
   =\frac{k+3}{q+4}-1
\]
for $1\le k\le q+2$.\\
}
\end{minipage}
\newpage
%XXXXXXXX. NEW PAGE XXXXXXXXXXXXXXXX
% --- Mini Columna 1---
\begin{minipage}[t]{0.45\textwidth}
\textbf{Family (c)}\\[0.1cm]
{\small
Discrepancies of ending curves:
\begin{align*}
    \chi/e &= \frac{(q+1)(r+1)-1}{(q+3)(r+3)-1} \\
    1+d_{1, r_1}&=\frac{q+2}{(q+3)(r+3)-1}\\[0.2cm]
    1+d_{2, r_2}&=\frac{q+r+4}{(q+3)(r+3)-1}\\[0.2cm]
    1+d_{3, r_3}& =\frac{r+2}{(q+3)(r+3)-1}
\end{align*}
Discrepancies of legs:\\
For $1\le k\le q+1$:
\[
    1+d_{1,k}=k\cdot(1+ d_{3,r_3})-\frac{\chi}{e}
\]
For $1\le k\le r+1$:
   \[ 1+d_{3,k}=k\cdot(1+ d_{1,r_1})-\frac{\chi}{e}\]
Special cases:\\
If $q=r$, then:
\[1+d_{1,k}=1+d_{3,k}=\frac{k-p}{p+4}\]
for $1\le k\le p+1$ and $i=1,3$.\\
}
\end{minipage}
\hfill
% ======= COLUMN 2 =======
\begin{minipage}[t]{0.45\textwidth}
\textbf{Family (d)}\\[0.1cm]
{\small
Discrepancies of ending curves:
\begin{align*}
\chi/e &= \frac{(q+2)(r+2)-2}{(q+4)(r+2)+2}\\
1+d_{1,r_1}& =\frac{r+2}{(q+4)(r+2)+2}\\[0.2cm]
    1+d_{2,r_2}& =\frac{2}{(q+4)(r+2)+2}\\[0.2cm]
    1+d_{3,r_3}& =\frac{r+4}{(q+4)(r+2)+2}
\end{align*}
Discrepancies of the first leg\\
For $1\leq k\leq q+1$:
\[
1+d_{1,k}=k\cdot(1+ d_{1,r_1})-\frac{\chi}{e}
\]
For $q+2\leq k\leq q+r+2$:
\[
1+d_{1,k}= (k-(q+1))\cdot (1+d_{2,r_2})+(1+d_{1,q+1})
\]
Special cases:\\
If $r=0$:\\
For $1\leq k\leq q+2$:
\[
1+d_{1,k}=-1+\frac{4+k}{5+q}
\]
If $r=2$:\\
For $1\leq k\leq q+1$:
\[
1+d_{1,k}=-1+\frac{2 (k+3)}{2 q+9}
\]
For $q+2\leq k\leq q+4$:
\[
1+d_{1,k}=\frac{k-(q+2)}{2 q+9}
\]
}
\end{minipage}
\newpage
%XXXXXXXX. NEW PAGE XXXXXXXXXXXXXXXX
% --- Mini Columna 1---
\begin{minipage}[t]{0.45\textwidth}
\textbf{Family (e)}\\[0.1cm]
{\small
Discrepancies of ending curves:
\begin{align*}
    \chi/e &= 1-\frac{3(p+3)}{2(p+3)(q+3)-3(q+2)}\\
    1+d_{1,r_1} &=\frac{2p+3}{2(p+3)(q+3)-3(q+2)}\\[0.2cm]
    1+d_{2,r_2} &=\frac{p+3}{2(p+3)(q+3)-3(q+2)}\\[0.2cm]
    1+d_{3,r_3} &=\frac{3}{2(p+3)(q+3)-3(q+2)}
\end{align*}
Discrepancies of the first leg:\\
For $1\leq k\leq q+1:$
\[
1+d_{1,k}=k\cdot(1+ d_{1,r_1})-\frac{\chi}{e}
\]
For $q+2\leq k\leq p+q+2$:
\[
1+d_{1,k}= k\cdot (1+d_{3,r_3})+(1+d_{1,q+1})
\]
Special cases:\\
If $p=0$:\\
For $1\leq k\leq q+2$:
\[
1+d_{1,k}=-1+\frac{k+3}{q+4}
\]
If $p=3$:\\
For $1\leq k\leq q+1$:
\[
1+d_{1,k}=-1+\frac{3 (k+2)}{3 q+10}
\]
For $q+2\leq k\leq q+5$:
\[
1+d_{1,k}=\frac{k-(q+2)}{3 q+10}
\]
}\\
\textbf{Family (f)}\\[0.1cm]
{\small
Discrepancies of ending curves:
\begin{align*}
\chi/e &= \frac{q}{q+6}\\
    1+d_{1,1} &=\frac{3}{q+6}\\[0.2cm]
    1+d_{2,2} &=\frac{2}{q+6}\\[0.2cm]
    1+d_{3,q+1} &=\frac{1}{q+6} 
\end{align*}
Discrepancies of the thrid leg\\
For $1\leq k\leq q+1$:
\[1+d_{3,k}=\frac{k-q}{q+6}\]
}
\end{minipage}
\hfill
% ======= COLUMN 2 =======
\begin{minipage}[t]{0.45\textwidth}
\textbf{Family (g)}\\[0.1cm]
{\small
Discrepancies of ending curves:
\begin{align*}
    \chi/e &=\frac{(p+2)(q+2)(r+3)-(p+q+5)}{(p+2)(q+3)(r+3)-q+r+1}\\
    1+d_{1,r_1} &=\frac{(p+2)(r+3)-1}{(p+2)(q+3)(r+3)-q+r+1}\\[0.2cm]
    1+d_{2,r_2} &=\frac{r+4}{(p+2)(q+3)(r+3)-q+r+1}\\[0.2cm]
    1+d_{3,r_3} &=\frac{p+3}{(p+2)(q+3)(r+3)-q+r+1}
\end{align*}
Discrepancies of the first leg:\\
For $1\leq k\leq q+1:$

\[1+d_{1,k}=k\cdot (1+d_{1,r_1})-\frac{\chi}{e}\]

For $q+2\leq k\leq q+r+2:$

\[1+d_{1,k}=k\cdot (1+d_{3,r_3})+(1+d_{1,q+1})\]

For $q+r+3\leq k\leq p+q+r+3:$

\[1+d_{1,k}=k\cdot (1+d_{2,r_2})+(1+d_{1,q+r+2})\]

Special cases:\\
If $p=0,\ r=2$, then:\\
For $1\leq k\leq q+1:$
\[
1+d_{1,k}=\frac{3 k-3 q-5}{3q+11}
\]
For $q+2\leq k\leq q+4:$
\[
1+d_{1,k}=\frac{k-(q+3)}{3q+11}
\]
For $k=q+5:$
\[
1+d_{1,q+5}=\frac{3}{3q+11}
\]

If $p=1,\ r=0$, then:\\
For $1\leq k\leq q+1:$
\[
1+d_{1,k}=-1+\frac{2 (k-q)-3}{2 q+7}
\]
For $q+2\leq k\leq q+4:$
\[
1+d_{1,k}=\frac{k-(q+2)}{2 q+7}
\]

}
\end{minipage}

\newpage
%XXXXXXXX. NEW PAGE XXXXXXXXXXXXXXXX
% --- Mini Columna 1---
\begin{minipage}[t]{0.45\textwidth}
\textbf{Family (h)}\\[0.1cm]
{\small
Discrepancies of ending curves:
\begin{align*}
\chi/e &= \frac{q+1}{q+3}\\
1+d_{1,r_1}& =\frac{2}{2(q+3)}\\
1+d_{2,r_2}& =\frac{1}{2(q+3)}\\
1+d_{3,r_3}& =\frac{1}{2(q+3)}
\end{align*}
Discrepancies of the first leg\\
For $1\leq k\leq q+2$:
$$1+d_{1,k}=\frac{k-(q+1)}{q+3}$$
\textbf{Family (i)}\\[0.1cm]
Discrepancies of ending curves:
\begin{align*}
\chi/e &=\frac{q+1}{q+3}\\
1+d_{1,r_1} &=\frac{3}{3(q+3)}\\
1+d_{2,r_2} &=\frac{2}{3(q+3)}\\
1+d_{3,r_3} &=\frac{1}{3(q+3)}
\end{align*}
Discrepancies of the first leg\\
For $1\leq k\leq q+2$:
$$1+d_{1,k}=\frac{k-(q+1)}{q+3}$$
\textbf{Family (j)}\\[0.1cm]
Discrepancies of ending curves:
\begin{align*}
\chi/e &=\frac{q+1}{q+4}\\
1+d_{1,r_1}& =\frac{3}{2(q+4)}\\
1+d_{2,r_2}& =\frac{2}{2(q+4)}\\
1+d_{3,r_3}& =\frac{1}{2(q+4)} 
\end{align*}
Discrepancies of the second leg\\
For $1\leq k\leq q+2$:
$$1+d_{2,k}=\frac{k-(q+1)}{q+4}$$
}

\end{minipage}
\hfill
% ======= COLUMN 2 =======
\begin{minipage}[t]{0.45\textwidth}
\textbf{Valency 4 singularities}\\[0.5cm]
{\small
Discrepancies of ending curves:
\begin{align*}
\chi/e &= \frac{p+1}{p+2}\\
\end{align*}
Discrepancies of the fourth leg\\
For $1\leq k\leq p+2$:
$$1+d_{4,k}=\frac{k-(p+1)}{p+2}$$
\textbf{Type (a)}\\[0.1cm]
Discrepancies of ending curves:
\begin{align*}
    1+d_{1,r_1}= 1+d_{2,r_2}=1+d_{3,r_3}&=\frac{1}{3(p+2)}
\end{align*}
\textbf{Type (b)}\\[0.1cm]
Discrepancies of ending curves:
\begin{align*}
    1+d_{1,r_1} &=\frac{1}{2(p+2)}\\
    1+d_{2,r_2}=1+d_{3,r_3}&=\frac{1}{4(p+2)}
\end{align*}
\textbf{Type (c)}\\[0.1cm]
Discrepancies of ending curves:
\begin{align*}
    1+d_{1,r_1}&=\frac{1}{2(p+2)}\\
    1+d_{2,r_2}&=\frac{1}{3(p+2)}\\
    1+d_{3,r_3}&=\frac{1}{6(p+2)}
\end{align*}
}
\end{minipage}

%%%%%%%%%%%%%%%%%%%%%%%%%%%%%%%%%%%%%%%%%%%%%%%%%%%%%%%%%%%%%%%%%%%%%%%%%%%%%%%%%%%%%%%%%%%%%%%%%%%%%%%%%%%%%

\end{document}